\documentclass[a4paper,twoside]{article}

\usepackage{amssymb}
\usepackage{amsmath}
\usepackage{amsthm}

\usepackage{mathrsfs}

\usepackage{natbib}

\usepackage{esint}                

\newtheoremstyle{citing}
  {3pt}
  {3pt}
  {\itshape}
  {}
  {\bfseries}
  {.}
  {.5em}
  {\thmnote{#3}}
\theoremstyle{citing}

\theoremstyle{plain}
\newtheorem{intro_theorem}{Theorem}
\newtheorem*{intro_theorem2}{Theorem}

\swapnumbers
\theoremstyle{plain}
\newtheorem{theorem}{Theorem}[section]
\newtheorem{lemma}[theorem]{Lemma}
\newtheorem{corollary}[theorem]{Corollary}

\theoremstyle{remark}
\newtheorem{remark}[theorem]{Remark}

\theoremstyle{definition}

\newtheorem{miniremark}[theorem]{}

\newcounter{counter1}
\newcounter{counter2}

\newcommand{\Var}{\mathbf{V}}     
\newcommand{\IVar}{\mathbf{IV}}   

\newcommand{\Lp}[1]{\mathbf{L}_{#1}}

\newcommand{\Sob}[3]{\mathbf{W}_{#1}^{{#2},{#3}}}

\newcommand{\ccspace}[1]{\mathscr{K}(#1)}

\newcommand{\qspace}{\mathbf{Q}}

\newcommand{\nat}{\mathscr{P}}

\newcommand{\integers}{\integers}

\newcommand{\rel}{\mathbf{R}}

\newcommand{\grass}[2]{\mathbf{G}(#1,#2)}

\newcommand{\orthproj}[2]{\mathbf{O}^\ast({#1},{#2})}

\newcommand{\pp}{\mathbf{p}}
\newcommand{\qq}{\mathbf{q}}

\newcommand{\oball}[2]{\mathbf{U}(#1,#2)}

\newcommand{\cball}[2]{\mathbf{B}(#1,#2)}

\newcommand{\density}{\boldsymbol{\Theta}}

\newcommand{\unitmeasure}[1]{\boldsymbol{\alpha}(#1)}

\newcommand{\isoperimetric}[1]{\boldsymbol{\gamma}(#1)}

\newcommand{\cylind}[3]{\mathbf{C} ( #1, #2, #3 )}

\newcommand{\cylinder}[4]{\mathbf{C} ( #1, #2, #3, #4 )}

\newcommand{\weakD}{\mathbf{D}}

\newcommand{\ud}{\ensuremath{\,\mathrm{d}}}

\DeclareMathOperator{\with}{:}
\newcommand{\classification}[3]{{#1} \cap \{ {#2} \with {#3} \}}
\newcommand{\eqclassification}[3]{{(#1)} \cap \{ {#2} \with {#3} \}}
\newcommand{\bigclassification}[3]{{#1} \cap \big \{ {#2} \with {#3} \big \}}

\newcommand{\project}[1]{#1_\natural}

\newcommand{\pluslim}[1]{\to {#1}+}

\newcommand{\lIm}{[}
\newcommand{\rIm}{]}
\newcommand{\biglIm}{\big [}
\newcommand{\bigrIm}{\big ]}

\newcommand{\vdim}{{m}}
\newcommand{\codim}{{n-m}}
\newcommand{\adim}{{n}}

\newcommand{\norm}[3]{\boldsymbol{|} #1 \boldsymbol{|}_{#2;#3}}

\newcommand{\idnorm}[4]{\boldsymbol{|} #1 \boldsymbol{|}_{#2,#3;#4}}
\newcommand{\dnorm}[3]{\idnorm{{#1}}{-1}{{#2}}{{#3}}}

\newcommand{\printRoman}[1]{\setcounter{counter1}{#1}\Roman{counter1}}

\newcommand{\leftB}{{\left < \rule{0pt}{2ex} \right .}}
\newcommand{\rightB}{{\left . \rule{0pt}{2ex} \right >}}

\DeclareMathOperator{\without}{\sim}

\newcommand{\restrict}{\mathop{\llcorner}}

\DeclareMathOperator{\aplus}{(+)}

\newcommand{\class}[1]{#1}

\newcommand{\tint}[2]{{\textstyle\int_{#1}^{#2}}}
\newcommand{\tfint}[2]{{\textstyle\fint_{#1}^{#2}}}

\newcommand{\tsum}[2]{{\textstyle\sum_{#1}^{#2}}}

\renewcommand{\max}{\sup}
\renewcommand{\min}{\inf}

\DeclareMathOperator{\card}{card}

\newcommand{\Clos}[1]{\mathop{\mathrm{Clos}}#1}

\newcommand{\measureball}[2]{{#1}\,{#2}}

\DeclareMathOperator{\Nor}{Nor}     
\DeclareMathOperator{\Tan}{Tan}     
\DeclareMathOperator{\spt}{spt}     
\DeclareMathOperator{\im}{im}       
\DeclareMathOperator{\Int}{Int}     
\DeclareMathOperator{\Lip}{Lip}     
\DeclareMathOperator{\dmn}{dmn}     
\DeclareMathOperator{\dist}{dist}   
\DeclareMathOperator{\Hom}{Hom}     
\DeclareMathOperator{\graph}{graph} 
\DeclareMathOperator{\ap}{ap}       

\DeclareMathOperator{\Lap}{Lap}

\newcommand{\taumu}[2]{\boldsymbol{\tau}_{#1} \circ \boldsymbol{\mu}_{#2}}
\newcommand{\mutau}[2]{\boldsymbol{\mu}_{1/{#1}} \circ \boldsymbol{\tau}_{-{#2}}}

\begin{document}


\title{Second order rectifiability of integral varifolds of locally bounded
first variation}
\author{Ulrich Menne\thanks{The author acknowledges financial support via the
DFG Forschergruppe 469. The major part of this work was accomplished while the
author was at the University of T\"ubingen. Some parts were done at the ETH
Z\"urich and the work was put in its final form at the AEI Golm. \textit{AEI
publication number.} AEI-2008-065}} \maketitle
\begin{abstract}
	It is shown that every integral varifold in an open subset of
	Euclidean space whose first variation with respect to area is
	representable by integration can be covered by a countable collection
	of submanifolds of the same dimension of class $\class{2}$ and that
	their mean curvature agrees almost everywhere with the variationally
	defined generalised mean curvature of the varifold.
\end{abstract}

\begin{center}
\end{center}

\section*{Introduction}
\paragraph{Overview} In the present paper the existence of an approximate
second order structure for integral varifolds in Euclidean space whose first
variation with respect to area is representable by integration is established.
Such varifolds are called ``of locally bounded first variation'' in
\cite{MR87a:49001}. Moreover, it is proven that the variationally defined
generalised mean curvature of the varifold agrees almost everywhere with the
mean curvature induced from the approximate second order structure. This
problem can be considered a geometric, nonlinear, higher multiplicity version
of the following linear one: Prove existence of approximate second order
differentials for weakly differentiable functions whose distributional
Laplacian is representable by integration (i.e., by a ``vector-valued Radon
measure'') and show that these differentials satisfy the equation Lebesgue
almost everywhere. Clearly, the linear case itself is not too hard to solve,
and in fact follows immediately from classical results if the distributional
Laplacian is integrable with respect to Lebesgue measure to a power larger
than $1$. Nevertheless, the main objective of the present paper is to develop
a method which is based on the study of the nearly linear case and is
sufficiently robust to be applied to the present elliptic system of geometric
partial differential equations involving higher multiplicity.

Results of the type obtained in the present paper have proven useful for
example in the context of Brakke's mean curvature flow or sharp and diffuse
interfaces or image reconstruction or the Willmore functional, see
\cite{MR485012,MR1906780,MR2047075,MR2253464,arxiv.0902.1816v1,
MR1959769,rsch:willmore} and the references therein.
\paragraph{Result of the present paper in the context of known results} Fix
positive integers $\vdim$ and $\adim$ with $\vdim < \adim$. The principal
result is as follows, see Section \ref{sec:notation} for the notation used.
\begin{intro_theorem} [see \ref{thm:main_theorem}] \label{ithm:c2}
	Suppose $U$ is an open subset of $\rel^\adim$, $V \in \IVar_\vdim ( U
	)$ and $\| \delta V \|$ is a Radon measure.

	Then there exists a countable collection $C$ of $\vdim$ dimensional
	submanifolds of $\rel^\adim$ of class $\class{2}$ such that $\| V \| (
	U \without \bigcup C ) = 0$ and each member $M$ of $C$ satisfies
	\begin{gather*}
		\mathbf{h} ( V; z ) = \mathbf{h} ( M ; z ) \quad \text{for
		$\| V \|$ almost all $z \in U \cap M$}.
	\end{gather*}
\end{intro_theorem}

In the terminology of Anzellotti and Serapioni \cite[3.1]{MR1285779} the first
part of the conclusion can be expressed equivalently by the condition that
$\classification{U}{z}{0 < \density^\vdim ( \| V \|, z ) < \infty}$ meets
every compact subset of $U$ in a set which is $( \mathscr{H}^\vdim, \vdim )$
rectifiable of class $\mathscr{C}^2$. The second part of the assertion is
sometimes called ``locality of the mean curvature'', see Sch\"atzle \cite[\S
4]{rsch:willmore}.

Theorem \ref{ithm:c2} contains (and reproves) the fact that $\mathbf{h}(V;z)
\in \Nor^\vdim ( \| V \|, z )$ for $\| V \|$ almost all $z$ previously
obtained by Brakke \cite[5.8]{MR485012}, see \ref{remark:brakke}. Moreover, it
is worth noting, see \ref{remark:hutchinson}, that if $V$ is a curvature
varifold with boundary in $U$ in the sense of Mantegazza \cite[Definition
3.1]{MR1412686} then $V$ satisfies the hypotheses of Theorem \ref{ithm:c2}
and, taking $C$ as in its conclusion, the second fundamental form of $V$
agrees almost everywhere with the second fundamental form induced by the
members $M$ of $C$.

Evidently, Theorem \ref{ithm:c2} implies that the function mapping $\| V \|$
almost every $z$ onto the orthogonal projection of $\rel^\adim$ onto the
approximate $\vdim$ dimensional tangent plane of $\| V \|$ at $z$ is $(\|V\|,
\vdim)$ approximately differentiable. If the first variation of $\| V \|$
satisfies the integrability condition \eqref{hp} below with sufficiently large
exponent $p$ then this map is in fact differentiable in a stronger $\Lp{2} (
\| V \|, \Hom ( \rel^\adim, \rel^\adim))$ sense. Whenever $U$ is an open
subset of $\rel^\adim$, $V \in \IVar_\vdim ( U )$ and $1 \leq p \leq \infty$,
the varifold $V$ is said to satisfy \eqref{hp} if and only if $\| \delta V \|$
is a Radon measure and, if $p > 1$,
\begin{gather} \label{hp}
	\begin{split}
		& ( \delta V ) (g) = - \tint{}{} \mathbf{h} ( V; z ) \bullet g
		(z) \ud \| V \| z \quad \text{for $g \in \mathscr{D} ( U,
		\rel^\adim )$}, \\
		& \mathbf{h}( V ; \cdot ) \in \Lp{p} ( \| V \| \restrict K,
		\rel^\adim ) \quad \text{whenever $K$ is a compact subset of
		$U$}.
	\end{split}
	\tag{$H_p$}
\end{gather}
\begin{intro_theorem} [see \ref{thm:decay_rates} and \ref{remark:l2_differentiability}] \label{ithm:l2_diff}
	Suppose $U$ is an open subset of $\rel^\adim$, $1 \leq p \leq \infty$,
	and $V \in \IVar_\vdim ( U )$ satisfies \eqref{hp}.

	If either $\vdim = 1$ or $\vdim = 2$ and $p > 1$ or $\vdim > 2$ and
	$p \geq 2\vdim /(\vdim+2)$, then for $\| V \|$ almost all $a$
	\begin{gather*}
		\tfint{\cball{a}{r}}{} ( | R (z) - R (a) - \left < R (a) (z-a),
		\ap DR (a) \right > |/|z-a| )^2 \ud \| V \| z \to 0
	\end{gather*}
	as $r \to 0+$ where $R(z) = \project{\Tan^\vdim ( \| V \|, z )} \in
	\Hom ( \rel^\adim, \rel^\adim )$ and the approximate differential is
	taken with respect to $( \| V \|, \vdim )$.
\end{intro_theorem}

With the possible exception of the case $\vdim = 2$ this differentiability
result is optimal with respect to the assumptions on $p$, i.e. whenever $\vdim
> 2$ and $\frac{\vdim p}{\vdim-p} < 2$ there exists an integral varifold
satisfying \eqref{hp} not having the property in question, see
\ref{remark:diff_optimal}.

In previous work Sch\"atzle established the following result in codimension
one of the existence of submanifolds of class $\class{\infty}$ touching a
given varifold, see \cite[Proposition 4.1, Theorem 5.1]{MR2064971} where it is
phrased in terms of upper and lower height functions.
\begin{intro_theorem2} [Sch\"atzle \protect{\cite{MR2064971}}] \label{ithm:rsch}
	Suppose $U$ is an open subset of $\rel^\adim$, $p > \vdim = \adim -
	1$, $p \geq 2$, and $V \in \IVar_\vdim ( U )$ satisfies \eqref{hp}.

	Then for $\| V \|$ almost all $a$ there exists $0 < r < \infty$ such
	that
	\begin{gather*}
		\oball{a+v}{r} \cap \spt \| V \| = \emptyset
	\end{gather*}
	whenever $v \in \Nor^\vdim ( \| V \|, a )$ with $|v|=r$.
\end{intro_theorem2}

This is the key to showing that such a varifold satisfies the conclusion of
Theorem \ref{ithm:c2}, see Sch\"atzle \cite[Theorem 6.1]{MR2064971}, and, in
combination with previous results of the author in \cite[3.7,
3.9]{snulmenn.isoperimetric}, also that it satisfies the conclusion of Theorem
\ref{ithm:l2_diff}. Evidently, see for example
\cite[1.2]{snulmenn.isoperimetric}, Sch\"atzle's Theorem does not extend to
the case $p < \vdim$. Also, the use of the theory of viscosity solutions for
fully nonlinear equations, more precisely the results of Caffarelli
\cite{MR1005611} and Trudinger \cite{MR995142}, leads to the restriction to
codimension one, i.e. $\vdim = \adim - 1$.

Therefore, in order to establish Theorem \ref{ithm:c2}, a different method
needs to be developed which is able to deal both with the low integrability of
the generalised mean curvature and with higher codimension. The main
independent result in this process is the following Theorem stated here in the
case of Laplace's operator.
\begin{intro_theorem} [see \ref{thm:criterion}] \label{ithm:crit}
	Suppose $U$ is an open subset of $\rel^\vdim$, $u : U \to \rel^\codim$
	is weakly differentiable, $j \in \{0,1\}$, $1 \leq q < \infty$,
	\begin{gather*}
		h(a,r) = \inf \big \{ \tsum{i=0}{j} r^{i-\vdim/q}
		\norm{\weakD^i (u-v)}{q}{a,r} \with \text{$v \in \mathscr{E} (
		\oball{a}{r}, \rel^\codim )$, $\Lap v = 0$} \big \}
	\end{gather*}
	whenever $a \in U$, $0 < r < \infty$ with $\oball{a}{r} \subset U$ and
	$A$ denotes the set of all $a \in U$ such that
	\begin{gather*}
		\limsup_{r \to 0+} r^{-2} h(a,r) < \infty.
	\end{gather*}

	Then for $\mathscr{L}^\vdim$ almost all $a \in A$ there exists a
	polynomial function $Q_a : \rel^\vdim \to \rel^\codim$ of degree at
	most $2$ such that
	\begin{gather*}
		\lim_{r \to 0+} r^{-2} \tsum{i=0}{j} r^{i-\vdim/q}
		\norm{\weakD^i (u-Q_a)}{q}{a,r} = 0.
	\end{gather*}
\end{intro_theorem}
Here the seminorms $\norm{\cdot}{q}{a,r}$ correspond to $\Lp{q} (
\mathscr{L}^\vdim \restrict \oball{a}{r} )$. The weaker statement which
results when the condition $\Lap v = 0$ is replaced by $D^2 v = 0$ is
contained in Calder\'on and Zygmund \cite[Theorem 5]{MR0136849} if $q >1$.
However, the construction of affine comparison functions at a given point from
information on the distributional Laplacian of $u$ may -- for integral orders
of differentiability -- fail at individual points, see
\cite[8.6]{snulmenn:decay.v2}. This corresponds to the well known fact of the
nonexistence of Schauder estimates for the H\"older exponent $1$. In this
respect the value of the current theorem stems from the fact that harmonic
comparison functions are readily constructed independent of the order of
differentiability considered, cp. \ref{corollary:lebesgue_points_p}. In fact,
if $j = 1$, $q>1$ and denoting by $T \in \mathscr{D}' ( U , \rel^\codim )$
the distributional Laplacian of $u$ then
\begin{gather*}
	\Gamma^{-1} h (a,r) \leq r^{1-\vdim/q} \dnorm{T}{q}{a,r} \leq \Gamma h
	(a,r)
\end{gather*}
whenever $a \in U$, $0 < r < \infty$, $\oball{a}{r} \subset U$ and $u |
\oball{a}{r} \in \Sob{}{1}{q} ( \oball{a}{r}, \rel^\codim )$ where $\Gamma$ is
a positive, finite number depending only on $\adim$ and $q$ and
$\dnorm{\cdot}{q}{a,r}$ denotes the seminorm corresponding to $\big (
\Sob{0}{1}{q/(q-1)} ( \oball{a}{r}, \rel^\codim ) \big )^\ast$. In particular, if
$T$ is representable by integration and $q < \vdim/(\vdim-1)$ if $\vdim > 1$
then one verifies $\mathscr{L}^\vdim ( U \without A ) = 0$. An
extensive study of both integral and nonintegral orders of differentiability
for solutions of linear elliptic partial differential equations in
nondivergence form can be found in Calder\'on and Zygmund \cite{MR0136849}.

In passing to divergence form equations, one is naturally lead to consider the
related problem for distributions:
\begin{intro_theorem} [see \ref{corollary:lebesgue_points_p} and
\ref{thm:distrib_diff}] \label{ithm:distribution}
	Suppose $U$ is an open subset of $\rel^\vdim$, $1 \leq q < \infty$, $T
	\in \mathscr{D}' ( U, \rel^\codim )$ and $A$ denotes the set of all $a
	\in U$ such that
	\begin{gather*}
		\limsup_{r \to 0+} r^{-1-\vdim/q} \dnorm{T}{q}{a,r} < \infty.
	\end{gather*}

	Then for $\mathscr{L}^\vdim$ almost every $a \in A$ there exists a
	unique constant distribution $T_a \in \mathscr{D}' ( U, \rel^\codim )$
	such that
	\begin{gather*}
		\lim_{r \to 0+} r^{-1-\vdim/q} \dnorm{T-T_a}{q}{a,r} = 0.
	\end{gather*}
\end{intro_theorem}
This may be seen as a Lebesgue point theorem for distributions. In case $q >
1$, it is in fact a corollary to Theorem \ref{ithm:crit} obtainable by
representing $T$ locally as distributional Laplacian of some function $u$. In
contrast, the case $q=1$ is independent from the other results of the present
paper.

Finally, it should be noted that the proof of Theorem \ref{ithm:crit} only
relies on a priori estimates in Lebesgue spaces, i.e. ``$\Lp{p}$ theory'',
which are known to hold for a much wider class of linear equations, see Agmon,
Douglis and Nirenberg \cite{MR0125307,MR0162050}.

\paragraph{Outline of the proofs}
To prove Theorem \ref{ithm:crit}, one considers the subsets of $A_k$ of $A$ of
all $a \in A$ such $h(a,r) \leq k r^2$ whenever $0 < r < 1/k$. Denoting by
$v_{a,r} : \oball{a}{r} \to \rel^\codim$ harmonic functions essentially
realising the infimum in the definition of $h$, one then uses the partition of
unity with estimates from \cite[3.1.13]{MR41:1976} together with well known a
priori estimates for the Laplace operator to construct functions $v_k :
\rel^\vdim \to \rel^\codim$ with the following properties, see
\ref{lemma:main_lemma}:
\begin{enumerate}
	\item There holds
	\begin{gather*}
		\tsum{i=0}{j} r^{i-\vdim/q} \norm{\weakD^i (v_k-u)}{q}{a,r}
		\leq \Gamma k r^2
	\end{gather*}
	for $a \in A_k$ and $0 < r < (36k)^{-1}$ and $\Gamma$ a positive,
	finite number depending only on $\adim$ and $q$, in particular $v_k
	(x) = u(x)$ for $\mathscr{L}^\vdim$ almost all $x \in A_k$.
	\item The distributional Laplacian of $v_k$ is represented by a
	function locally in $\Lp{\infty} ( \mathscr{L}^\vdim, \rel^\codim )$.
\end{enumerate}
Then clearly $v_k$ locally belongs to $\Sob{}{2}{q} ( \rel^\vdim, \rel^\codim
)$ for $1 \leq q < \infty$ and the conclusion of Theorem \ref{ithm:crit}
follows from by now classical differentiability results for functions in
Sobolev spaces which where also obtained by Calder\'on and Zygmund in
\cite{MR0136849}. An important feature of this proof is that it is readily
adapted to the case where the Laplace operator is replaced by the Euler
Lagrange differential operator $L_F$ corresponding to an integrand $F : \Hom (
\rel^\vdim, \rel^\codim ) \to \rel$ of class $\class{2}$ sufficiently close to
the Dirichlet integrand, i.e. $\Lip D^2 F < \infty$ and
\begin{gather*}
	\big | \left < ( \tau_1, \tau_2 ), D^2 F ( \sigma ) \right > - \tau_1
	\bullet \tau_2 \big | \leq \varepsilon \quad \text{for $\sigma, \tau_1,
	\tau_2 \in \Hom ( \rel^\vdim, \rel^\codim )$}
\end{gather*}
with suitable number $\varepsilon$.

Next, it will explained how this result on a rather restricted class of
differential operators can be used to treat the general case. For this purpose
let $U$ be an open subset of $\rel^\adim$ and let $V \in \IVar_\vdim ( U )$ be such
that $\| \delta V \|$ is a Radon measure. Comparing the behaviour of $V$ near
certain ``good'' points to the behaviour of harmonic functions, a
procedure developed by De~Giorgi in \cite{MR0179651} and Almgren in
\cite{MR0225243}, one proves the \emph{tilt decay estimate}
\begin{gather*}
	\limsup_{r \to 0+} r^{-\tau-\vdim/2} \big ( \tint{\oball{a}{r} \times
	\grass{\adim}{\vdim}}{} | \project{S} - \project{T} |^2 \ud V (z,S)
	\big )^{1/2} < \infty
\end{gather*}
for $V$ almost all $(a,T)$ where $0 < \tau < 1$ if $\vdim \in \{1,2\}$ and
$\tau = \frac{\vdim}{2(\vdim-1)} < 1$ if $\vdim > 2$. This has been done by
the author in \cite[8.6]{snulmenn:decay.v2} extending results of Brakke
\cite[5.7,\,5]{MR485012} who proved the case $\tau = 1/2$ with ``$< \infty$''
replaced by $=0$ which is sufficient for the proof of all Theorems stated in
the Introduction. As the order of differentiability considered is nonintegral,
i.e. $0 < \tau < 1$, the argument applies, in contrast to those of the present
paper, in a direct way to all points satisfying a simple set of conditions,
see \cite[8.3]{snulmenn:decay.v2}.

The principal idea to prove Theorem \ref{ithm:c2} is now to use the tilt decay
estimate, to construct a sequence functions $g_i : \rel^\vdim \to
\rel^\codim$,  $\mathscr{L}^\vdim$ measurable sets $K_i \subset
\rel^\vdim$ and distributions $T_i \in \mathscr{D}' ( \rel^\vdim,
\rel^\codim)$ with the following properties:
\begin{enumerate}
	\item The varifold is covered by suitably rotated graphs of the $g_i |
	K_i$.
	\item The distribution $T_i$ corresponds to the Euler Lagrange
	differential operator associated to the nonparametric area integrand
	$\Phi$ applied to $g_i$.
	\item \label{iitem:diff} There holds
	\begin{gather*}
		\lim_{r \to 0+} r^{-1-\vdim} \tint{\oball{x}{r}}{} | Dg_i
		(\zeta) - Dg_i(x)|^2 \ud \mathscr{L}^\vdim \zeta = 0 \quad
		\text{whenever $x \in K_i$}.
	\end{gather*}
	\item \label{iitem:lip} The Lipschitz constant of the $g_i$ is small.
	\item \label{iitem:dis} The distributions $T_i$ satisfy the conclusion
	of Theorem \ref{ithm:distribution} with $q=1$ and $A$ replaced by
	$K_i$ with
	constant distribution given by the generalised mean curvature of the
	varifold.
\end{enumerate}
Condition \eqref{iitem:lip} is the minimum condition needed to be able to
replace $\Phi$ with some integrand $F$ of the type discussed before in the
definition of $T_i$ without changing it, see \ref{lemma:cutoff}. The basis for
the construction of $g_i$, $K_i$, and $T_i$ is an approximation by $\qspace_Q
( \rel^\codim )$ valued functions where the space $\qspace_Q ( \rel^\codim )$
is isometric to the $Q$ fold product of $\rel^\codim$ divided by the action of
group of permutations of $\{ 1, \ldots, Q \}$. Here the version of the author
in \cite[4.8]{snulmenn:decay.v2} is employed which contains some estimates
designed for the current applications and was obtained by combining and
extending similar constructions of Almgren in \cite[\S 3]{MR1777737} and
Brakke in \cite[5.4]{MR485012}.  This yields Lipschitzian functions $f_i : K_i
\to \qspace_{Q_i} ( \rel^\codim )$ with small Lipschitz constant for suitable
positive integers $Q_i$.  Denoting the ``centre'' of $S \in \qspace_Q (
\rel^\codim )$ by $\boldsymbol{\eta}_Q ( S ) = Q^{-1} \sum_{j=1}^Q y_j$
whenever $y_1, \ldots, y_Q \in \rel^\codim$ correspond to $S$, the functions
$g_i$ are then constructed in \ref{lemma:great_approx} as extensions of
$\boldsymbol{\eta}_{Q_i} \circ f_i$.  In this process the conditions
\eqref{iitem:diff} and \eqref{iitem:dis} are ultimately consequences of the
tilt decay estimate.

The final step in the proof of Theorem \ref{ithm:c2} is now to construct for
fixed $i$ and $x \in K_i$ comparison functions $v_r \in \Sob{}{1}{2} (
\oball{x}{r} , \rel^\codim )$ with $L_F (v_r) =0$ for $0 < r < \infty$ and
estimating $g_i-v_r$ in $\oball{x}{r}$, see
\ref{lemma:l1_estimate}--\ref{lemma:arbi}. The natural choice is to take $v_r$
as solution of the Dirichlet problem with boundary values given by $g_i$.  If
$q$ in \eqref{iitem:dis} would satisfy $q > 1$ this would immediately yield an
estimate of $g_i-v_r$ in $\Sob{}{1}{q} ( \oball{x}{r} , \rel^\codim )$. In
case $q=1$ the estimate needs to be obtained differently, namely, linearising
$F$ and estimating the remaining terms with the help of condition
\eqref{iitem:diff}, one obtains an estimate in $\Lp{1} ( \oball{x}{r},
\rel^\codim )$ instead, see \ref{lemma:ms}. Then the extended version of
Theorem \ref{ithm:crit} with $L_F$ replacing $\Lap$, see \ref{thm:criterion},
implies the first part of Theorem \ref{ithm:c2}.  Recalling condition
\eqref{iitem:dis}, the second part is derived similarly by using functions
$w_r \in \Sob{}{1}{2} ( \oball{x}{r}, \rel^\codim )$ with $L_F ( w_r ) = (
T_i)_x$ where $( T_i )_x$ is the constant distribution corresponding to $T_i$
at $x$ as in Theorem \ref{ithm:distribution}.

\paragraph{Organisation of paper}
In Section \ref{sec:notation} the Notation is fixed. Section \ref{sec:crit}
contains all results which can be phrased solely in terms of elliptic partial
differential equations and distributions, in particular Theorem
\ref{ithm:crit} and the case $q>1$ of Theorem \ref{ithm:distribution}. Section
\ref{sec:main_theorem} is devoted to the proof of Theorem \ref{ithm:c2} whereas
Section \ref{sec:applications} contains Theorem \ref{ithm:l2_diff}. Finally,
Appendix \ref{sec:app} gives the proof of the case $q=1$ of Theorem
\ref{ithm:distribution}.

\paragraph{Acknowledgements} The author offers his thanks to his PhD advisor
Professor Dr. Reiner Sch\"atzle who lead him towards the study of this
problem. The author also thanks Professor Dr. Tom Ilmanen for several related
discussions.

\section{Notation} \label{sec:notation}
The notation from Federer \cite{MR41:1976} and Allard \cite{MR0307015} is used
with some modifications and additions described in \cite[\S 1,\,\S
2]{snulmenn:decay.v2}. Additionally, whenever $M$ is a submanifold of
$\rel^\adim$ of class $\class{2}$ the mean curvature of $M$ at $z \in M$ is
denoted by $\mathbf{h} (M;z)$, cp. Allard \cite[2.5\,(2)]{MR0307015}. And if
$U$ is an open subset of $\rel^\vdim$ and $Y$ is a Banach space then $T$ is
called a \emph{constant distribution in $U$ of type $Y$} if and only if for
some $\alpha \in Y^\ast$ there holds $T(\theta) = \int_U \alpha \circ \theta
\ud \mathscr{L}^\vdim$ for $\theta \in \mathscr{D} ( U, Y )$. Moreover, a
subset of a topological space is called \emph{universally measurable} if and
only if it is measurable with respect to every Borel measure on that space.

The reader might want to recall the following maybe less commonly used symbols
either taken from \cite[2.2.6, 2.8.1, 1.10.1]{MR41:1976} or introduced in
\cite[\S 1]{snulmenn:decay.v2}: $\nat$ denoting the positive integers,
$\oball{a}{r}$ and $\cball{a}{r}$ denoting respectively the open and closed
ball with centre $a$ and radius $r$, $\bigodot^i ( V,W )$ and $\bigodot^i V$
denoting the vector space of all $i$ linear symmetric functions (forms)
mapping $V^i$ into $W$ and $\rel$ respectively and the seminorms
\begin{align*}
	\norm{f}{p}{a,r} & = \big ( \tint{\oball{a}{r}}{} |f|^p \ud
	\mathscr{L}^\vdim \big )^{1/p} \quad \text{if $p < \infty$}, \\
	\norm{f}{\infty}{a,r} & = \inf \{ t \with \mathscr{L}^\vdim (
	\classification{\oball{a}{r}}{x}{|f(x)|>t} ) = 0 \}, \\
	\idnorm{T}{i}{q}{a,r} & = \sup \{ T ( \theta ) \with \theta \in
	\mathscr{D} ( U, \rel^\codim ), \spt \theta \subset \oball{a}{r},
	\norm{D^{-i} \theta}{p}{a,r} \leq 1 \}
\end{align*}
whenever $\vdim, \adim \in \nat$, $\vdim < \adim$, $U$ is an open subset of
$\rel^\vdim$, $a \in \rel^\vdim$, $0 < r < \infty$ with $\oball{a}{r} \subset
U$, $f$ is an $\mathscr{L}^\vdim \restrict \oball{a}{r}$ measurable function
with values in a Hilbert space, $1 \leq p \leq \infty$, $1 \leq q \leq \infty$
with $1/p + 1/q = 1$, $i$ is a negative integer, and $T \in \mathscr{D}' ( U,
\rel^\vdim )$.

\section{A criterion for second order differentiability in Lebesgue spaces}
\label{sec:crit}
The purpose of this section is to prove \ref{thm:criterion} which contains
Theorem \ref{ithm:crit} of the Introduction and to provide the preparations
necessary for its application in Section \ref{sec:main_theorem}.

First, in \ref{miniremark:situation_F} the situation studied is described.
Then, for the convenience of the reader, in
\ref{thm:global_w1p_estimate}--\ref{lemma:solutions} adaptions and
applications of standard theory are carried out. The main ingredient in the
proof of \ref{thm:criterion} is contained in \ref{lemma:main_lemma}. The part
$q>1$ of Theorem \ref{ithm:distribution} is provided in
\ref{corollary:lebesgue_points_p}. Finally, in
\ref{lemma:l1_estimate}--\ref{lemma:arbi} it is shown how a certain
nonintegral differentiability condition on the solution $u$ allows to treat
the case where estimates for $L_F(u)$, see \ref{miniremark:situation_F}, are
only available in $\dnorm{\cdot}{1}{a,r}$.
\begin{miniremark} \label{miniremark:situation_F}
	Suppose $\vdim, \adim \in \nat$, $\vdim < \adim$,
	\begin{gather*}
		\text{$e_1, \ldots, e_\vdim$ and $X_1, \ldots, X_\vdim$}
	\end{gather*}
	are dual orthonormal bases of $\rel^\vdim$ and $\bigodot^1 \rel^\vdim$
	and
	\begin{gather*}
		\text{$\upsilon_1, \ldots, \upsilon_\codim$ and $Y_1, \ldots,
		Y_\codim$}
	\end{gather*}
	are dual orthonormal bases of $\rel^\codim$ and $\bigodot^1
	\rel^\codim$.  The norm $\| \Psi \|$ denotes for any $\Psi \in
	\bigodot^2 \Hom ( \rel^\vdim, \rel^\codim )$ the smallest nonnegative
	number $M$ such that
	\begin{gather*}
		\Psi ( \sigma, \tau ) \leq M | \sigma | | \tau | \quad
		\text{for $\sigma, \tau \in \Hom ( \rel^\vdim, \rel^\codim)$}.
	\end{gather*}
	The expression $\Psi ( \sigma, \tau )$ for $\sigma, \tau \in \Hom (
	\rel^\vdim, \rel^\codim )$ will be denoted alternately by $\left < (
	\sigma, \tau ), \Psi \right >$ and, using $\odot$ to denote
	multiplication in $\bigodot_\ast \Hom ( \rel^\vdim, \rel^\codim )$,
	see \cite[1.9.1]{MR41:1976}, also by $\left < \sigma \odot \tau, \Psi
	\right >$. It equals
	\begin{gather*}
		\sum_{i=1}^\vdim \sum_{j=1}^\codim \sum_{k=1}^\vdim
		\sum_{l=1}^\codim \Psi_{i,j;k,l} \left < \sigma (e_i ), Y_j
		\right > \left < \tau (e_k), Y_l \right >
	\end{gather*}
	where $\Psi_{i,j;k,l} = \Psi ( X_i \, \upsilon_j, X_k \upsilon_l )$
	and $X \, \upsilon$ maps $x \in \rel^\vdim$ onto $X(x) \upsilon \in
	\rel^\codim$ whenever $X \in \bigodot^1 \rel^\vdim$ and $\upsilon \in
	\rel^\codim$.
	
	Let $\Upsilon \in \bigodot^2 \Hom ( \rel^\vdim, \rel^\codim )$ be
	defined by
	\begin{gather*}
		\Upsilon ( \sigma, \tau ) = \sigma \bullet \tau \quad
		\text{for $\sigma, \tau \in \Hom ( \rel^\vdim ,
		\rel^\codim)$},
	\end{gather*}
	and suppose $F : \Hom ( \rel^\vdim , \rel^\codim ) \to \rel$ is of
	class $\class{2}$, $0 \leq \varepsilon < \infty$, and
	\begin{gather*}
		\| D^2 F ( \sigma ) - \Upsilon \| \leq \varepsilon \quad
		\text{whenever $\sigma \in \Hom ( \rel^\vdim , \rel^\codim
		)$}.
	\end{gather*}
	The quantity $\Lip D^2 F$ will be computed with respect to $| \cdot |$
	on $\Hom ( \rel^\vdim, \rel^\codim )$ and $\| \cdot \|$ on $\bigodot^2
	\Hom ( \rel^\vdim, \rel^\codim )$.

	To each such $F$ there corresponds the Euler Lagrange differential
	operator $L_F$ which associates to every $u \in \Sob{}{1}{1} ( U,
	\rel^\codim )$ for some open subset $U$ of $\rel^\vdim$ a distribution
	$L_F (u)$ in $\mathscr{D} ' ( U, \rel^\codim)$ defined by
	\begin{gather*}
		L_F ( u ) ( \theta ) = - {\textstyle\int_U} \left < D \theta
		(x), DF ( \weakD u (x) ) \right > \ud \mathscr{L}^\vdim x
		\quad \text{for $\theta \in \mathscr{D} ( U , \rel^\codim )$}.
	\end{gather*}
	There also occurs the linear function $C_F ( \sigma ) : \bigodot^2 (
	\rel^\vdim , \rel^\codim ) \to \rel^\codim$ which for $\sigma \in \Hom
	( \rel^\vdim, \rel^\codim )$ is given by
	\begin{gather*}
		\left < \phi, C_F ( \sigma ) \right > = \sum_{i=1}^\vdim
		\sum_{j=1}^\codim \sum_{k=1}^\vdim \sum_{l=1}^\codim \left < (
		X_i \upsilon_j, X_k \upsilon_l ), D^2 F ( \sigma ) \right >
		\left < \phi ( e_i, e_k ) , Y_j \right > \upsilon_l
	\end{gather*}
	whenever $\phi \in \bigodot^2 ( \rel^\vdim , \rel^\codim )$.
	The function $C_F ( \sigma )$ is uniquely determined by $D^2 F (
	\sigma)$, see \cite[5.2.11]{MR41:1976}. One obtains by partial
	integration for $u \in \Sob{}{2}{1} ( U, \rel^\codim )$, $\theta \in
	\mathscr{D} ( U, \rel^\codim )$
	\begin{gather*}
		L_F (u) ( \theta ) = {\textstyle\int_U} \theta (x) \bullet
		\left < \weakD^2 u (x), C_F ( \weakD u(x) ) \right > \ud
		\mathscr{L}^\vdim x.
	\end{gather*}
	Sometimes also $S : \bigodot^2 ( \rel^\vdim, \rel^\codim ) \to
	\rel^\codim$ corresponding to the Dirichlet integrand, i.e. $F
	(\sigma) = | \sigma |^2/2$ for $\sigma \in \Hom ( \rel^\vdim,
	\rel^\vdim )$, (and therefore to $\Upsilon$) will be used. Note $\left
	< \phi, S \right > = \sum_{i=1}^\vdim \phi ( e_i, e_i )$ whenever
	$\phi \in \bigodot^2 ( \rel^\vdim, \rel^\codim )$. One may check that
	with $\kappa = 2^{1/2} \vdim (\codim)$
	\begin{gather*}
		| C_F ( \sigma ) | \leq \kappa \| D^2 F (\sigma) \|, \quad
		| C_F ( \sigma ) - S | \leq \kappa \varepsilon, \\
		| C_F ( \sigma ) - C_F ( \tau ) | \leq \kappa \| D^2 F (
		\sigma ) - D^2 F ( \tau ) \|
	\end{gather*}
	for $\sigma, \tau \in \Hom ( \rel^\vdim, \rel^\codim )$ where
	$| \cdot |$ denotes the norm associated to the inner product on $\Hom
	\big ( \bigodot^2 ( \rel^\vdim, \rel^\codim ) , \rel^\codim \big )$,
	see \cite[1.7.9, 1.10.6]{MR41:1976}.
\end{miniremark}
\begin{theorem} \label{thm:global_w1p_estimate}
	Suppose $\adim \in \nat$ and $1 < p < \infty$.

	Then there exist positive, finite numbers $\varepsilon$ and $\Gamma$
	with the following property.

	If $\adim > \vdim \in \nat$, $\Upsilon$ is as in
	\ref{miniremark:situation_F}, $a \in \rel^\vdim$, $0 < r < \infty$,
	\begin{gather*}
		\text{$A : \oball{a}{r} \to {\textstyle\bigodot^2} \Hom (
		\rel^\vdim, \rel^\codim )$ is $\mathscr{L}^\vdim
		\restrict \oball{a}{r}$ measurable}, \\
		\| A (x) - \Upsilon \| \leq \varepsilon \quad \text{whenever
		$x \in \oball{a}{r}$},
	\end{gather*}
	then for every 
	$T \in \mathscr{D}' ( \oball{a}{r}, \rel^\codim )$ with
	$\dnorm{T}{p}{a,r} < \infty$ there exists an $\mathscr{L}^\vdim
	\restrict \oball{a}{r}$ almost unique $u \in \Sob{0}{1}{p} (
	\oball{a}{r}, \rel^\codim)$ such that
	\begin{gather*}
		- \tint{\oball{a}{r}}{} \left < D \theta (x) \odot
		\weakD u(x), A (x) \right > \ud \mathscr{L}^\vdim x = T
		( \theta ) \quad \text{for $\theta \in \mathscr{D} (
		\oball{a}{r}, \rel^\codim )$}.
	\end{gather*}
	Moreover, whenever $u$ and $T$ are related as above there holds
	\begin{gather*}
		\norm{\weakD u}{p}{a,r} \leq \Gamma \dnorm{T}{p}{a,r}.
	\end{gather*}
\end{theorem}
\begin{proof}
	By the Neumann series (cf. \cite[3.1.11]{MR41:1976}) it is enough to
	consider the case $\varepsilon = 0$. Note also that there exists $g
	\in \Lp{p} ( \mathscr{L}^\vdim \restrict \oball{a}{r}, \Hom (
	\rel^\vdim , \rel^\codim ) )$ with $T ( \theta ) = -
	\int_{\oball{a}{r}} g \bullet D\theta \ud \mathscr{L}^\vdim$ for
	$\theta \in \mathscr{D} ( \oball{a}{r}, \rel^\codim)$ and
	$\dnorm{T}{p}{a,r} = \norm{g}{p}{a,r}$ by Hahn Banach's theorem.

	The conclusion then follows from \cite[Theorem 10.15]{MR1962933} in
	case $p \geq 2$ to which the case $p < 2$ reduces by use of a duality
	argument.
%
\end{proof}
\begin{theorem} \label{thm:linear_A}
	Suppose $\adim \in \nat$, $1 < q < \infty$, and $1 < p < \infty$.

	Then there exists a positive, finite number $\varepsilon$ with the
	following property.

	If $\adim > \vdim \in \nat$, $\Upsilon$ is as in
	\ref{miniremark:situation_F}, $a \in \rel^\vdim$, $0 < r < \infty$,
	\begin{gather*}
		\text{$A : \oball{a}{r} \to {\textstyle\bigodot^2} \Hom (
		\rel^\vdim, \rel^\codim )$ is $\mathscr{L}^\vdim
		\restrict \oball{a}{r}$ measurable}, \\
		\| A (x) - \Upsilon \| \leq \varepsilon \quad \text{whenever
		$x \in \oball{a}{r}$},
	\end{gather*}
	and $u \in \Sob{}{1}{q} ( \oball{a}{r}, \rel^\codim)$,
	$T \in \mathscr{D}' ( \oball{a}{r}, \rel^\codim )$ satisfy
	\begin{gather*}
		- {\textstyle\int_{\oball{a}{r}}} \left < D \theta (x) \odot
		\weakD u(x), A (x) \right > \ud \mathscr{L}^\vdim x = T
		( \theta ) \quad \text{for $\theta \in \mathscr{D} (
		\oball{a}{r}, \rel^\codim )$},
	\end{gather*}
	then
	\begin{gather*}
		\norm{\weakD u}{p}{a,r/2} \leq \Gamma \big (
		r^{-\vdim-1+\vdim/p} \norm{u}{1}{a,r} + \dnorm{T}{p}{a,r} \big
		)
	\end{gather*}
	where $\Gamma$ is a positive, finite number depending only on $\adim$
	and $p$.
\end{theorem}
\begin{proof}
	Let $0 < \delta \leq 1$, suppose $\adim$, $q$, $p$, $\vdim$,
	$\Upsilon$, $a$, $r$, $A$, $u$, and $T$ satisfy the hypotheses in the
	body of the theorem with $\varepsilon$ replaced by $\delta$ and assume
	$q \leq p$. It will be shown that $u$ satisfies the estimate in the
	conclusion of the theorem provided $\delta$ is suitably small.

	The problem will be reduced.

	First, to the \emph{case $p=q$} by constructing as solutions of
	approximating Dirichlet problems by use of 
	\ref{thm:global_w1p_estimate} a sequence of functions $u_i \in
	\Sob{}{1}{p} ( \oball{a}{r}, \rel^\codim )$ such that $u_i \to u$ in
	$\Sob{}{1}{q} ( \oball{a}{r}, \rel^\codim )$ as $i \to \infty$ and for
	$i \in \nat$
	\begin{gather*}
		- \tint{\oball{a}{r}}{} \left < D \theta (x) \odot \weakD u_i
		  (x), A (x) \right > \ud \mathscr{L}^\vdim x = T ( \theta )
	\quad \text{for $\theta \in \mathscr{D} ( \oball{a}{r}, \rel^\codim )$}
	\end{gather*}
	provided $\delta \leq \inf \{
	\varepsilon_{\ref{thm:global_w1p_estimate}} ( \adim, p ),
	\varepsilon_{\ref{thm:global_w1p_estimate}} ( \adim, q) \}$.

	Secondly, to the \emph{case $p=q$ and $\delta = 0$} by considering
	Simon's absorption lemma in \cite[p.~398]{MR1459795}.
	
	Thirdly, to the \emph{case $p = q$, $\delta = 0$ and $T=0$} by use of
	\ref{thm:global_w1p_estimate} and Poincar\'e's inequality.

	Finally, the remaining case follows by convolution from \cite[Theorems
	2.8, 2.10]{MR1814364}.
\end{proof}
\begin{theorem} \label{thm:linear_B}
	Suppose $\adim \in \nat$ and $1 < p < \infty$.

	Then there exists a positive, finite number $\varepsilon$ with the
	following property.

	If $\adim > \vdim \in \nat$, $S$ is as in
	\ref{miniremark:situation_F}, $a \in \rel^\vdim$, $0 < r < \infty$,
	\begin{gather*}
		\text{$B : \oball{a}{r} \to \Hom \big ( {\textstyle\bigodot^2} (
		\rel^\vdim, \rel^\codim ), \rel^\codim \big )$ is
		$\mathscr{L}^\vdim \restrict \oball{a}{r}$ measurable}, \\
		| B (x) - S | \leq \varepsilon \quad \text{whenever $x \in
		\oball{a}{r}$},
	\end{gather*}
	and $u \in \Sob{}{2}{p} ( \oball{a}{r}, \rel^\codim
	)$, $f \in \Lp{p} ( \mathscr{L}^\vdim \restrict \oball{a}{r},
	\rel^\codim )$ satisfy
	\begin{gather*}
		\left < \weakD^2 u (x), B (x) \right > = f(x) \quad
		\text{for $\mathscr{L}^\vdim$ almost all $x \in
		\oball{a}{r}$},
	\end{gather*}
	then
	\begin{gather*}
		\norm{\weakD^2 u}{p}{a,r/2} \leq \Gamma \big (
		r^{-2-\vdim+\vdim/p} \norm{u}{1}{a,r} + \norm{f}{p}{a,r}
		\big )
	\end{gather*}
	where $\Gamma$ is a positive, finite number depending only on $\adim$
	and $p$.
\end{theorem}
\begin{proof}
	From \cite[Theorem 7.22]{MR1814364} and Ehring's lemma, see e.g.
	\cite[Theorem\,\printRoman{1}.7.3]{MR895589}, it follows that for
	every $0 < \kappa < \infty$ there exists a positive, finite number
	$\Delta$ depending only on $\adim$, $p$, and $\kappa$ such that
	\begin{gather*}
		r^{-2-\vdim/p} \norm{v}{p}{a,r} \leq \kappa r^{-\vdim/p}
		\norm{\weakD^2 v}{p}{a,r} + \Delta r^{-2-\vdim}
		\norm{v}{1}{a,r}
	\end{gather*}
	for $v \in \Sob{}{2}{p} ( \oball{a}{r}, \rel^\codim )$.

	Now, one may readily use \cite[Theorem 9.11]{MR1814364} in conjunction
	with the absorption lemma in Simon \cite[p.~398]{MR1459795} to obtain
	the conclusion.
\end{proof}
\begin{lemma} \label{lemma:higher_differentiability}
	Suppose $\adim \in \nat$, $1 < q < \infty$, and $1 < p < \infty$.

	Then there exists a positive, finite number $\varepsilon$ with the
	following property.

	If $F$ is related to $\varepsilon$ as in \ref{miniremark:situation_F},
	$a \in \rel^\vdim$, $0 < r < \infty$, $u \in \Sob{}{1}{q} (
	\oball{a}{r}, \rel^\codim )$, and $f \in \Lp{p} ( \mathscr{L}^\vdim
	\restrict \oball{a}{r}, \rel^\codim )$ satisfy
	\begin{gather*}
		L_F (u)(\theta) = {\textstyle\int_{\oball{a}{r}}} \theta (x)
		\bullet f (x) \ud \mathscr{L}^\vdim x \quad \text{whenever
		$\theta \in \mathscr{D} ( \oball{a}{r}, \rel^\codim )$},
	\end{gather*}
	then $u$ is twice weakly differentiable and for every affine function
	$P : \rel^\vdim \to \rel^\codim$ there holds
	\begin{gather*}
		\norm{\weakD^2 u}{p}{a,r/2} \leq \Gamma \big (
		r^{-2-\vdim+\vdim/p} \norm{u - P}{1}{a,r} + \norm{f}{p}{a,r}
		\big )
	\end{gather*}
	where $\Gamma$ is a positive, finite number depending only on $\adim$
	and $p$.
\end{lemma}
\begin{proof}
	Let $\varepsilon = \varepsilon_{\ref{thm:linear_A}} ( \adim, q, p )$
	and suppose $F$, $a$, $r$, $u$, $f$, and $P$ satisfy the hypotheses in
	body of the lemma.

	Let $v = u - P$, $i \in \{ 1, \ldots, m \}$ and define for $0 < h <
	r$, $x \in \oball{a}{r-h}$
	\begin{gather*}
		u_h (x) = h^{-1} ( u (x+he_i) - u (x) ), \quad v_h
		(x) = h^{-1} ( v ( x + h e_i ) - v (x) ), \\
		A_h (x) = {\textstyle\int_0^1} D^2F ( t \weakD u (x+he_i)
		+ (1-t) \weakD u (x)  ) \ud \mathscr{L}^1 t,
	\end{gather*}
	and let $S_h \in \mathscr{D} ' ( \oball{a}{r-h}, \rel^\codim
	)$ be characterised by
	\begin{gather*}
		S_h ( \theta | \oball{a}{r-h} ) = h^{-1} \tint{\oball{a}{r}}{}
		( \theta (x-he_i) - \theta ( x ) ) \bullet f (x) \ud
		\mathscr{L}^\vdim x
	\end{gather*}
	whenever $\theta \in \mathscr{D} ( \rel^\vdim, \rel^\codim )$ with
	$\spt \theta \subset \oball{a}{r-h}$.  One readily verifies, noting
	$\weakD u_h = \weakD v_h$,
	\begin{gather*}
		- \tint{\oball{a}{r-h}}{} \left < D \theta (x) \odot \weakD
		  v_h (x), A_h (x) \right > \ud \mathscr{L}^\vdim x = S_h (
		  \theta )
	\end{gather*}
	for $\theta \in \mathscr{D} ( \oball{a}{r-h}, \rel^\codim )$. Hence,
	by \ref{thm:linear_A},
	\begin{gather*}
		\norm{\weakD v_h}{p}{a,(r-h)/2} \leq \Delta \big ( ( r -
		h)^{-1-\vdim+\vdim/p} \norm{v_h}{1}{a,r-h} +
		\dnorm{S_h}{p}{a,r-h} \big )
	\end{gather*}
	where $\Delta = \Gamma_{\ref{thm:linear_A}} (\adim,p)$. Since
	$\norm{v_h}{1}{a,r-h} \leq \norm{\weakD v}{1}{a,r}$ and
	$\dnorm{S_h}{p}{a,r-h} \leq \norm{f}{p}{a,r}$, taking the limit $h \to
	0+$ one infers that $v$, hence $u$, is twice weakly differentiable and
	satisfies the desired estimate, using Simon's absorption lemma
	\cite[p.~398]{MR1459795} as before.
\end{proof}
\begin{remark}
	In general, even if $\Lip u \leq L < \infty$ and $P=0$ the condition
	involving $\varepsilon$ cannot be replaced by some uniform strong
	ellipticity condition on $D^2 F ( \sigma )$ for $\sigma \in \Hom
	(\rel^\vdim, \rel^\codim)$ with $\| \sigma \| \leq L$ as may be seen
	from the example of Lawson and Osserman in \cite[Theorem
	7.1]{MR452745}.
\end{remark}
\begin{lemma} \label{lemma:solutions}
	Suppose $\adim \in \nat$, and $1 < q \leq p < \infty$.

	Then there exists a positive, finite number $\varepsilon$ with the
	following property.

	If $\adim > \vdim \in \nat$, $F$ is related to $\varepsilon$ as in
	\ref{miniremark:situation_F}, $\Lip D^2 F < \infty$, $a \in
	\rel^\vdim$, $0 < r < \infty$, and $u_i \in \Sob{}{1}{q} (
	\oball{a}{r}, \rel^\codim )$ with $i \in \{1,2\}$ satisfy $L_F (u_i) =
	0$, then $u_i$ are twice weakly differentiable and for every affine
	function $P : \rel^\vdim \to \rel^\codim$ there holds
	\begin{multline*}
		r^{-\vdim/p+1} \norm{\weakD^2 ( u_2 - u_1 )}{p}{a,r/2} \leq
		\Gamma \big ( r^{-\vdim-1} \norm{u_2 - u_1}{1}{a,r} \\
		+ ( r^{-\vdim-1} \norm{u_1 - P}{1}{a,r} ) \Lip ( D^2 F ) (
		r^{-\vdim-1} \norm{u_2 - u_1}{1}{a,r} ) \big )
	\end{multline*}
	where $\Gamma$ is a positive, finite number depending only on
	$\adim$ and $p$.
\end{lemma}
\begin{proof}
	Using an elementary covering argument, it is enough to prove the
	assertion with $\norm{\weakD^2 ( u_2 - u_1 )}{p}{a,r/2}$ replaced by
	$\norm{\weakD^2 ( u_2 - u_1 )}{p}{a,r/4}$. For this purpose let
	$\kappa = 2^{1/2} \adim^2$
	\begin{gather*}
		\varepsilon = \inf \{
		\varepsilon_{\ref{lemma:higher_differentiability}} ( \adim,
		 q, 2p ),  \varepsilon_{\ref{thm:linear_B}} ( \adim,
		p ) / \kappa, \varepsilon_{\ref{thm:linear_A}} (
		\adim, q, 2p ) \}, \quad
		\Delta_1 = \Gamma_{\ref{lemma:higher_differentiability}}
		( \adim, 2p ), \\
		\Delta_2 = \Gamma_{\ref{thm:linear_B}} ( \adim, p ), \quad
		\Delta_3 = \Gamma_{\ref{thm:linear_A}} ( \adim, 2p), \quad
		\Gamma = \Delta_2 \, \sup \{ 2^{1 + \adim}, \kappa \Delta_1
		\Delta_3 \}.
	\end{gather*}
	Suppose $F$, $a$, $r$, and $u_i$ satisfy the hypotheses with
	$\varepsilon$ and that $P : \rel^\vdim \to \rel^\codim$ is an affine
	function. In order to show that they satisfy the modified conclusions
	with $\Gamma$, it will be assumed $a=0$ and $r=1$. Abbreviate $\Lambda
	= \Lip D^2 F$.

	By \ref{lemma:higher_differentiability} the functions $u_i$ are
	twice weakly differentiable with
	\begin{gather*}
		\norm{\weakD^2 u_i}{2p}{0,1/2} \leq \Delta_1 \norm{u_i -
		P}{1}{0,1} \quad \text{for $i \in \{1,2\}$}
	\end{gather*}
	and one obtains from \ref{miniremark:situation_F} for
	$\mathscr{L}^\vdim$ almost all $x \in \oball{0}{1}$
	\begin{gather*}
		\left < \weakD^2 u_i (x), C_F ( \weakD u_i (x) ) \right
		> = 0 \quad \text{for $i \in \{1,2\}$}, \\
		\left < \weakD^2 ( u_2 - u_1 ) (x), C_F ( \weakD u_2 (x)
		) \right > = \left < \weakD^2 u_1 (x), C_F ( \weakD u_1
		(x) ) - C_F ( \weakD u_2 (x) ) \right >.
	\end{gather*}
	Therefore by \ref{thm:linear_B}, \ref{miniremark:situation_F} and
	H\"older's inequality
	\begin{multline*}
		\boldsymbol{|} \weakD^2 ( u_2 - u_1 ) \boldsymbol{|}_{a,1/4;p}
		\leq \Delta_2 \big ( 2^{2+m-m/p} \boldsymbol{|} u_2 - u_1
		\boldsymbol{|}_{0,1/2;1} \\
		+ \kappa \Lambda \norm{\weakD^2 u_1}{2p}{0,1/2} \norm{\weakD
		(u_2-u_1)}{2p}{0,1/2} \big).
	\end{multline*}
	To estimate $\norm{\weakD (u_2-u_1)}{2p}{0,1/2}$, one computes for
	$\theta \in \mathscr{D} ( \oball{0}{1}, \rel^\codim )$
	\begin{gather*}
		- {\textstyle\int_{\oball{0}{1}}} \left < D \theta (x) \odot
		\weakD ( u_2 - u_1 ) (x), A (x) \right > \ud
		\mathscr{L}^\vdim x = 0, \\
		\text{where $A(x) = {\textstyle\int_0^1} D^2 F ( t \weakD u_2
		(x) + ( 1-t ) \weakD u_1 (x) ) \ud \mathscr{L}^1 t$},
	\end{gather*}
	and obtains from \ref{thm:linear_A}
	\begin{gather*}
		\norm{\weakD ( u_2 - u_1 )}{2p}{0,1/2} \leq \Delta_3
		\norm{u_2 - u_1}{1}{0,1}
	\end{gather*}
	and the conclusion follows.
\end{proof}
\begin{lemma} \label{lemma:main_lemma}
	Suppose $\vdim, \adim \in \nat$, $\vdim < \adim$, $1 \leq p \leq r <
	\infty$, and $1 < q < \infty$.

	Then there exist a positive, finite number $\varepsilon$, a positive,
	finite number $\Gamma_1$ depending only on $\vdim$ and $p$, and a
	positive, finite number $\Gamma_2$ depending only on $\vdim$, $\adim$,
	$p$, and $r$ with the following property.

	If $F$ is related to $\varepsilon$ as in \ref{miniremark:situation_F},
	$\Lip D^2 F < \infty$, $j \in \{0,1\}$, $A$ is a closed subset of
	$\rel^\vdim$, $u : \rel^\vdim \cap \{ x \with \dist (x,A) < 1 \} \to
	\rel^\codim$ is $j$ times weakly differentiable, $0 \leq \gamma <
	\infty$, and if for each $a \in A$, $0 < \varrho \leq 1$ there are
	$v_{a,\varrho} \in \Sob{}{1}{q} ( \oball{a}{\varrho}, \rel^\codim )$
	and an affine function $P_{a,\varrho} : \rel^\vdim \to \rel^\codim$
	such that
	\begin{gather*}
		L_F ( v_{a,\varrho} ) = 0, \\
		{\textstyle\sum_{i=0}^j} \varrho^{-\vdim/p+i} \norm{\weakD^i (
		u - v_{a,\varrho} )}{p}{a,\varrho} \leq \gamma \varrho^2,
		\quad \varrho^{-\vdim/p} \norm{u -
		P_{a,\varrho}}{p}{a,\varrho} \leq \gamma \varrho
	\end{gather*}
	then there exists a twice weakly differentiable function $v :
	\rel^\vdim \cap \{ x \with \dist (x,A) < \frac{1}{36} \} \to
	\rel^\codim$ with
	\begin{gather*}
		{\textstyle\sum_{i=0}^j} \varrho^{-\vdim/p+i} \norm{\weakD^i (
		u-v)}{p}{a,\varrho} \leq \Gamma_1 \gamma \varrho^2, \\
		\varrho^{-\vdim/r} \norm{\weakD^2 v}{r}{a,\varrho} \leq
		\Gamma_2 \big ( \gamma ( 1 + \Lip ( D^2 F ) \gamma )^2 +
		\varrho^{-\vdim-2} \norm{u - P_{a,2\varrho}}{1}{a,2\varrho}
		\big )
	\end{gather*}
	whenever $a \in A$, $0 < \varrho \leq \frac{1}{36}$.
\end{lemma}
\begin{proof}
	Assume $r \geq q$ and define
	\begin{gather*}
		\varepsilon = \min \{ 1,
		\varepsilon_{\ref{lemma:higher_differentiability}} ( \adim,
		q, 2r ), \varepsilon_{\ref{lemma:solutions}} ( \adim, q, 2r
		), \varepsilon_{\ref{lemma:higher_differentiability}} (
		\adim, q, r ) \}.
	\end{gather*}
	Suppose $F$, $j$, $A$, $u$, $\gamma$, $v_{a,\varrho}$, and
	$P_{a,\varrho}$ are as in the hypotheses in the body of the lemma with
	$\varepsilon$ and abbreviate $\Lambda = \Lip D^2 F$.

	By \ref{lemma:higher_differentiability} and H\"older's inequality
	\begin{gather*}
		{\textstyle\sum_{i=0}^j} \norm{\weakD^i
		v_{a,\varrho}}{2r}{a,1/2} < \infty, \quad
		{\textstyle\sum_{i=0}^j} \norm{\weakD^i u}{p}{a,1/2} <
		\infty
	\end{gather*}
	whenever $a \in A$. Therefore taking limits (for example by use of an
	interpolation inequality similar to \cite[Lemma 6.2.2]{MR0202511} and
	weak compactness properties of Sobolev spaces \cite[Theorem
	3.2.4(e)]{MR0202511}) the conclusion can be deduced from the following
	assertion: \emph{There exist a positive, finite number $\Gamma_1$
	depending only on $\vdim$ and $p$, and a positive, finite number
	$\Gamma_2$ depending only on $\vdim$, $\adim$, $p$ and $r$ such that
	for every $0 < \delta \leq \frac{1}{18}$ there exists a function $v :
	\rel^\vdim \to \rel^\codim$ whose restriction to
	$\classification{\rel^\vdim}{ x}{ \dist (x,A) < \frac{1}{18} }$ is
	twice weakly differentiable satisfying
	\begin{gather*}
		{\textstyle\sum_{i=0}^j} \varrho^{-\vdim/p+i} \norm{\weakD^i (
		u - v )}{p}{a,\varrho} \leq \Gamma_1 \gamma \varrho^2, \\
		(\varrho/2)^{-\vdim/r} \norm{\weakD^2 v}{r}{a,\varrho/2} \leq
		\Gamma_2 \big ( \gamma ( 1 + \Lambda \gamma )^2 +
		(\varrho/2)^{-\vdim-2} \norm{u - P_{a,\varrho}}{1}{a,\varrho}
		\big )
	\end{gather*}
	whenever $a \in A$, $\delta \leq \varrho \leq \frac{1}{18}$.}

	Assume $A \neq \emptyset$, let $\Phi = \{ \rel^\vdim \without A \}
	\cup \{ \oball{a}{\delta} \with a \in A \}$, note $\bigcup \Phi =
	\rel^\vdim$, define $h : \rel^\vdim \to \rel$ by
	\begin{gather*}
		h(x) = {\textstyle\frac{1}{20}} \sup \{ \min \{ 1, \dist (x,
		\rel^\vdim \without U ) \} \with U \in \Phi \} \quad
		\text{for $x \in \rel^\vdim$},
	\end{gather*}
	and apply \cite[3.1.13]{MR41:1976} to obtain a countable subset $S$ of
	$\rel^\vdim$ and functions $\varphi_s : \rel^\vdim \to \{ t \with
	0 \leq t \leq 1 \}$ of class $\class{\infty}$ corresponding to $s \in
	S$ such that with $S_x = \classification{S}{ s}{ \cball{x}{10h(x)}
	\cap \cball{s}{10h(s)} \neq \emptyset }$ for $x \in \rel^\vdim$ and
	a sequence $V_i$ of positive, finite numbers depending only on $\vdim$
	there holds
	\begin{gather*}
		\card S_x \leq ( 129 )^\vdim, \quad \spt \varphi_s \subset
		\cball{s}{10h(s)} \quad \text{for $s \in S$}, \\
		1/3 \leq h(x) / h(s) \leq 3 \quad \text{for $s \in S_x$},
		\quad | D^i \varphi_s (x) | \leq V_i (h(x))^{-i} \quad
		\text{for $s \in S$, $i \in \nat$}, \\
		\sum_{s \in S} \varphi_s (y) = \sum_{s \in S_x} \varphi_s (y)
		= 1, \quad \sum_{s \in S} D^i \varphi_s (y) = \sum_{s \in S_x}
		D^i \varphi_s (y) = 0 \quad \text{for $i \in \nat$}
	\end{gather*}
	whenever $x \in \rel^\vdim$, $y \in \cball{x}{10h(x)}$. Note for
	$x \in \rel^\vdim$, $y \in \cball{x}{10h(x)}$, $s \in S$, $i \in
	\nat$
	\begin{gather*}
		| D^i \varphi_s (y) | \leq V_i ( h(y) )^{-i} \leq ( 20 )^i V_i
		( 10 h(x) )^{-i},
	\end{gather*}
	because $h(x)-h(y) \leq \frac{1}{20} |x-y| \leq \frac{1}{2} h(x)$.
	Choose $\xi : S \to A$ such that
	\begin{gather*}
		| \xi (s) - s | = \dist (s,A) \quad \text{whenever $s \in S$}.
	\end{gather*}
	Note $20 h(x) \leq \max \{ \dist (x,A), \delta \}$ for $x \in
	\rel^\vdim$ and observe
	\begin{gather*}
		\cball{x}{20h(x)} \subset \cball{\xi(s)}{120h(s)}, \quad
		120 h(s) \leq 1
	\end{gather*}
	whenever $x \in \rel^\vdim$, $\dist (x,A) \leq \frac{1}{18}$, $s \in
	S_x$, because
	\begin{gather*}
		| x-s | \leq 10 h(x) + 10 h(s) \leq 40 h(x)
		\leq 2 \max \{ \dist (x,A), \delta \} \leq 1/9, \\
		| s - \xi (s) | = \dist (s,A) \leq |x-s| + \dist (x,A) \leq
		1/6, \\
		| x - \xi (s) | \leq | x-s | + | s - \xi(s) | \leq 40 h(s) +
		20 h(s) = 60 h(s), \\
		| x - \xi (s) | + 20 h(x) \leq 120 h(s) \leq 360 h(x) \leq 1.
	\end{gather*}
	Define $R = \bigcup \{ S_x \with \text{$x \in \rel^\vdim$ and
	$\dist (x,A) \leq \frac{1}{18}$} \}$,
	\begin{gather*}
		v_s = v_{\xi(s),120h(s)} \quad \text{and} \quad P_s =
		P_{\xi(s),120h(s)} \quad \text{for $s \in R$}
	\end{gather*}
	and, denoting by $v_s'$ the extension of $v_s$ to $\rel^\vdim$ by
	$0$, $v : \rel^\vdim \to \rel^\codim$ by
	\begin{gather*}
		v(x) = \sum_{s \in R} \varphi_s (x) v_s' (x) \quad
		\text{whenever $x \in \rel^\vdim$}.
	\end{gather*}

	Suppose for the rest of the proof $x \in \rel^\vdim$ with $\dist
	(x,A) \leq \frac{1}{18}$ and observe
	\begin{gather*}
		v(y) = \sum_{s \in S_x} \varphi_s (y) v_s (y) \quad
		\text{whenever $y \in \cball{x}{10h(x)}$}.
	\end{gather*}
	The asserted weak differentiability is a consequence of
	\ref{lemma:higher_differentiability}.

	One estimates
	\begin{gather*}
		\begin{aligned}
			& \norm{\weakD^i ( u - v_s )}{p}{x,20h(x)} \leq
			\norm{\weakD^i ( u - v_s )}{p}{s,120h(s)} \\
			& \qquad \leq \gamma (120h(s))^{\vdim/p+2-i} \leq
			(18)^{\vdim/p+2} \gamma (20h(x))^{\vdim/p+2-i}
		\end{aligned}
	\end{gather*}
	for $i \in \{0,j\}$, $s \in S_x$, hence by H\"older's inequality
	\begin{gather}
		\begin{aligned}
			& \phantom{\leq} \ (20h(x))^{-\vdim} \norm{u -
			v_s}{1}{x,20h(x)} \\
			& \leq \unitmeasure{\vdim}^{1-1/p}
			{\textstyle\sum_{i=0}^j} ( 20 h(x))^{-\vdim/p+i}
			\norm{\weakD^i (u-v_s)}{p}{x,20h(x)} \leq 2\Delta_1 \gamma
			(20h(x))^2
		\end{aligned} \label{eqn:uvs}
	\end{gather}
	for $s \in S_x$ where $\Delta_1 = \unitmeasure{\vdim}^{1-1/p}
	(18)^{\vdim/p+2}$. Also
	\begin{gather}
		\begin{aligned}
			(20h(x))^{-\vdim} \norm{u - P_s}{1}{x,20h(x)} & \leq
			\unitmeasure{\vdim}^{1-1/p} ( 20 h(x) )^{-\vdim/p}
			\norm{u - P_s}{p}{\xi(s),120h(s)} \\
			& \leq \Delta_1 \gamma ( 20 h(x) ),
		\end{aligned} \notag \\
		(20 h(x) )^{-\vdim} \norm{v_s - P_s}{1}{x,20h(x)} \leq 3
		\Delta_1 \gamma ( 20 h(x)) \label{eqn:vsws}
	\end{gather}
	for $s \in S_x$. Using
	\begin{gather*}
		v(y) -  u(y) = \sum_{s \in S_x} \varphi_s (y) ( v_s (y) - u
		(y) ) \quad \text{whenever $y \in \cball{x}{10h(x)}$}
	\end{gather*}
	and the Leibnitz formula, one obtains from \eqref{eqn:uvs}
	\begin{gather*}
		{\textstyle\sum_{i=0}^j} ( 10h(x) )^{-\vdim/p+i}
		\norm{\weakD^i ( u-v )}{p}{x,10h(x)} \leq \Delta_2 \gamma (
		10h(x))^2
	\end{gather*}
	where $\Delta_2 = \unitmeasure{\vdim}^{1/p-1} 8 \Delta_1 2^{\vdim/p} ( 1 + 20V_1
	) (129)^\vdim$.

	In case $x \in \cball{a}{\varrho}$ for some $a \in A$, $\delta
	\leq \varrho \leq \frac{1}{18}$,
	\begin{gather*}
		20 h(x) \leq \max \{ \dist (x,A), \delta \} \leq \varrho,
		\quad \cball{x}{20h(x)} \subset \cball{a}{2\varrho}
	\end{gather*}
	and Vitali's covering theorem yields a countable subset $T$ of
	$\cball{a}{\varrho}$ such that
	\begin{gather*}
		\text{$\{ \cball{t}{2h(t)} \with t \in T \}$ is disjointed},
		\quad \cball{a}{\varrho} \subset \bigcup \{
		\cball{t}{10h(t)} \with t \in T \}
	\end{gather*}
	and one estimates for $i \in \{0,j\}$
	\begin{gather*}
		\begin{aligned}
			& \norm{\weakD^i ( u - v)}{p}{a,\varrho}^p \\
			& \qquad \leq {\textstyle\sum_{t \in T}}
			\norm{\weakD^i ( u - v )}{p}{t,10h(t)}^p \\
			& \qquad \leq ( \Delta_2 \gamma )^p {\textstyle\sum_{t
			\in T}} ( 10 h(t) )^{\vdim+(2-i)p} \\
			& \qquad = ( 5^{\vdim/p+2-i} \Delta_2 \gamma )^p
			\unitmeasure{\vdim}^{-1-(2-i)p/\vdim}
			{\textstyle\sum_{t \in T}} \mathscr{L}^\vdim (
			\cball{t}{2h(t)} )^{1+(2-i)p/\vdim} \\
			& \qquad \leq ( 5^{\vdim/p+2-i} \Delta_2 \gamma )^p
			\unitmeasure{\vdim}^{-1-(2-i)p/\vdim}
			\mathscr{L}^\vdim (
			\cball{a}{2\varrho})^{1+(2-i)p/\vdim} \\
			& \qquad = \big ( (10)^{\vdim/p+2-i} \Delta_2 \gamma
			\big )^p \varrho^{\vdim+(2-i)p}.
		\end{aligned}
	\end{gather*}
	Therefore one obtains for $a \in A$, $\delta \leq \varrho \leq
	\frac{1}{18}$, $i \in \{0,j\}$
	\begin{gather} \label{eqn:u_v}
		\varrho^{-\vdim/p+i} \norm{\weakD^i ( u-v )}{p}{a,\varrho}
		\leq (10)^{\vdim/p+2} \Delta_2 \gamma \varrho^2
	\end{gather}
	and one may take $\Gamma_1 =2 (10)^{\vdim/p+2} \Delta_2$ in the first
	estimate of the assertion.

	According to \ref{lemma:higher_differentiability} the functions $v_s$
	are twice weakly differentiable and satisfy for $s \in S_x$
	\begin{gather*}
		( 20 h(x) )^{-\vdim/(2r)+2} \norm{\weakD^2 v_s}{2r}{x,10h(x)}
		\leq \Delta_3 (20h(x))^{-\vdim} \norm{v_s - P_s}{1}{x,20h(x)}
	\end{gather*}
	where $\Delta_3 = \Gamma_{\ref{lemma:higher_differentiability}} (
	\adim, 2r )$. Combining this with \eqref{eqn:vsws} yields
	\begin{gather} \label{eqn:d2vs}
		( 10 h(x) )^{-\vdim/(2r)+2} \norm{\weakD^2 v_s}{2r}{x,10h(x)}
		\leq 2^{\vdim/(2r)} 3 \Delta_1 \Delta_3 \gamma ( 10h(x))
	\end{gather}
	for $s \in S_x$.

	Using \ref{lemma:solutions}, one obtains for $s,t \in S_x$
	\begin{multline*}
		(20h(x))^{-\vdim/(2r)+1} \norm{\weakD^2 ( v_s -
		v_t)}{2r}{x,10h(x)} \leq \Delta_4 \big ( ( 20h(x))^{-\vdim-1}
		\norm{v_s - v_t}{1}{x,20h(x)} \\
		+ \Lambda ( ( 20h(x) )^{-\vdim-1} \norm{v_s -
		P_s}{1}{x,20h(x)} ) ( ( 20h(x) )^{-\vdim-1} \norm{v_s -
		v_t}{1}{x,20h(x)} ) \big )
	\end{multline*}
	where $\Delta_4 = \Gamma_{\ref{lemma:solutions}} ( \adim,
	2r)$. Since
	\begin{gather*}
		(20h(x))^{-\vdim} \norm{v_s - v_t}{1}{x,20h(x)} \leq 4\Delta_1
		\gamma ( 20h(x) )^2
	\end{gather*}
	by \eqref{eqn:uvs}, one estimates using \eqref{eqn:vsws}
	\begin{gather*}
		(10h(x))^{-\vdim/(2r)} \norm{\weakD^2 ( v_s - v_t
		)}{2r}{x,10h(x)} \leq \Delta_5 \gamma ( 1 + \Lambda \gamma )
	\end{gather*}
	where $\Delta_5 = 2^{\vdim+2} \Delta_1 \Delta_4 \sup \{ 3 \Delta_1, 1
	\}$. Using an interpolation inequality (which may be proven similarly
	to \cite[Lemma 6.2.2]{MR0202511}), one infers with a positive, finite
	number $\Delta_6$ depending only $\adim$ and $r$
	\begin{align*}
		& \tsum{i=0}{2} (10h(x))^{-\vdim/(2r)+i} \norm{\weakD^i ( v_s
		- v_t)}{2r}{x,10h(x)} \\
		& \qquad \leq
		\begin{aligned}[t]
			\Delta_6 \big ( & ( 10 h(x))^{-\vdim/(2r)+2}
			\norm{\weakD^2 ( v_s - v_t)}{2r}{x,10h(x)} \\
			& + ( 10h(x))^{-\vdim} \norm{v_s - v_t}{1}{x,10h(x)}
			\big)
		\end{aligned} \\
		& \qquad \leq \Delta_6 \big ( \Delta_5 ( 1 + \Lambda \gamma ) +
		2^{\vdim+4} \Delta_1 \big ) \gamma ( 10h(x) )^2.
	\end{align*}
	This implies for $s,t \in S_x$
	\begin{gather*}
		{\textstyle\sum_{i=0}^2} ( 10h(x) )^{-\vdim/(2r)+i}
		\norm{\weakD^i ( v_s - v_t )}{2r}{x,10h(x)} \leq \Delta_7 \gamma (
		1 + \Lambda \gamma ) ( 10 h(x) )^2
	\end{gather*}
	where $\Delta_7 = \Delta_6 ( \Delta_5 + 2^{\vdim+4} \Delta_1 )$. Noting $( v - v_s )
	(y) = \sum_{t \in S_x} \varphi_t (y) ( v_t - v_s ) (y)$ for $s \in
	S_x$, $y \in \oball{x}{10h(x)}$, one infers using the Leibnitz
	formula
	\begin{gather} \label{eqn:divvs}
		(10h(x))^{-\vdim/(2r)+i} \norm{\weakD^i ( v - v_s
		)}{2r}{x,10h(x)} \leq \Delta_8 \gamma ( 1 + \Lambda
		\gamma ) ( 10 h(x) )^2
	\end{gather}
	for $s \in S_x$, $i \in \{0,1,2\}$ where $\Delta_8 = 2 ( 1 + 20 V_1 + 400
	V_2 ) \Delta_7 ( 129)^\vdim$.

	Using \ref{miniremark:situation_F}, one defines
	\begin{gather*}
		f(y) = \left < \weakD^2 v (y) , C_F ( \weakD v (y) ) \right >
	\end{gather*}
	whenever $y \in \oball{z}{10h(z)}$ for some $z \in \rel^\vdim$
	with $\dist (z,A) \leq \frac{1}{18}$ and computes for $s \in S_x$
	\begin{gather*}
		f (y) = \left < \weakD^2 v_s (y), C_F ( \weakD v (y) ) -
		C_F ( \weakD v_s (y) ) \right > + \left < \weakD^2 ( v -
		v_s ) (y), C_F ( \weakD v(y) ) \right >
	\end{gather*}
	for $\mathscr{L}^\vdim$ almost all $y \in \oball{x}{10h(x)}$.
	H\"older's inequality implies
	\begin{align*}
		\norm{f}{r}{x,10h(x)} & \leq \kappa \Lambda
		\norm{\weakD (v-v_s)}{2r}{x,10h(x)} \norm{\weakD^2
		v_s}{2r}{x,10h(x)} \\
		& \phantom{\leq} + 2 \kappa
		\unitmeasure{\vdim}^{1/(2r)} ( 10 h(x))^{\vdim/(2r)}
		\norm{\weakD^2 ( v-v_s)}{2r}{x,10h(x)} ,
	\end{align*}
	hence by \eqref{eqn:d2vs} and \eqref{eqn:divvs}
	\begin{gather*}
		(10h(x))^{-\vdim/r} \norm{f}{r}{x,10h(x)} \leq \Delta_9 \gamma
		( 1 + \Lambda \gamma )^2
	\end{gather*}
	where $\Delta_9 = \kappa \Delta_8 \sup \big \{ 2^{\vdim/(2r)} 3
	\Delta_1 \Delta_3, 2 \unitmeasure{\vdim}^{1/(2r)} \big \}$. Similarly
	but simpler as in the deduction of \eqref{eqn:u_v}, one obtains for
	$\delta \leq \varrho \leq \frac{1}{18}$, $a \in A$
	\begin{gather*}
		\norm{f}{r}{a,\varrho} \leq \Delta_9 (10)^{\vdim/r} \gamma ( 1 +
		\Lambda \gamma )^2 \varrho^{\vdim/r}
	\end{gather*}
	and thus, using \ref{lemma:higher_differentiability} with $\Delta_{10}
	= \Gamma_{\ref{lemma:higher_differentiability}} (\adim,r)$ and
	\eqref{eqn:u_v},
	\begin{gather*}
		\begin{aligned}
			\varrho^{-\vdim/r} \norm{\weakD^2 v}{r}{a,\varrho/2}
			& \leq \Delta_{10} \big ( \varrho^{-\vdim-2} (
			\norm{u-v}{1}{a,\varrho} + \norm{u -
			P_{a,\varrho}}{1}{a,\varrho} ) + \varrho^{-\vdim/r}
			\norm{f}{r}{a,\varrho} \big ) \\
			& \leq \Delta_{11} \big ( \gamma ( 1 + \Lambda
			\gamma)^2 + \varrho^{-\vdim-2} \norm{u -
			P_{a,\varrho}}{1}{a,\varrho} \big )
		\end{aligned}
	\end{gather*}
	where $\Delta_{11} = \Delta_{10} ( \unitmeasure{\vdim}^{1-1/p}
	(10)^{\vdim/p+2} \Delta_2 + \Delta_9 (10)^{\vdim/r} + 1 )$. Therefore
	one may take $\Gamma_2 = 2^{\vdim/r} \Delta_{11}$ in the second
	estimate of the assertion and the proof is completed.
\end{proof}
\begin{remark} \label{remark:diff_lemma}
	In fact, by Calder\'on and Zygmund \cite[Theorem 10\,(ii)]{MR0136849}
	(see also \cite[Lemma 3.7.2]{MR1014685}) or by
	\cite[3.1]{snulmenn.isoperimetric}
	\begin{gather*}
		\lim_{\varrho \pluslim{0}} \varrho^{-2}
		{\textstyle\sum_{i=0}^j} \varrho^{-\vdim/p+i} \norm{\weakD^i
		(u-v)}{p}{a,\varrho} = 0
	\end{gather*}
	for $\mathscr{L}^\vdim$ almost all $a \in A$. Now, Re{\v s}etnyak's
	result in \cite{Reshetnyak_diff} applied to $v$ yields that for
	$\mathscr{L}^\vdim$ almost all $a \in A$ there exists a polynomial
	function $Q_a : \rel^\vdim \to \rel^\codim$ of degree at most $2$ such
	that
	\begin{gather*}
		\limsup_{\varrho \pluslim{0}} \varrho^{-2}
		{\textstyle\sum_{i=0}^j} \varrho^{-\vdim/p+i} \norm{\weakD^i (
		u - Q_a )}{p}{a,\varrho} = 0.
	\end{gather*}
	Alternately, this latter fact could have also been deduced by use of
	Calder\'on and Zygmund \cite[Theorem 12]{MR0136849} (see also
	\cite[Theorem 3.4.2]{MR1014685}).
\end{remark}
\begin{theorem} \label{thm:criterion}
	Suppose $\vdim, \adim \in \nat$, $\vdim < \adim$, $1 \leq p < \infty$,
	and $1 < q < \infty$.

	Then there exists a positive, finite number $\varepsilon$ with the
	following property.

	If $F$ is related to $\varepsilon$ as in \ref{miniremark:situation_F},
	$\Lip D^2 F < \infty$, $U$ is an open subset of $\rel^\vdim$, $j \in
	\{0,1\}$, $u : U \to \rel^\codim$ is weakly differentiable,
	\begin{gather*}
		h(a,r) = \\
		\inf \left \{ 
		{\textstyle\sum_{i=0}^j} r^{-\vdim/p+i} \norm{\weakD^i ( u - v
		)}{p}{a,r} \with \text{$v \in \Sob{}{1}{q} ( \oball{a}{r} ,
		\rel^\codim )$ and $L_F ( v ) = 0$} \right \}
	\end{gather*}
	whenever $\oball{a}{r} \subset U$ for some $a \in \rel^\vdim$, $0 <
	r < \infty$, and if $A$ denotes the set of all $a \in U$ such that
	\begin{gather*}
		\limsup_{r \pluslim{0}} r^{-2} h (a,r) < \infty,
	\end{gather*}
	then $A$ is a Borel set and for $\mathscr{L}^\vdim$ almost all $a \in
	A$ there exists a polynomial function $Q_a : \rel^\vdim \to
	\rel^\codim$ with degree at most $2$ such that
	\begin{gather*}
		\lim_{r \pluslim{0}} r^{-2} {\textstyle\sum_{i=0}^j}
		r^{-\vdim/p+i} \norm{\weakD^i ( u - Q_a )}{p}{a,r} = 0.
	\end{gather*}
\end{theorem}
\begin{proof}
	In view of \ref{lemma:higher_differentiability} one may assume $q
	\geq p$. Let $\varepsilon = \varepsilon_{\ref{lemma:main_lemma}} (
	\vdim, \adim, p,p,q ).$ Suppose $F$, $U$, $j$, and $u$ satisfy the
	hypotheses with $\varepsilon$. Define the open set $V$ by
	\begin{gather*}
		V = \bigclassification{U}{x}{\text{${\textstyle\sum_{i=0}^j}
		\norm{\weakD^i u}{p}{x,r} < \infty$ for some $0 < r < \dist (
		x, \rel^\vdim \without U )$}}
	\end{gather*}
	and note $A \subset V$. Denote by $D$ the set of all $v \in
	\Sob{}{1}{q} (\oball{0}{1}, \rel^\codim )$ such that $L_F
	(v)=0$ and define
	\begin{gather*}
		W = \eqclassification{V \times \rel}{(a,r)}{0 < r < \dist (
		a, \rel^\vdim \without V )}
	\end{gather*}
	and the continuous map $T : W \to \Sob{}{1}{p} ( \oball{0}{1},
	\rel^\codim )$ by
	\begin{gather*}
		T (a,r) (x) = r^{-1} u ( a + r x ) \quad
		\text{whenever $(a,r) \in W$, $x \in \oball{0}{1}$}.
	\end{gather*}
	Since $D \neq \emptyset$ and
	\begin{gather*}
		h(a,r) = r \inf \big \{ {\textstyle\sum_{i=0}^j}
		\norm{\weakD^i ( T(a,r) - v )}{p}{0,1} \with v \in D \big \}
		\quad \text{for $(a,r) \in W$},
	\end{gather*}
	$h$ is continuous. Therefore $A$ is a Borel set. Similarly, denoting
	by $D'$ the set of all affine functions mapping $\rel^\vdim$ into
	$\rel^\codim$ one defines a continuous map $h' : W \to \rel$ by
	\begin{gather*}
		h' (a,r) = r \inf \{ \norm{T (a,r) - w}{1}{0,1} \with w \in
		D' \} \quad \text{for $(a,r) \in W$}.
	\end{gather*}
	By Re{\v s}etnyak \cite{Reshetnyak_diff} or
	\cite[4.5.9\,(26)\,(\printRoman{2})\,(\printRoman{3})]{MR41:1976} one
	notes
	\begin{gather*}
		\limsup_{\varrho \pluslim{0}} \varrho^{-1} h' ( a, \varrho ) <
		\infty \quad \text{for $\mathscr{L}^\vdim$ almost all $a \in
		U$}.
	\end{gather*}

	Define
	\begin{gather*}
		C_k = \classification{V}{x}{\dist (x, \rel^\vdim \without V )
		\geq 1/k}, \\
		A_k = \classification{C_k}{a}{\text{$h (a,r) \leq k r^2$ and
		$h' (a,r) \leq k r$ for $0 < r < 1/k$}}
	\end{gather*}
	for $k \in \nat$ and observe that the sets $A_k$ are closed and
	\begin{gather*}
		\mathscr{L}^\vdim ( A \without {\textstyle\bigcup} \{ A_k
		\with k \in \nat \} ) =0.
	\end{gather*}
	Finally, the conclusion is obtained by applying (for each $k \in
	\nat$) \ref{lemma:main_lemma} in conjunction with
	\ref{remark:diff_lemma} to rescaled versions of $u$, $A_k$ and a
	suitable number $\gamma$.
\end{proof}
\begin{remark}
	Instead of using Re{\v s}etnyak \cite{Reshetnyak_diff} or
	\cite[4.5.9\,(26)\,(\printRoman{2})\,(\printRoman{3})]{MR41:1976}, one
	can also use the functions $v$ occurring in the definition of $h(a,r)$
	in a way reminiscent of the familiar harmonic approximation procedure
	to deduce
	\begin{gather*}
		\limsup_{\varrho \pluslim{0}} \varrho^{-1} h'(a,\varrho) <
		\infty \quad \text{whenever $a \in A$}.
	\end{gather*}
	Therefore $u$ could have been required to be merely $j$ times weakly
	differentiable.
\end{remark}
\begin{corollary} \label{corollary:lebesgue_points_p}
	Suppose $\vdim, \adim \in \nat$, $\vdim < \adim$, $1 < p < \infty$,
	$U$ is an open subset of $\rel^\vdim$, $T \in \mathscr{D} ( U,
	\rel^\codim )$ and $A$ denotes the set of all $a \in U$ such that
	\begin{gather*}
		\limsup_{r \to 0+} r^{-1-\vdim/p} \dnorm{T}{p}{a,r} < \infty.
	\end{gather*}

	Then $A$ is a Borel set and for $\mathscr{L}^\vdim$ almost all $a \in
	A$ there exists a unique constant distribution $T_a \in \mathscr{D}' (
	U , \rel^\codim )$ such that
	\begin{gather*}
		\lim_{r \to 0+} r^{-1-\vdim/p} \dnorm{T-T_a}{p}{a,r} = 0.
	\end{gather*}
\end{corollary}
\begin{proof}
	The conclusion is local and for each $a \in A$ there exists $0 < r <
	\infty$ with $\dnorm{T}{p}{a,r} < \infty$, hence one may assume $\spt
	T$ to be compact, $U = \rel^\vdim$ and $\dnorm{T}{p}{0,R} < \infty$,
	$\spt T \subset \oball{0}{R}$ for some $0 < R < \infty$.

	For example using \ref{thm:global_w1p_estimate}, one obtains functions
	$u \in \Sob{0}{1}{p} ( \oball{0}{R}, \rel^\codim )$ and $v_{a,r} \in
	\mathscr{E} ( \oball{a}{r}, \rel^\codim )$ whenever $a \in
	\rel^\vdim$, $0 < r < \infty$ and $\oball{a}{r} \subset \oball{0}{R}$
	such that
	\begin{gather*}
		- \tint{\oball{0}{R}}{} \weakD u \bullet D \theta
		\ud\mathscr{L}^\vdim = T ( \theta ) \quad \text{for $\theta
		\in \mathscr{D} ( \oball{0}{R}, \rel^\codim )$}, \\
		u-v_{a,r} \in \Sob{0}{1}{p} ( \oball{a}{r} , \rel^\codim ),
		\quad \Lap v_{a,r} = 0.
	\end{gather*}
	By \ref{thm:global_w1p_estimate} and Poincar\'e's inequality
	\begin{gather*}
		\tsum{i=0}{1} r^{i-1} \norm{\weakD^i ( u-v_{a,r} )}{p}{a,r}
		\leq \Delta \dnorm{T}{p}{a,r}
	\end{gather*}
	for some positive, finite number $\Delta$ depending only on $\adim$
	and $p$, hence the set $A$ agrees with the set ``$A$'' defined in
	\ref{thm:criterion} with $q=p$, $F$ the Dirichlet integrand and $j=1$.
	Therefore, applying \ref{thm:criterion}, one may take $T_a \in
	\mathscr{D}' ( \oball{0}{R}, \rel^\codim )$ defined by $T_a ( \theta )
	= \int \theta (x) \bullet \Lap Q_a(a) \ud \mathscr{L}^\vdim x$ for
	$\theta \in \mathscr{D} ( \rel^\vdim, \rel^\codim )$.

	The uniqueness follows, since every $T_a$ admissible in the conclusion
	satisfies
	\begin{gather*}
		r^{-\vdim} T_a ( \theta \circ \mutau{r}{a})= T_a ( \theta ),
		\quad r^{-\vdim} T ( \theta \circ \mutau{r}{a} ) \to T_a (
		\theta )\quad \text{as $r \pluslim{0}$}.
	\end{gather*}
	whenever $\theta  \in \mathscr{D} ( \rel^\vdim, \rel^\codim )$.
\end{proof}
\begin{lemma} \label{lemma:l1_estimate}
	Suppose $\vdim, \adim \in \nat$, $\vdim < \adim$, $\Phi \in \bigodot^2
	\Hom ( \rel^\vdim, \rel^\codim )$, $0 < c \leq M < \infty$, $\| \Phi
	\| \leq M$, $\Phi$ is strongly elliptic with ellipticity bound $c$, $a
	\in \rel^\vdim$, $0 <  r < \infty$, $u \in \Sob{0}{1}{1} (
	\oball{a}{r}, \rel^\codim )$, $T \in \mathscr{D}' ( \oball{a}{r},
	\rel^\codim )$, and
	\begin{gather*}
		- {\textstyle\int_{\oball{a}{r}}} \left < D \theta (x) \odot
		\weakD u (x), \Phi \right > \ud \mathscr{L}^\vdim x = T (
		\theta ) \quad \text{for $\theta \in \mathscr{D} (
		\oball{a}{r}, \rel^\codim )$}.
	\end{gather*}

	Then
	\begin{gather*}
		\norm{u}{1}{a,r} \leq \Gamma r \dnorm{T}{1}{a,r}
	\end{gather*}
	where $\Gamma$ is a positive, finite number depending only on
	$\adim$, $c$, and $M$.
\end{lemma}
\begin{proof}
	See \cite[6.8]{snulmenn:decay.v2}.
\end{proof}
\begin{lemma} \label{lemma:elem}
	Suppose $\vdim, \adim \in \nat$, $\vdim < \adim$, $0 < c \leq M <
	\infty$,
	\begin{gather*}
		\text{$F : \Hom ( \rel^\vdim, \rel^\codim ) \to \rel$ is of
		class $\class{2}$}, \\
		\| D^2 F ( \sigma ) \| \leq M, \quad \left < ( \tau, \tau ),
		D^2 F ( \sigma ) \right > \geq c | \tau |^2 \qquad \text{for
		$\sigma, \tau \in \Hom ( \rel^\vdim, \rel^\codim)$},
	\end{gather*}
	$a \in \rel^\vdim$, $0 < r < \infty$, and $u,v \in \Sob{}{1}{2} (
	\oball{a}{r}, \rel^\codim )$ with
	\begin{gather*}
		u-v \in \Sob{0}{1}{2} ( \oball{a}{r}, \rel^\codim ).
	\end{gather*}

	Then for every affine function $P : \rel^\vdim \to \rel^\codim$
	\begin{gather*}
		\norm{\weakD (v-u)}{2}{a,r} \leq c^{-1} \big ( M
		\norm{\weakD (u-P)}{2}{a,r} + \dnorm{L_F(v)}{2}{a,r} \big )
	\end{gather*}
	where $L_F$ is defined as in \ref{miniremark:situation_F}.
\end{lemma}
\begin{proof}
	Compute for $\theta \in \mathscr{D} ( \oball{a}{r} , \rel^\codim )$
	\begin{gather*}
		\begin{aligned}
			L_F (v) ( \theta ) & = -
			{\textstyle\int_{\oball{a}{r}}} \left < D \theta
			(x), DF ( \weakD v (x) ) - DF ( DP (x) ) \right > \ud
			\mathscr{L}^\vdim x \\
			& = - {\textstyle\int_{\oball{a}{r}}} \left < D
			\theta (x) \odot \weakD ( v-P ) (x), A(x) \right > \ud
			\mathscr{L}^\vdim x
		\end{aligned} \\
		\text{where $A(x) = {\textstyle\int_0^1} D^2 F ( t \weakD v
		(x) + (1-t) DP(x) ) \ud \mathscr{L}^1 t$}.
	\end{gather*}
	This implies for $\theta \in \mathscr{D} ( \oball{a}{r}, \rel^\codim
	)$
	\begin{gather*}
		\begin{aligned}
			& {\textstyle\int_{\oball{a}{r}}} \left < D \theta
			(x) \odot \weakD (v-u) (x), A(x) \right > \ud
			\mathscr{L}^\vdim x \\
			& \qquad = - {\textstyle\int_{\oball{a}{r}}} \left <
			D \theta (x) \odot \weakD (u-P) (x), A (x) \right > \ud
			\mathscr{L}^\vdim x - L_F (v) ( \theta ).
		\end{aligned}
	\end{gather*}
	Letting $\theta$ approximate $v-u$ in $\Sob{}{1}{2} ( \oball{a}{r},
	\rel^\codim )$, one obtains
	\begin{gather*}
		c ( \norm{\weakD (v-u)}{2}{a,r} )^2
		\leq \big ( M \norm{\weakD (u-P)}{2}{a,r} + \dnorm{L_F
		(v)}{2}{a,r} \big ) \norm{\weakD (v-u)}{2}{a,r}.
		\qedhere
	\end{gather*}
\end{proof}
\begin{lemma} \label{lemma:ms}
	Suppose $\vdim, \adim \in \nat$, $\vdim < \adim$, $\varepsilon=1/2$ is
	related to $F$ as in \ref{miniremark:situation_F}, $\Lip D^2 F <
	\infty$, $a \in \rel^\vdim$, $0 < r < \infty$, and $u,v \in
	\Sob{}{1}{2} ( \oball{a}{r}, \rel^\codim )$ with $u-v \in
	\Sob{0}{1}{2} ( \oball{a}{r}, \rel^\codim )$.

	Then for every affine function $P : \rel^\vdim \to \rel^\codim$
	\begin{multline*}
		r^{-1-\vdim} \norm{v-u}{1}{a,r} \leq \Gamma r^{-\vdim} \big
		( \dnorm{L_F (v)-L_F (u)}{1}{a,r} \\
		+ \Lip ( D^2 F ) ( \norm{\weakD (u-P)}{2}{a,r} + \norm{\weakD
		( v-P )}{2}{a,r} )^2 \big )
	\end{multline*}
	where $\Gamma = \Gamma_{\ref{lemma:l1_estimate}} ( \adim, 1/2/, 3/2 )$.
\end{lemma}
\begin{proof}
	Let $\Lambda = \Lip D^2 F$, choose $\sigma \in \Hom ( \rel^\vdim,
	\rel^\codim )$ such that $DP(x) = \sigma$ for $x \in \rel^\vdim$, and
	define $T = L_F (v) - L_F (u)$, the $\mathscr{L}^\vdim \restrict
	\oball{a}{r}$ measurable function $A : \oball{a}{r} \to \bigodot^2
	\Hom ( \rel^\vdim, \rel^\codim )$ by
	\begin{gather*}
		A(x) = {\textstyle\int_0^1} D^2F ( t \weakD v (x) + ( 1-t )
		\weakD u (x) ) - D^2 F ( \sigma ) \ud \mathscr{L}^1 t 
	\end{gather*}
	whenever $x \in \oball{a}{r}$, and $S \in \mathscr{D}' (
	\oball{a}{r}, \rel^\codim )$ by
	\begin{gather*}
		S ( \theta ) = {\textstyle\int_{\oball{a}{r}}} \left < D
		\theta (x) \odot \weakD (v-u) (x), A (x) \right > \ud
		\mathscr{L}^\vdim x + T ( \theta )
	\end{gather*}
	whenever $\theta \in \mathscr{D} ( \oball{a}{r}, \rel^\codim
	)$. One computes
	\begin{gather*}
		\begin{aligned}
			& DF ( \weakD v (x) ) - DF ( \weakD u (x) ) \\
			& \qquad = \big < \weakD (v-u) (x),
			{\textstyle\int_0^1} DD F ( t \weakD v(x) + (1-t)
			\weakD u(x) ) \ud \mathscr{L}^1 t \big >
		\end{aligned}
	\end{gather*}
	for $\mathscr{L}^n$ almost all $x \in \oball{a}{r}$ and infers
	\begin{gather*}
		S ( \theta ) = - {\textstyle\int_{\oball{a}{r}}} \left < D
		\theta (x) \odot \weakD (v-u) (x), D^2 F ( \sigma ) \right >
		\ud \mathscr{L}^\vdim x
	\end{gather*}
	whenever $\theta \in \mathscr{D} ( \oball{a}{r}, \rel^\codim
	)$, hence by \ref{lemma:l1_estimate} with $\Phi$ replaced by $D^2 F (
	\sigma )$
	\begin{gather*}
		r^{-1-\vdim} \norm{v-u}{1}{a,r} \leq \Gamma r^{-\vdim}
		\dnorm{S}{1}{a,r}
	\end{gather*}
	It remains to estimate $\dnorm{S}{1}{a,r}$. By use of the definition of
	$S$ one estimates
	\begin{gather*}
		\begin{aligned}
			\| A(x) \| & \leq {\textstyle\int_0^1} \| D^2F ( t
			\weakD v (x) + (1-t) \weakD u(x) ) - D^2F ( t \sigma +
			(1-t) \sigma ) \| \ud \mathscr{L}^1 t \\
			& \leq \Lambda {\textstyle\int_0^1} t | \weakD
			(v-P) (x) | + (1-t) | \weakD (u-P) (x) | \ud
			\mathscr{L}^1 t \\
			& = \Lambda ( | \weakD (v-P) (x) | + | \weakD
			(u-P) (x) | ) / 2
		\end{aligned}
	\end{gather*}
	for $\mathscr{L}^\vdim$ almost all $x \in \oball{a}{r}$. Finally,
	\begin{gather*}
		\dnorm{S}{1}{a,r} \leq \dnorm{T}{1}{a,r} + \Lambda/2
		{\textstyle\int_{\oball{a}{r}}} ( | \weakD (u-P)
		(x) | + | \weakD (v-P) (x) | )^2 \ud \mathscr{L}^\vdim x.
		\qedhere
	\end{gather*}
\end{proof}
\begin{miniremark} \label{miniremark:constant_distributions}
	Whenever $\vdim, \adim \in \nat$, $\vdim < \adim$, $U$ is an open
	subset of $\rel^\vdim$, $a \in U$, and $T \in \mathscr{D}' ( U,
	\rel^\codim )$ there exists at most one constant distribution $T_a \in
	\mathscr{D}' ( U, \rel^\codim)$ such that
	\begin{gather*}
		\lim_{r \to 0+} r^{-\vdim-1} \dnorm{T-T_a}{1}{a,r} = 0,
	\end{gather*}
	see the last paragraph of the proof of
	\ref{corollary:lebesgue_points_p}.
\end{miniremark}
\begin{lemma} \label{lemma:arbi}
	Suppose $\vdim, \adim \in \nat$, $\vdim < \adim$.

	Then there exists a positive, finite number $\varepsilon$ with the
	following property.

	If $F$ is related to $\varepsilon$ as in \ref{miniremark:situation_F},
	$\Lip D^2 F < \infty$, $U$ is an open subset of $\rel^\vdim$, $u : U
	\to \rel^\codim$ is weakly differentiable, $A_1$ denotes the set of
	all $a \in U$ such that
	\begin{gather*}
		\limsup_{r \pluslim{0}} r^{-\vdim-1} \dnorm{L_F (u)}{1}{a,r} <
		\infty,
	\end{gather*}
	$A_2$ denotes the set of all $a \in U$ such that there exists a
	(unique, see \ref{miniremark:constant_distributions}) constant
	distribution $T_a \in \mathscr{D}' ( U, \rel^\codim )$ such that
	\begin{gather*}
		\lim_{r \pluslim{0}} r^{-\vdim-1} \dnorm{L_F(u)-T_a}{1}{a,r} =
		0,
	\end{gather*}
	$B_1$ denotes the set of all $b \in \dmn \weakD u$ such that
	\begin{gather*}
		\limsup_{r \pluslim{0}} r^{-\vdim-1}
		{\textstyle\int_{\oball{b}{r}}} | \weakD u (x) -
		\weakD u (b) |^2 \ud \mathscr{L}^\vdim x < \infty,
	\end{gather*}
	and $B_2$ denotes the set of all $b \in \dmn \weakD u$ such that
	\begin{gather*}
		\lim_{r \pluslim{0}} r^{-\vdim-1}
		{\textstyle\int_{\oball{b}{r}}} | \weakD u (x) - \weakD u (b)
		|^2 \ud \mathscr{L}^\vdim x = 0,
	\end{gather*}
	then the following two statements hold:
	\begin{enumerate}
		\item \label{item:arbi:1} For $\mathscr{L}^\vdim$ almost all
		$a \in A_1 \cap B_1$ there exists a polynomial function $Q_a :
		\rel^\vdim \to \rel^\codim$ of degree at most $2$ such that
		\begin{gather*}
			\lim_{r \pluslim{0}} r^{-2-\vdim} \norm{u -
			Q_a}{1}{a,r} = 0.
		\end{gather*}
		\item \label{item:arbi:2} If $a \in A_2 \cap B_2$ satisfies
		the conclusion of \eqref{item:arbi:1} with $Q_a$ then
		\begin{gather*}
			T_a ( \theta ) = {\textstyle\int_U} \theta (x) \bullet
			\left < D^2 Q_a (a), C_F ( DQ_a (a) ) \right > \ud
			\mathscr{L}^\vdim x
		\end{gather*}
		for $\theta \in \mathscr{D} (U,\rel^\codim)$ where $C_F$ is
		defined as in \ref{miniremark:situation_F}.
	\end{enumerate}
\end{lemma}
\begin{proof}
	Let
	\begin{gather*}
		\varepsilon = \min \{ 1/2, \varepsilon_{\ref{thm:criterion}} (
		\vdim, \adim, 1, 2),
		\varepsilon_{\ref{lemma:higher_differentiability}} ( \adim, 2,
		2 ) \}.
	\end{gather*}
	Suppose $F$ and $u$ satisfy the hypotheses with $\varepsilon$.
	Abbreviate $\Lambda = \Lip D^2 F$ and $T = L_F (u)$. Fix $a \in A_1
	\cap B_1$ and $0 < R < \infty$ such that $\cball{a}{R} \subset U$ and
	$u | \oball{a}{R} \in \Sob{}{1}{2} ( \oball{a}{R}, \rel^\codim )$.

	To prove part \eqref{item:arbi:1}, the criterion \ref{thm:criterion}
	will be verified with $q=2$, $j=0$. Using the direct method of the
	calculus of variation, see e.g. \cite[Theorems 4.5,\,6, Remark
	4.1]{MR1962933}, one constructs for $0 < r < R$ functions $v_r \in
	\Sob{}{1}{2} ( \oball{a}{r}, \rel^\codim)$ such that
	\begin{gather*}
		v_r - u \in \Sob{0}{1}{2} ( \oball{a}{r},
		\rel^\codim ), \quad L_F (v_r) = 0.
	\end{gather*}
	By \ref{lemma:ms} one estimates
	\begin{gather*}
		\begin{aligned}
			& r^{-1-\vdim} \norm{v_r - u}{1}{a,r} \\
			& \qquad \leq \Delta_1 r^{-\vdim} \big (
			\dnorm{T}{1}{a,r} + \Lambda ( \norm{\weakD (
			u-\weakD u (a))}{2}{a,r} + \norm{\weakD ( v_r - \weakD u (a))}{2}{a,r} )^2
			\big ).
		\end{aligned}
	\end{gather*}
	with $\Delta_1 = \Gamma_{\ref{lemma:ms}} ( \adim )$. By
	\ref{lemma:elem} with $c=1/2$, $M=2$ one infers
	\begin{gather*}
		\norm{\weakD ( v_r - u )}{2}{a,r} \leq 4
		\norm{\weakD ( u - \weakD u (a) )}{2}{a,r},
	\end{gather*}
	hence
	\begin{gather*}
		r^{-1-\vdim} \norm{v_r - u}{1}{a,r} \leq \Delta_1
		r^{-\vdim} \big ( \dnorm{T}{1}{a,r} + \Lambda ( 6 \norm{\weakD
		( u - \weakD u (a) )}{2}{a,r} )^2 \big ).
	\end{gather*}
	Since $a \in A_1 \cap B_1$, this implies
	\begin{gather*}
		\limsup_{r \pluslim{0}} r^{-2-\vdim} \norm{v_r
		- u}{1}{a,r} < \infty.
	\end{gather*}
	Therefore part \eqref{item:arbi:1} follows from
	\ref{thm:criterion}.

	To prove part \eqref{item:arbi:2}, assume now additionally that the
	assumptions of \eqref{item:arbi:2} are valid for $a$, i.e. $a \in A_2
	\cap B_2$ and $Q_a$ satisfies the conclusion of \eqref{item:arbi:1}.
	Choose $y \in \rel^\codim$ such that
	\begin{gather*}
		T_a ( \theta ) = {\textstyle\int_U} \theta (x) \bullet y \ud
		\mathscr{L}^\vdim x \quad \text{for $\theta \in \mathscr{D} (
		U, \rel^\codim )$}.
	\end{gather*}
	Using the direct method of the calculus of variation as before, one
	constructs for $0 < r < R$ functions $w_r \in \Sob{}{1}{2} (
	\oball{a}{r}, \rel^\codim )$ such that
	\begin{gather*}
		w_r - u \in \Sob{0}{1}{2} ( \oball{a}{r},
		\rel^\codim ), \\
		L_F ( w_r ) ( \theta ) = {\textstyle\int_{\oball{a}{r}}}
		\theta (x) \bullet y \ud \mathscr{L}^\vdim x \quad
		\text{whenever $\theta \in \mathscr{D} ( \oball{a}{r},
		\rel^\codim )$}.
	\end{gather*}
	By \ref{lemma:ms} one estimates
	\begin{multline*}
		r^{-1-\vdim} \norm{w_r - u}{1}{a,r} \leq \Delta_1 r^{- \vdim}
		\big ( \dnorm{T-T_a}{1}{a,r} \\
		+ \Lambda ( \norm{\weakD (u-\weakD u (a))}{2}{a,r} +
		\norm{\weakD ( w_r - \weakD u (a) )}{2}{a,r} )^2 \big ).
	\end{multline*}
	Since, by Poincar\'e's inequality,
	\begin{gather*}
		\big | {\textstyle\int_{\oball{a}{r}}} \theta (x)
		\bullet y \ud \mathscr{L}^\vdim x \big | \leq |y| \Delta_2
		r^{1+\vdim/2} \norm{D\theta}{2}{a,r}
	\end{gather*}
	where $\Delta_2$ is a positive, finite number depending only on
	$\adim$, one infers from \ref{lemma:elem}
	\begin{gather*}
		\norm{\weakD ( w_r - u )}{2}{a,r} \leq 4
		\norm{\weakD ( u - \weakD u (a) )}{2}{a,r} + 2 \Delta_2 |y|
		r^{1+\vdim/2},
	\end{gather*}
	hence
	\begin{gather*}
		\begin{aligned}
			& r^{-1-\vdim} \norm{w_r - u}{1}{a,r} \\
			& \qquad \leq \Delta_1 r^{-\vdim} \big (
			\dnorm{T-T_a}{1}{a,r} + \Lambda ( 6 \norm{\weakD ( u -
			\weakD u (a) )}{2}{a,r} + 2 \Delta_2 |y| r^{1+\vdim/2}
			)^2 \big).
		\end{aligned}
	\end{gather*}
	Since $a \in A_2 \cap B_2$, this implies
	\begin{gather*}
		\lim_{r \pluslim{0}} r^{-2-\vdim} \norm{w_r -
		u}{1}{a,r} = 0.
	\end{gather*}
	Therefore by the assumption on $Q_a$
	\begin{gather*}
		\lim_{r \pluslim{0}} r^{-2-\vdim} \norm{w_r -
		Q_a}{1}{a,r} = 0.
	\end{gather*}

	In order to estimate derivatives of $w_r - Q_a$, define $P :
	\rel^\vdim \to \rel^\codim$ by $P(x) = Q_a (a) + \left < x-a, DQ_a (a)
	\right >$ for $x \in \rel^\vdim$, $R=Q_a-P$, $S : \rel^\vdim \to
	\rel^\codim$ by $S(x) = \frac{1}{2} \left <(x,x), D^2Q_a(a) \right >$
	for $x \in \rel^\vdim$ and note $r^{-2} R \circ \taumu{a}{r} = S$ and
	\begin{gather*}
		r^{-2} ( w_r - P ) \circ \taumu{a}{r} |
		\oball{0}{1} \to S | \oball{0}{1} \quad \text{in
		$\Lp{1} ( \oball{0}{1}, \rel^\codim )$}
	\end{gather*}
	as $r \pluslim{0}$. By \ref{lemma:higher_differentiability}
	\begin{gather*}
		r^{-\vdim/2} \norm{\weakD^2 ( w_r - P )}{2}{a,r/2} \leq
		\Delta_3 ( r^{-2-\vdim} \norm{w_r - P}{1}{a,r} + |y| )
	\end{gather*}
	where $\Delta_3 = \max \{ 1, \unitmeasure{\vdim}^{1/2} \}
	\Gamma_{\ref{lemma:higher_differentiability}} ( \adim, 2)$, hence
	\begin{gather*}
		\limsup_{r \pluslim{0}} r^{-\vdim/2} \norm{\weakD^2 ( w_r
		- P )}{2}{a,r/2} < \infty.
	\end{gather*}
	By Rellich's embedding theorem
	\begin{gather*}
		r^{-2} ( w_r - P ) \circ \taumu{a}{r} | \oball{0}{1/2} \to S |
		\oball{0}{1/2} \quad \text{in $\Sob{}{1}{2} ( \oball{0}{1/2},
		\rel^\codim )$}, \\
		r^{-2} ( w_r - Q_a ) \circ \taumu{a}{r} | \oball{0}{1/2} \to 0
		\quad \text{in $\Sob{}{1}{2} ( \oball{0}{1/2}, \rel^\codim)$}
	\end{gather*}
	as $r \pluslim{0}$. This convergence implies
	\begin{gather*}
		\big | r^{-\vdim-1} {\textstyle\int_{\oball{a}{r/2}}} \leftB (
		D \theta ) \circ \mutau{r}{a} (x), DF ( \weakD w_r (x) ) - DF
		( DQ_a (x) ) \rightB \ud \mathscr{L}^\vdim x \big | \\
		\leq r^{-\vdim/2-1} ( \Lip DF ) \norm{D\theta}{2}{0,1}
		\norm{\weakD (w_r - Q_a)}{2}{a,r} \to 0 \quad \text{as $r
		\pluslim{0}$}
	\end{gather*}
	for $\theta \in \mathscr{D} ( \oball{0}{1/2}, \rel^\codim )$.
	Therefore, noting
	\begin{gather*}
		{\textstyle\int_{\oball{0}{1/2}}} \theta (x) \bullet y \ud
		\mathscr{L}^\vdim x = r^{-\vdim}
		{\textstyle\int_{\oball{a}{r/2}}} ( \theta \circ \mutau{r}{a})
		(x) \bullet y \ud \mathscr{L}^\vdim x \\
		= - r^{-\vdim-1} {\textstyle\int_{\oball{a}{r/2}}} \leftB ( D
		\theta ) \circ \mutau{r}{a} (x) , DF ( \weakD w_r (x) )
		\rightB \ud \mathscr{L}^\vdim x
	\end{gather*}
	for $\theta \in \mathscr{D} ( \oball{0}{1/2}, \rel^\codim )$ and
	\begin{gather*}
		\begin{aligned}
			& - r^{-\vdim-1} {\textstyle\int_{\oball{a}{r/2}}}
			\leftB ( D\theta ) \circ \mutau{r}{a} (x) , DF ( DQ_a
			(x) ) \rightB \ud \mathscr{L}^\vdim x \\
			& \qquad = r^{-\vdim}
			{\textstyle\int_{\oball{a}{r/2}}} ( \theta \circ
			\mutau{r}{a} ) ( x ) \bullet \left < D^2 Q_a (x) , C_F
			( DQ_a (x) ) \right > \ud \mathscr{L}^\vdim x \\
			& \qquad \quad \to {\textstyle\int_{\oball{0}{1/2}}}
			\theta (x) \bullet \left < D^2 Q_a (a), C_F ( DQ_a (a)
			) \right > \ud \mathscr{L}^\vdim x \quad \text{as $r
			\pluslim{0}$},
		\end{aligned}
	\end{gather*}
	for $\theta \in \mathscr{D} ( \oball{0}{1/2}, \rel^\codim )$, one
	infers
	\begin{gather*}
		y = \left < D^2 Q_a (a), C_F ( DQ_a (a) ) \right >,
	\end{gather*}
	as asserted.
\end{proof}
\begin{remark} \label{remark:arbi}
	Clearly, by Re{\v s}etnyak \cite{Reshetnyak_diff} or
	\cite[4.5.9\,(26)\,(\printRoman{2})\,(\printRoman{3})]{MR41:1976} for
	$\mathscr{L}^\vdim$ almost all $a \in A_1 \cap B_1$
	\begin{gather*}
		Q_a (a) = u (a), \quad DQ_a(a) = \weakD u (a).
	\end{gather*}
	Also by Calder\'on and Zygmund \cite[Theorem 9]{MR0136849} (see also
	\cite[3.6--8]{MR1014685}), there exists a sequence of functions $u_i :
	\rel^\vdim \to \rel^\codim$ of class $\class{2}$ such that
	\begin{gather*}
		\mathscr{L}^\vdim \left ( A_1 \cap B_1 \without
		\bigcup_{i=1}^\infty \left \{ a \with \text{$D^k u_i (a) = D^k
		Q_a (a)$ for $k \in \{ 0,1,2 \}$} \right \} \right ) = 0.
	\end{gather*}
\end{remark}
\begin{remark}
	In \ref{thm:distrib_diff} it will be shown $\mathscr{L}^\vdim ( A_1
	\without A_2 ) = 0$.
\end{remark}
\begin{lemma} \label{lemma:cutoff}
	Suppose $H$ is a Hilbert space with $\dim H = N < \infty$, $k,l \in
	\nat \cup \{ 0 \}$, $l \geq k$, $\Phi : H \to \rel$ is of class $l$, $a
	\in H$, $0 < \delta < \infty$, and
	\begin{gather*}
		s = \sup \{ \| D^k \Phi ( x ) - D^k \Phi ( a ) \| \with x \in
		\cball{a}{\delta} \}.
	\end{gather*}

	Then there exists $F : H \to \rel$ of class $l$ such that
	\begin{gather*}
		D^i F (x) = D^i \Phi (x) \quad \text{for $x \in
		\cball{a}{\delta/2}$, $i = 0, \ldots, k$}, \\
		\| D^k F ( x ) - D^k \Phi ( a ) \| \leq \Gamma s \quad \text{for
		$x \in H$}, \\
		\text{$F | H \without \cball{a}{\delta}$ is the restriction
		of a polynomial function of degree at most $k$}
	\end{gather*}
	where $\Gamma$ is a positive, finite number depending only on $N$ and
	$k$.
\end{lemma}
\begin{proof}
	Choosing $\varphi \in \mathscr{E}^0 ( \rel )$ with $0 \leq \varphi (t)
	\leq 1$ for $t \in \rel$ and
	\begin{gather*}
		\{ t \with -\infty < t \leq 1/2 \} \subset \Int \{ t \with
		\varphi (t) = 1 \}, \quad \{ t \with 1 \leq t < \infty \}
		\subset \Int \{ t \with \varphi (t) = 0 \}
	\end{gather*}
	one defines $P : H \to \rel$, $F : H \to \rel$ by
	\begin{gather*}
		P (x) = {\textstyle\sum_{i=0}^k} \left < ( x-a )^i / i! , D^i
		\Phi ( a ) \right >, \\
		F (x) = P ( x ) + \varphi ( |x-a|/ \delta ) ( \Phi ( x ) - P (x)
		)
	\end{gather*}
	for $x \in H$ and readily estimates $\| D^k F ( x ) - D^k \Phi ( a ) \|$
	be means of Taylor's formula (cf. \cite[3.1.11]{MR41:1976}).
\end{proof}
\section{An approximate second order structure for certain integral varifolds}
\label{sec:main_theorem}
In this Section \ref{thm:main_theorem} which is Theorem \ref{ithm:c2} of the
Introduction is proven. In order to do this a general lemma is established
which states that the part of the varifold exhibiting a certain decay of its
tilt-excess can be covered with some accuracy by suitable rotated graphs of
Lipschitzian function having similar decay properties of their
``tilt-excess''. This is done by carefully combining the approximation by
$\qspace_Q ( \rel^\codim )$ valued functions of \cite[4.8]{snulmenn:decay.v2}
with more basic differentiability results in \cite{snulmenn.isoperimetric}.
The ``tilt-excess'' decay of the Lipschitzian functions is the nonintegral
differentiability condition used in Section \ref{sec:crit} to compensate for
the use of the weak norm $\dnorm{\cdot}{1}{a,s}$ in the estimates which seems
to be inavoidable, see \ref{remark:dnorm}.
\begin{lemma} \label{lemma:lipschitz_approximation}
	Suppose $\adim, Q \in \nat$, $0 < L < \infty$, $1 \leq M < \infty$,
	$0 < \delta_i \leq 1$ for $i \in \{1,2,3\}$, and $0 < \delta_4 \leq
	1/4$.
	
	Then there exists a positive, finite number $\varepsilon$ with the
	following property.
	
	If $\vdim \in \nat$, $\vdim < \adim$, $0 < s < \infty$, $S = \im
	\pp^\ast$,
	\begin{gather*}
		U = \eqclassification{\rel^\vdim \times
		\rel^\codim}{(x,y)}{\dist ((x,y), \cylinder{S}{0}{s}{s}) <
		2s},
	\end{gather*}
	$V \in \IVar_\vdim ( U )$, $\| \delta V \|$ is a Radon measure,
	\begin{gather*}
		( Q - 1 + \delta_1 ) \unitmeasure{\vdim} s^\vdim \leq \| V \| (
		\cylinder{S}{0}{s}{s} ) \leq ( Q + 1 - \delta_2 )
		\unitmeasure{\vdim} s^\vdim, \\
		\| V \| ( \cylinder{S}{0}{s}{s+\delta_4 s} \without
		\cylinder{S}{0}{s}{s-2\delta_4 s}) \leq ( 1 - \delta_3 )
		\unitmeasure{\vdim} s^\vdim, \\
		\| V \| ( U ) \leq M \unitmeasure{\vdim} s^\vdim,
	\end{gather*}
	$0 < \delta \leq \varepsilon$, $B$ denotes the set of all $z
	\in \cylinder{S}{0}{s}{s}$ with $\density^{\ast \vdim} ( \| V \|, z) >
	0$ such that
	\begin{gather*}
		\text{either} \quad
		\measureball{\| \delta V \|}{\cball{z}{t}} >
		\delta \, \| V \| ( \cball{z}{t} )^{1-1/\vdim}
		\quad \text{for some $0 < t < 2 s$}, \\
		\text{or} \quad {\textstyle\int_{\cball{z}{t} \times
		\grass{\adim}{\vdim}}} | \project{R} - \project{S} | \ud V
		(\xi,R) > \delta \, \measureball{\| V
		\|}{\cball{z}{t}} \quad \text{for some $0 < t <
		2 s$},
	\end{gather*}
	$A = \cylinder{S}{0}{s}{s} \without B$, $A (x) =
	\classification{A}{z}{\pp (z) = x}$ for $x \in \rel^\vdim$, $X_1$ is
	the set of all $x \in \rel^\vdim \cap \cball{0}{s}$ such that
	\begin{gather*}
		{\textstyle\sum_{z \in A(x)}} \density^\vdim ( \| V \|, z )
		= Q \quad \text{and} \quad \text{$\density^\vdim ( \| V \|,
		z ) \in \nat \cup \{0\}$ for $z \in A(x)$},
	\end{gather*}
	$X_2$ is the set of all $x \in \rel^\vdim \cap \cball{0}{s}$ such
	that
	\begin{gather*}
		{\textstyle\sum_{z \in A(x)}} \density^\vdim ( \| V \|, z )
		\leq Q - 1 \quad \text{and} \quad \text{$\density^\vdim
		( \| V \|, z ) \in \nat \cup \{ 0 \}$ for $z \in A(x)$},
	\end{gather*}
	$N = \rel^\vdim \cap \cball{0}{s} \without ( X_1 \cup X_2 )$, and $f :
	X_1 \to \qspace_Q ( \rel^{\codim} )$ is characterised by the
	requirement
	\begin{gather*}
		\density^\vdim ( \| V \|, z) = \density^0 ( \| f (x) \|, \qq
		(z) ) \quad \text{whenever $x \in X_1$ and $z \in A(x)$},
	\end{gather*}
	then the following seven statements hold:
	\begin{enumerate}
		\item \label{item:lipschitz_approximation:yz} $X_1$ and $X_2$
		are universally measurable, and $\mathscr{L}^\vdim (N) =0$.
		\item \label{item:lipschitz_approximation:ab} $A$ and $B$
		are Borel sets and
		\begin{gather*}
			\qq \lIm A \cap \spt \| V \| \rIm \subset
			\cball{0}{s-\delta_4s}.
		\end{gather*}
		\item \label{item:lipschitz_approximation:para} $\pp \lIm
		\classification{A}{z}{\density^\vdim ( \| V \|, z ) = Q} \rIm
		\subset X_1$.
		\item \label{item:lipschitz_approximation:lip} The function
		$f$ is Lipschitzian with $\Lip f \leq L$.
		\item \label{item:lipschitz_approximation:misc} For
		$\mathscr{L}^\vdim$ almost all $x \in X_1$ the following is
		true:
		\begin{enumerate}
			\item
			\label{item:item:lipschitz_approximation:misc:a} The
			function $f$ is approximately strongly affinely
			approximable at $x$.
			\item
			\label{item:item:lipschitz_approximation:misc:b}
			If $(x,y) \in \graph_Q f$ then
			\begin{gather*}
				\Tan^\vdim ( \| V \|, (x,y) ) = \Tan \big (
				\graph_Q \ap A f ( x ) , (x,y) \big) \in
				\grass{\adim}{\vdim}.
			\end{gather*}
		\end{enumerate}
		\item \label{item:lipschitz_approximation:estimate_b} If $a
		\in A$, $\density^\vdim ( \| V \|, a ) = Q$, $0 < t \leq
		s - | \pp(a) |$, $| \qq(a) | + \delta_4 t \leq s$, and
		\begin{align*}
			B_{a,t} & = \cylinder{S}{a}{t}{\delta_4
			t} \cap B, \\
			C_{a,t} & = \cball{\pp (a)}{t}
			\without ( X_1 \without \pp \lIm B_{a,t} \rIm ),
			\\
			D_{a,t} & = \cylinder{S}{a}{t}{\delta_4
			t} \cap \pp^{-1} \lIm C_{a,t} \rIm,
		\end{align*}
		then $B_{a,t}$ is a Borel set, $C_{a,t}$ and
		$D_{a,t}$ are universally measurable and
		\begin{gather*}
			\mathscr{L}^\vdim ( C_{a,t} ) + \| V \| (
			D_{a,t}) \leq
			\Gamma_{\eqref{item:lipschitz_approximation:estimate_b}}
			\, \| V \| ( B_{a,t} )
		\end{gather*}
		with
		$\Gamma_{\eqref{item:lipschitz_approximation:estimate_b}} =
		3 + 2Q + ( 12Q + 6 ) 5^\vdim$.
		\item \label{item:lipschitz_approximation:pde} If $a$,
		$t$, $C_{a,t}$, $D_{a,t}$ are as in
		\eqref{item:lipschitz_approximation:estimate_b}, $g :
		\rel^\vdim \to \rel^\codim$, $\Lip g < \infty$, $g | X_1 =
		\boldsymbol{\eta}_Q \circ f$, $\tau \in \Hom ( \rel^\vdim,
		\rel^\codim )$, $\theta \in \mathscr{D} ( \rel^\vdim,
		\rel^\codim)$, $\eta \in \mathscr{D}^0 ( \rel^\codim )$,
		\begin{gather*}
			\spt \theta \subset \oball{\pp(a)}{t}, \qquad
			0 \leq \eta (y) \leq 1 \quad \text{for $y \in
			\rel^\codim$}, \\
			\spt \eta \subset \oball{\qq (a)}{\delta_4 t},
			\quad \cball{\qq (a)}{\delta_4 t/2 } \subset
			\Int ( \classification{\rel^\codim}{y}{\eta (y) = 1 }
			),
		\end{gather*}
		and $\Psi^\S$ denotes the nonparametric integrand associated
		to the area integrand $\Psi$, then
		\begin{gather*}
			\big | Q \tint{}{} \big < D \theta (x), D\Psi^\S_0 ( D
			g (x)) \big > \ud \mathscr{L}^\vdim x - ( \delta V )
			( ( \eta \circ \qq) \cdot ( \qq^\ast \circ \theta
			\circ \pp ) ) \big | \\
			\begin{aligned}
				& \leq \gamma_1 Q \vdim^{1/2} \Lip g
				{\textstyle\int_{C_{a,t}}} | D \theta |
				\ud \mathscr{L}^\vdim \\
				& \phantom{\leq} \ + \gamma_2
				{\textstyle\int_{E_{a,t} \without
				C_{a,t}}} | D \theta (x) | | \ap A f (x)
				\aplus ( - \tau ) |^2 \ud \mathscr{L}^\vdim x
				\\
				& \phantom{\leq} \ + \vdim^{1/2}
				{\textstyle\int_{D_{a,t}}} | D ( ( \eta
				\circ \qq ) \cdot ( \qq^\ast \circ \theta
				\circ \pp ) ) | \ud \| V \|
			\end{aligned}
		\end{gather*}
		where
		\begin{gather*}
			\gamma_1 = \sup \| D^2 \Psi_0^\S \| \lIm
			\cball{0}{\vdim^{1/2}\Lip g} \rIm, \\
			\gamma_2 = \Lip \big ( D^2 \Psi_0^\S|
			\cball{0}{\vdim^{1/2}(L+2\| \tau \|) } \big ), \\
			E_{a,t} = \cball{\pp (a)}{t } \cap
			\classification{X_1}{x}{\density^0 ( \| f (x) \|, g
			(x)) \neq Q }.
		\end{gather*}
	\end{enumerate}
\end{lemma}
\begin{proof}
	This follows from \cite[4.8,\,10]{snulmenn:decay.v2}; in fact the
	statements
	\eqref{item:lipschitz_approximation:yz}--\eqref{item:lipschitz_approximation:misc}
	are those in \cite[4.8]{snulmenn:decay.v2} with $r$, $h$, $T$ replaced
	by $s$, $s$, $S$ and \cite[4.10]{snulmenn:decay.v2} shows that the
	additional conditions $a \in A$ and $\density^\vdim ( \| V, \|, a ) =
	Q$ in
	\eqref{item:lipschitz_approximation:estimate_b}\,\eqref{item:lipschitz_approximation:pde}
	can be arranged to imply
	\begin{gather*}
		\graph_Q f | \cball{\pp (a)}{t} \subset
		\cylinder{S}{a}{t}{\delta_4 t/2}, \\
		\| V \| ( \cylinder{S}{a}{t}{\delta_4t} ) \geq ( Q
		- 1/4 ) \unitmeasure{\vdim} t^\vdim,
	\end{gather*}
	hence 
	\eqref{item:lipschitz_approximation:estimate_b}\,\eqref{item:lipschitz_approximation:pde}
	are consequences of \cite[4.8\,(6)\,(7)\,(9)]{snulmenn:decay.v2}.
\end{proof}
\begin{miniremark} \label{miniremark:situation}
	The following situation will be studied: $\vdim, \adim \in \nat$,
	$\vdim < \adim$, $1 \leq p \leq \infty$, $U$ is an open subset of
	$\rel^\adim$, $V \in \Var_\vdim ( U)$, $\| \delta V \|$ is a Radon
	measure and, if $p > 1$,
	\begin{gather*}
		( \delta V ) ( g ) = - {\textstyle\int} g (z) \bullet
		\mathbf{h} (V;z) \ud \| V \| (z) \quad \text{whenever $g \in
		\mathscr{D} ( U, \rel^\adim )$}, \\
		\mathbf{h} (V;\cdot) \in \Lp{p} ( \| V \| \restrict K,
		\rel^\adim ) \quad \text{whenever $K$ is a compact subset of
		$U$}.
	\end{gather*}

	If $p < \infty$ then the measure $\psi$ is defined by
	\begin{gather*}
		\psi = \| \delta V \| \quad \text{if $p = 1$}, \qquad \psi =
		| \mathbf{h} ( V ; \cdot ) |^p \| V \| \quad \text{if $p >
		1$}.
	\end{gather*}
\end{miniremark}
\begin{lemma} \label{lemma:great_approx}
	Suppose $\vdim, \adim \in \nat$, $\vdim < \adim$, $1 \leq p \leq
	\vdim$, $1 \leq q < \infty$, $0 < \alpha \leq 1$, $\alpha q ( \vdim -
	p ) \leq \vdim p$, $0 < L < \infty$, $V \in \IVar_\vdim ( U )$, $\psi$
	is related to $p$ and $V$ as in \ref{miniremark:situation}, and $P$ is
	the set of all $a \in U$ such that $\Tan^\vdim ( \| V \|, a ) \in
	\grass{\adim}{\vdim}$ and
	\begin{gather*}
		\limsup_{s \to 0+} s^{-\alpha-\vdim/q} \big (
		\tint{\cball{a}{s} \times \grass{\adim}{\vdim}}{} |
		\project{S} - \project{\Tan^\vdim ( \| V \|, a )} |^q \ud V
		(z,S) \big )^{1/q} < \infty.
	\end{gather*}

	Then there exists a countable, disjointed family $H$ of $\| V \|$
	measurable subsets of $P$ such that $\| V \| ( P \without \bigcup H )
	= 0$ and for each $Z \in H$ there exists a nonempty open subset $O$ of
	$\orthproj{\adim}{\vdim}$ such that for each $\pi_1 \in O$ there exist
	\begin{gather*}
		g : \rel^\vdim \to \rel^\codim, \quad G : \rel^\vdim
		\to \rel^\adim, \quad K \subset \rel^\vdim, \quad Q
		\in \nat, \\
		\pi_2 \in \orthproj{\adim}{\codim} , \quad T \in \mathscr{D}'
		( \rel^\vdim, \rel^\codim )
	\end{gather*}
	with the following six properties:
	\begin{enumerate}
		\item \label{item:great_approx:def} $\pi_2 \circ \pi_1^\ast =
		0$, $G = \pi_1^\ast + \pi_2^\ast \circ g$, and $G \lIm K \rIm =
		Z$.
		\item \label{item:great_approx:lip} $\Lip g \leq L$.
		\item \label{item:great_approx:K} $K$ is an
		$\mathscr{L}^\vdim$ measurable subset of $\dmn Dg$.
		\item \label{item:great_approx:defT} $\int \big <
		D\theta(x), D\Psi_0^\S(Dg(x)) \big > \ud \mathscr{L}^\vdim x =
		T(\theta)$ for $\theta \in \mathscr{D} ( \rel^\vdim,
		\rel^\codim )$ where $\Psi$ denotes the area integrand.
		\item \label{item:great_approx:pde} Whenever $x \in K$ there
		holds with $z = G (x)$ and $R = \Tan^\vdim ( \| V \|, z )$
		\begin{gather*}
			\density^\vdim ( \| V \|, z ) = Q, \quad \im DG
			(x) = R, \\
			\begin{aligned}
				& \limsup_{s \pluslim{0}} s^{-\beta-\vdim/r}
				\big ( \tint{\cball{x}{s}}{} | Dg (\zeta) -
				Dg(x) |^r \ud \mathscr{L}^\vdim \zeta \big
				)^{1/r} \\
				& \qquad \leq 2 \vdim^{1/2} \limsup_{s
				\pluslim{0}} s^{-\beta-\vdim/r} \big(
				\tint{\cball{z}{s} \times
				\grass{\adim}{\vdim}}{} | \project{S} -
				\project{R} |^r \ud V ( \xi, S ) \big )^{1/r}
			\end{aligned}
		\end{gather*}
		whenever $0 < \beta \leq 1$, $1 \leq r < \infty$ and $\beta r
		\leq \alpha q$.
		\item \label{item:great_approx:diff} Whenever $x \in K$ there
		holds
		\begin{gather*}
			\lim_{s \pluslim{0}} s^{-\vdim-1}
			\dnorm{T-T_x}{1}{x,s} = 0
		\end{gather*}
		where $T_x \in \mathscr{D}' ( \rel^\vdim, \rel^\codim
		)$ is defined by
		\begin{gather*}
			T_x ( \theta ) = - {\textstyle\int} \Psi_0^\S ( Dg(x))
			\mathbf{h} (V;G(x) ) \bullet ( \pi_2^\ast \circ
			\theta ) (\zeta) \ud \mathscr{L}^\vdim \zeta
		\end{gather*}
		whenever $\theta \in \mathscr{D} ( \rel^\vdim,
		\rel^\codim )$.
	\end{enumerate}
\end{lemma}
\begin{proof}
	First, observe that if some $\| V \|$ measurable set $Z$ has the
	properties listed in the conclusion so does every $\| V \|$ measurable
	subset of $Z$. Therefore, in order to prove the assertion, it is
	enough to show that for $\| V \|$ almost all $a \in P$ there exists a
	$\| V \|$ measurable set $Z$ having the stated properties and
	additionally satisfies $\density^{\ast \vdim} ( \| V \| \restrict Z,
	a) > 0$; in fact one can then take a maximal, disjointed family $H$ of
	such $Z$ (hence $\| V \| (Z) > 0$) and note $H$ is countable and
	$\density^\vdim ( \| V \| \restrict \bigcup H, a ) = 0$ for
	$\mathscr{H}^\vdim$ almost all $a \in U \without \bigcup H$ by
	\cite[2.10.19\,(4)]{MR41:1976} so that $\| V \| ( P \without \bigcup
	H) > 0$ would contradict the maximality of $H$.

	Define $P'$ to be the set of all $z \in U$ such that $\Tan^\vdim ( \|
	V \|, z ) \in \grass{\adim}{\vdim}$ and
	\begin{gather*}
		\lim_{t \to 0+} t^{-1/2-\vdim/2} \big (
		\tint{\cball{z}{t} \times \grass{\adim}{\vdim}}{} |
		\project{S} - \project{\Tan^\vdim ( \| V \|, z )} |^2 \ud V
		(\xi,S) \big )^{1/2} = 0.
	\end{gather*}
	By Brakke \cite[5.7,\,5]{MR485012} or \cite[8.6]{snulmenn:decay.v2}
	there holds $\| V \| ( U \without P') = 0$. Therefore one may assume
	$\alpha q \geq 1$  possibly replacing $\alpha$, $q$ by $1/2$, $2$ if
	$\alpha q < 1$. Assume further $L \leq 1/8$ and suppose $Q \in \nat$.
	The remaining assertion will be shown to hold for $\| V \|$ almost all
	$a \in P$ with $\density^\vdim ( \| V \|, a ) = Q$. For this purpose
	define
	\begin{gather*}
		\delta_1 = \delta_2 = \delta_3 = 1/2, \quad \delta_4 = 1/4,
		\quad M = 5^m Q, \\
		\varepsilon = \inf \big \{
		\varepsilon_{\ref{lemma:lipschitz_approximation}} (\adim, Q,
		L, M, \delta_1, \delta_2, \delta_3, \delta_4 ), ( 2
		\isoperimetric{\vdim} )^{-1} \big \},
	\end{gather*}
	and $R : \classification{U}{z}{\Tan^\vdim ( \| V \|, z ) \in
	\grass{\adim}{\vdim}} \to \Hom ( \rel^\adim, \rel^\adim )$ by
	\begin{gather*}
		R ( z ) = \project{\Tan^\vdim ( \| V \|, z )} \quad
		\text{whenever $z \in U$ with $\Tan^\vdim ( \| V \|, z ) \in
		\grass{\adim}{\vdim}$}.
	\end{gather*}
	For $i \in \nat$ let $C_i$ denote the set of all $z \in \spt
	\| V \|$ such that either $\cball{z}{1/i} \not \subset U$ or
	\begin{gather*}
		\measureball{\| \delta V \|}{\cball{z}{t}} >
		(2 \varepsilon/3) \, \| V \| ( \cball{z}{t} )^{1-1/\vdim}
		\quad \text{for some $0 < t < 1/i$},
	\end{gather*}
	let $D_i (w)$ for $w \in \dmn R$ denote the set of all $z \in U$
	such that either $\cball{z}{1/i} \not \subset U$ or
	\begin{gather*}
		\tint{\cball{z}{t}}{} | R ( \xi ) - R(w) |^q \ud \| V \|
		\xi > ( \varepsilon / 3 )^q \, \measureball{\| V \|}{
		\cball{z}{t}} \quad \text{for some $0 < t < 1/i$}
	\end{gather*}
	and define $X_i$ for $i \in \nat$ by
	\begin{gather*}
		X_i =
		\bigclassification{U}{z}{\density^{\vdim^2 / ( \vdim-p )} (
		\| V \| \restrict C_i, z ) = 0 } \quad \text{if $p < \vdim$},
		\\
		X_i = U \without \Clos{C_i} \quad \text{if $p=\vdim$},
	\end{gather*}
	as well as $Y_i$ for $i \in \nat$ by
	\begin{gather*}
		Y_i = (\dmn R) \cap \big \{ w \with \density^{\vdim+\alpha q}
		( \| V \| \restrict D_i( w ), w ) = 0 \big \}.
	\end{gather*}
	Since $C_{i+1} \subset C_i$ and $D_{i+1} (w) \subset D_i (w)$ for $w
	\in \dmn R$, one notes $X_i \subset X_{i+1}$ and $Y_i \subset Y_{i+1}$
	for $i \in \nat$. $X_i$ are Borel sets. $Y_i$ are $\| V \|$ measurable
	sets by \cite[3.7\,(ii)]{snulmenn.isoperimetric}. $P$ is
	$\| V\|$ measurable by \cite[3.7]{snulmenn.isoperimetric}. Moreover,
	\begin{gather*}
		\| V \| \big ( U \without {\textstyle\bigcup} \{ X_i \with i
		\in \nat \} \big ) = 0, \quad 
		\| V \| \big ( P \without {\textstyle\bigcup} \{ Y_i \with i
		\in \nat \} \big ) = 0
	\end{gather*}
	by \cite[2.5,\,9,\,10, 3.7\,(ii)]{snulmenn.isoperimetric}.

	Define a measure $\mu$ on $U$ such that $\mu + | \mathbf{h} ( V ;
	\cdot ) | \| V \| = \| \delta V \|$ and $J =
	\classification{P}{z}{\density^\vdim ( \| V \|, z ) = Q}$. The
	remaining assertion will be shown at a point $a$ such that for some $i
	\in \nat$
	\begin{gather*}
		a \in X_i \cap Y_i \cap ( \dmn R ), \quad \cball{a}{4/i}
		\subset U, \\
		\density^\vdim ( \| V \|, a ) = Q, \quad \density^\vdim ( \| V
		\| \restrict U \without ( J \cap X_i \cap Y_i ) , a ) = 0, \\
		\text{$R$ is approximately continuous at $a$ with respect to
		$\| V \|$}.
	\end{gather*}
	These conditions are satisfied by $\| V \|$ almost all $a \in J$ by
	the preceding remarks and \cite[2.9.11,\,13]{MR41:1976}. Fix such $a$
	and $i$, choose $0 < \kappa \leq 1/2$ such that $ ( 1 + \kappa )^\vdim
	Q < Q + 1/2$, and define $\lambda = ( 1 + \kappa^2 )^{-1/2}$ and
	$\delta = (1 - \lambda)/2$. Noting for $S \in \grass{\adim}{\vdim}$
	with $| \project{S} - R (a) | < \delta$ and $0 < s < \infty$
	\begin{gather*}
		\classification{\rel^\adim}{z}{ | \project{S} ( z-a ) | \leq
		\lambda | z-a |} \subset \classification{\rel^\adim}{z}{ |
		R(a) (z-a) | \leq ( \lambda + \delta ) | z-a |}, \\
		\classification{\cylind{S}{a}{s}}{z}{ | \project{S} (z-a) | >
		\lambda |z-a| } \subset \cylinder{S}{a}{s}{\kappa s} \subset
		\cball{a}{(1+\kappa) s}, \\
		0 < \lambda + \delta < 1, \quad \density^\vdim ( \| V \|
		\restrict \{ z \with | R(a) (z-a) | \leq ( \lambda + \delta )
		|z-a| \} , a ) = 0
	\end{gather*}
	by \cite[3.2.16]{MR41:1976}, one infers the existence of $0 < s <
	(2i)^{-1}$ such that
	\begin{gather*}
		( Q-1/2 ) \unitmeasure{\vdim} s^\vdim \leq \| V \| (
		\cylinder{S}{a}{s}{s} ) \leq ( Q+1/2 ) \unitmeasure{\vdim}
		s^\vdim, \\
		\| V \| ( \cylinder{S}{a}{s}{5s/4} \without
		\cylinder{S}{a}{s}{s/2}) \leq (1/2) \unitmeasure{\vdim}
		s^\vdim, \\
		\| V \| ( \classification{\rel^\adim}{z}{\dist ( z,
		\cylinder{S}{a}{s}{s} ) < 2s } ) \leq \measureball{\| V \|}{
		\cball{a}{4s}} \leq M \unitmeasure{\vdim} s^\vdim,
	\end{gather*}
	whenever $S \in \grass{\adim}{\vdim}$ with $| \project{S} - R(a) | <
	\delta$.
	
	Define $A$ to be the set of all $z \in \oball{a}{s} \cap \spt \| V \|$
	such that
	\begin{gather*}
		\measureball{\| \delta V \|}{\cball{z}{t}} \leq ( 2
		\varepsilon / 3 ) \| V \| ( \cball{z}{t} )^{1-1/\vdim}, \\
		\tint{\cball{z}{t}}{} | R ( \xi ) - R ( a ) | \ud \| V \| \xi
		\leq ( 2 \varepsilon / 3 ) \measureball{\| V \|}{\cball{z}{t}}
	\end{gather*}
	whenever $0 < t < 2s$,
	\begin{gather*}
		O = \classification{\orthproj{\adim}{\vdim}}{\pi}{ | \pi^\ast
		\circ \pi - R(a) | < \inf \{ \delta, \varepsilon/ 3 \} }, \\
		W = \classification{\oball{a}{s} \cap X_i \cap Y_i}{w}{ | R(w)
		- R (a) | \leq \varepsilon / 3}, \quad Z = W \cap A \cap J
		  \without N
	\end{gather*}
	where $N$ is the set of all $w \in W$ such that one of the following
	three conditions is violated
	\begin{gather*}
		w \in P', \quad \density^\vdim ( \mu, w ) = 0, \quad \lim_{t
		\to 0+} t^{-\vdim} \tint{\cball{w}{t}}{} | \mathbf{h} (V;\xi) -
		\mathbf{h} ( V; w ) | \ud \| V \| \xi = 0.
	\end{gather*}
	Note $\| V \| ( N ) = 0$ by \cite[2.9.10,\,11]{MR41:1976}.

	Now, fix $\pi_1 \in O$, $S = \im \pi_1^\ast$ and choose $\pi_2 \in
	\orthproj{\adim}{\codim}$ with $\pi_2 \circ \pi_1^\ast = 0$. The proof
	will be concluded by showing $\density^\vdim ( \| V \| \restrict Z , a
	) = Q$ and constructing $g$, $G$, $K$ and $T$ with the asserted
	properties. For this purpose assume $a = 0$ and $\pi_1 = \pp$ and
	$\pi_2 = \qq$ using isometries and identifying $\rel^\adim \simeq
	\rel^\vdim \times \rel^\codim$. Define
	\begin{gather*}
		u (w) = (s - |w-a|)/2 \quad \text{for $w \in W$}
	\end{gather*}
	and note $u(w) > 0$. Moreover, define $B$, $f$ as in
	\ref{lemma:lipschitz_approximation} with $\delta$
	replaced by $\varepsilon$ and whenever $w \in W$ and $0
	< t \leq u(w)$ define $B_{w,t}$, $C_{w,t}$ and $D_{w,t}$ as in
	\ref{lemma:lipschitz_approximation}\,\eqref{item:lipschitz_approximation:estimate_b}\,\eqref{item:lipschitz_approximation:pde}
	with additionally $a$, $s$ replaced by $w$, $t$. Since $| \project{S}
	- R (a) | \leq \varepsilon/3$ and $Z \subset
	\classification{A}{z}{\density^\vdim ( \| V \|, z ) = Q}$, one
	  infers from
	\ref{lemma:lipschitz_approximation}\,\eqref{item:lipschitz_approximation:para}
	that $Z \subset \graph_Q f$ and
	\begin{gather*}
		\density^0 ( \| f ( \pp (z)) \|, \qq (z) ) = Q, \quad
		( \pp^\ast + \qq^\ast \circ \boldsymbol{\eta}_Q \circ f ) (
		\pp ( z ) ) = z
	\end{gather*}
	whenever $z \in Z$. Using Kirszbraun's theorem (cf.
	\cite[2.10.43]{MR41:1976}) one extends $\boldsymbol{\eta}_Q \circ f$
	to a function $g : \rel^\vdim \to \rel^\codim$ such that
	\begin{gather*}
		\Lip g = \Lip ( \boldsymbol{\eta}_Q \circ f )
	\end{gather*}
	and defining $G = \pp^\ast + \qq^\ast \circ g$, $K = \pp \lIm Z
	\rIm$ and $T \in \mathscr{D}' ( \rel^\vdim, \rel^\codim )$ by
	\begin{gather*}
		T ( \theta ) = \tint{}{} \big < D \theta (x), D \Psi_0^\S ( D
		g (x) ) \big > \ud \mathscr{L}^\vdim x \quad \text{for
		$\theta \in \mathscr{D} ( \rel^\vdim, \rel^\codim )$},
	\end{gather*}
	the properties \eqref{item:great_approx:def},
	\eqref{item:great_approx:lip} and \eqref{item:great_approx:defT} are
	evident noting
	\ref{lemma:lipschitz_approximation}\,\eqref{item:lipschitz_approximation:lip}.

	Next, it will be shown
	\begin{gather*}
		B_{w,t} \subset \oball{a}{s} \cap ( \spt \| V \| ) \without A
		\subset C_i \cup D_i (w)
	\end{gather*}
	whenever $w \in W$, $0 < t \leq u(w)$. The first inclusion is readily
	verified noting $| \project{S} - R (a) | \leq \varepsilon/3$. If $z \in \oball{a}{s} \cap (
	\spt \| V \| ) \without A$, then
	\begin{gather*}
		\text{either} \quad \measureball{\| \delta V
		\|}{\cball{z}{t}} > (2\varepsilon/3) \, \| V \| ( \cball{z}{t}
		)^{1-1/\vdim} \quad \text{for some $0 < t < 2s$}, \\
		\text{or} \quad {\textstyle\int_{\cball{z}{t}}} | R ( \xi )
		- R(a) | \ud \| V \| \xi > (2\varepsilon/3) \,
		\measureball{\| V \|}{\cball{z}{t}} \quad \text{for some $0 <
		t < 2s$}.
	\end{gather*}
	In the first case, this implies $z \in C_i$, in the second case,
	\begin{gather*}
		\begin{aligned}
			& (2\varepsilon/3) \, \measureball{\| V
			\|}{\cball{z}{t}} < {\textstyle\int_{\cball{z}{t}}}
			| R ( \xi ) - R(a) | \ud \| V \| \xi \\
			& \qquad \leq {\textstyle\int_{\cball{z}{t}}} | R
			(\xi) - R (w) | \ud \| V \| \xi + | R (a) - R(w) |
			\, \measureball{\| V \|}{\cball{z}{t}},
		\end{aligned} \\
		\begin{aligned}
			( \varepsilon / 3 ) \,
			\measureball{\| V \|}{\cball{z}{t}} & <
			{\textstyle\int_{\cball{z}{t}}} | R ( \xi ) - R (w)
			| \ud \| V \| \xi \\
			& \leq \| V \| ( \cball{z}{t} )^{1-1/q} \big (
			{\textstyle\int_{\cball{z}{t}}} | R(\xi)-R(w)|^q \ud
			\| V \| \xi \big )^{1/q},
		\end{aligned}
	\end{gather*}
	hence $z \in D_i (w)$, and the second inclusion and hence the claim
	are proven. The inclusions imply the \emph{density estimate}
	\begin{gather*}
		\density^{\vdim + \alpha q} ( \| V \| \restrict B, w ) =
		\density^{\vdim + \alpha q} ( \| V \| \restrict ( U \without
		A), w ) = 0 \quad \text{whenever $w \in W$}.
	\end{gather*}
	Noting $a \in W$ and $\density^\vdim ( \| V \| \restrict U \without (W
	\cap J), a ) = 0$, one infers in particular
	\begin{gather*}
		\density^\vdim ( \| V \| \restrict U \without Z, a ) =
		0, \quad \density^\vdim ( \| V \| \restrict Z, a ) = Q
	\end{gather*}
	and it remains to verify that $g$, $G$, $K$, and $T$ satisfy
	\eqref{item:great_approx:K}, \eqref{item:great_approx:pde} and
	\eqref{item:great_approx:diff}.

	In preparation to this, the following \emph{tilt estimate} will be
	shown with $\Delta_1 = ( 1 + L^2 )^{1/2} ( 1 - L^2 )^{-1/2}
	\vdim^{1/2}$
	\begin{gather*}
		\begin{aligned}
			& Q^{-1/2} \big ( \tint{\cball{\pp (z)}{t} \cap \dmn
			f}{} | \ap A f (x) \aplus ( - \tau ) |^r \ud
			\mathscr{L}^\vdim x \big )^{1/r} \\
			& \qquad \leq \Delta_1 \big ( 
			\tint{\cylinder{S}{z}{t}{\delta_4 t}}{} |
			R ( \xi ) - \project{\tau} |^r \ud \| V \| \xi \big
			)^{1/r}
		\end{aligned}
	\end{gather*}
	whenever $1 \leq r < \infty$, $z \in Z$, $0 < t \leq u(z)$, $\tau \in
	\Hom ( \rel^\vdim, \rel^\codim )$ with $\| \tau \| \leq L$ (here the
	identification $\tau \subset \rel^\vdim \times \rel^\codim \simeq
	\rel^\adim$ is used); in fact, recalling $L \leq 1/8$ and $z \in
	\graph_Q f$, one notes
	\begin{gather*}
		\graph_Q f | \cball{\pp (z)}{t} \subset
		\cylinder{S}{z}{t}{\delta_4t} \subset \cylinder{S}{a}{s}{s},
	\end{gather*}
	hence for $0 < \gamma < \infty$
	\begin{gather*}
		\bigclassification{\cball{\pp (z)}{t)}}{x}{Q^{-1/2} | \ap A f
		(x) \aplus ( -\tau ) | > \gamma }
	\end{gather*}
	is $\mathscr{H}^\vdim$ almost contained in
	\begin{gather*}
		\pp \biglIm \classification{\cylinder{S}{z}{t}{\delta_4t}}{
		\xi}{ \Delta_1 | R ( \xi ) - \project{\tau} | > \gamma}
		\bigrIm
	\end{gather*}
	by
	\ref{lemma:lipschitz_approximation}\,\eqref{item:lipschitz_approximation:lip}\,\eqref{item:lipschitz_approximation:misc}
	and Allard \cite[8.9\,(5)]{MR0307015}. For $x \in K$, taking $z = G
	(x)$ and $\tau$ associated to $\im R (z)$, one infers, noting
	$\density^{\vdim + \alpha q} ( \mathscr{L}^\vdim \restrict \rel^\vdim
	\without \dmn f, x ) = 0$ by the density estimate for $B$ and
	\ref{lemma:lipschitz_approximation}\,\eqref{item:lipschitz_approximation:estimate_b}
	and $\Delta_1 \leq 2 \vdim^{1/2}$,
	\begin{gather*}
		\begin{aligned}
			& \limsup_{t \pluslim{0}} t^{-\beta-\vdim/r} \big (
			\tint{\cball{x}{t}}{} | Dg (\zeta) - \tau |^r \ud
			\mathscr{L}^\vdim \zeta \big)^{1/r} \\
			& \qquad \leq 2 \vdim^{1/2} \limsup_{t \pluslim{0}}
			t^{-\beta-\vdim/r} \big( \tint{\cball{z}{t}}{} |
			R(\xi) - R(z) |^r \ud \| V \| \xi \big )^{1/r}
		\end{aligned}
	\end{gather*}
	whenever $x \in K$, $0 < \beta \leq 1$, $1 \leq r < \infty$, and
	$\beta r \leq \alpha q$, hence in particular, taking $\beta = \alpha
	\inf \{ 1,q/r \}$ and noting that the right hand side in this case is
	finite by \cite[2.4.17]{MR41:1976} as $z \in P$,
	\begin{gather*}
		\lim_{t \to 0+} \big (
		{\textstyle\fint_{\cball{x}{t}}} | D g ( \zeta ) - \tau
		|^r \ud \mathscr{L}^\vdim \zeta \big )^{1/r} = 0 \quad
		\text{for $1 \leq r < \infty$}
	\end{gather*}
	and $g$ is differentiable at $x$ with $Dg(x) = \tau$ by the argument
	in \cite[Theorem 6.2.1]{MR1158660}. Since $Z \subset \im G$, $K$
	is $\mathscr{L}^\vdim$ measurable, hence \eqref{item:great_approx:K}
	and \eqref{item:great_approx:pde} are now proven and it remains to
	prove \eqref{item:great_approx:diff}.

	Choose $\eta \in \mathscr{D}^0 ( \rel^\codim )$ such that
	\begin{gather*}
		0 \leq \eta (y) \leq 1 \quad \text{for $y \in \rel^\codim$},
		\\
		\spt \eta \subset \oball{0}{1/4}, \quad \cball{0}{1/8} \subset
		\Int (\classification{\rel^\codim}{y}{\eta (y) = 1 })
	\end{gather*}
	and define $T_x$ for $x \in K$ as in \eqref{item:great_approx:diff}.
	Fix $x \in K$, let $z = G(x)$, note $\pp (z) = x$  and abbreviate
	\begin{gather*}
		\theta_t = t^{-\vdim} \theta \circ
		\mutau{t}{\pp(z)}, \quad \eta_t = \eta
		\circ \mutau{t}{\qq(z)}
	\end{gather*}
	whenever $0 < t \leq u(z)$ and $\theta \in \mathscr{D}
	( \rel^\vdim , \rel^\codim )$. The
	remaining estimate will carried out by showing that
	\begin{gather*}
		Q T_x ( \theta_t ) - ( \delta V ) ( (
		\eta_t \circ \qq ) \cdot ( \qq^\ast \circ
		\theta_t \circ \pp ) ), \\
		( \delta V ) ( ( \eta_t \circ \qq ) \cdot (
		\qq^\ast \circ \theta_t \circ \pp ) ) - Q
		\tint{}{} \big < D \theta_t (\zeta), D\Psi_0^\S ( Dg(\zeta) )
		\big > \ud \mathscr{L}^\vdim \zeta 
	\end{gather*}
	both tend to $0$ as $t \to 0+$ uniformly with respect to $\theta
	\in \mathscr{D} ( \rel^\vdim, \rel^\codim )$ such that $\spt \theta
	\subset \oball{0}{1}$ and $\norm{D\theta}{\infty}{0,1} \leq 1$.

	To prove the first estimate, one notes that the conditions
	$\density^{\vdim-1} ( \| \delta V \|, z ) = 0$, $\density^\vdim ( \| V
	\|, z ) = Q$ and $z \in P$ imply, for example using Allard
	\cite[6.4,\,5]{MR0307015} and \cite[3.1]{snulmenn.poincare},
	\begin{gather*}
		t^{-\vdim} \tint{}{} \phi ( t^{-1} (\xi-z), \im R
		( \xi) ) \ud \| V \| \xi \to Q \tint{\im R(z)}{} \phi ( \xi,
		\im R ( z ) ) \ud \mathscr{H}^\vdim \xi
	\end{gather*}
	as $t \to 0+$ whenever $\phi \in \ccspace{\rel^\adim \times
	\grass{\adim}{\vdim}}$. Since also, noting
	\begin{gather*}
		( \eta_t \circ \qq ) \cdot ( \qq^\ast \circ
		\theta_t \circ \pp ) = t^{-\vdim} \big ( (
		\eta \circ \qq ) \cdot ( \qq^\ast \circ \theta \circ \pp )
		\big ) \circ \mutau{t}{z}, \\
		\cylind{T}{0}{1} \cap \Tan^\vdim ( \| V \|, z ) \subset
		\cylinder{T}{0}{1}{1/8}
	\end{gather*}
	as $L \leq 1/8$ and $z \in \graph_Q f$, one readily uses the
	conditions on $\delta V$ and $\mathbf{h} ( V ; \cdot )$ imposed by the
	fact $z \notin N$ to infer
	\begin{gather*}
		\begin{aligned}
			& \lim_{t \pluslim{0}} ( \delta V ) ( (
			\eta_t \circ \qq ) \cdot ( \qq^\ast
			\circ \theta_t \circ \pp) ) \\
			& \qquad = - Q \tint{\im R ( z)}{} \mathbf{h} ( V ; z )
			\bullet ( \eta \circ \qq ) (\xi) ( \qq^\ast
			\circ \theta \circ \pp ) (\xi) \ud \mathscr{H}^\vdim
			\xi  \\
			& \qquad = - Q {\textstyle\int} \Psi_0^\S ( Dg( x )
			) \mathbf{h} ( V ; z ) \bullet ( \qq^\ast \circ
			\theta ) (\zeta) \ud \mathscr{L}^\vdim \zeta = Q T_x (
			\theta_t )
		\end{aligned}
	\end{gather*}
	and the convergence is uniform with respect to $\theta \in \mathscr{D}
	( \rel^\vdim, \rel^\codim )$ such that $\spt \theta \subset
	\oball{0}{1}$ and $\norm{D\theta}{\infty}{0,1} \leq 1$ as this family
	of functions is compact with respect to $\norm{\cdot}{\infty}{0,1}$ by
	\cite[2.10.21]{MR41:1976} and $\density^{\ast \vdim} ( \| \delta V \|,
	z ) < \infty$.

	To prove the second estimate, define
	\begin{gather*}
		\gamma_1 = \sup \| D^2 \Psi_0^\S \| \lIm
		\cball{0}{\vdim^{1/2}L}\rIm, \quad \gamma_2 = \Lip \big ( D^2
		\Psi_0^\S | \cball{0}{3\vdim^{1/2} L } \big).
	\end{gather*}
	Apply
	\ref{lemma:lipschitz_approximation}\,\eqref{item:lipschitz_approximation:pde}
	with $\tau = Dg(x)$ and $0 < t \leq u(z)$ to obtain
	\begin{align*}
		& \big | Q \tint{}{} \big < D \theta_t (x),
		D\Psi_0^\S (
		Dg(x) ) \big > \ud \mathscr{L}^\vdim x - ( \delta V ) ( (
		\eta_t \circ \qq ) \cdot ( \qq^\ast \circ
		\theta_t \circ \pp ) ) \big | \\
		& \qquad \leq
		\begin{aligned}[t]
			& \gamma_1 Q \vdim^{1/2} L
			{\textstyle\int_{C_{z,t}}} | D
			\theta_t | \ud \mathscr{L}^\vdim \\
			& + \gamma_2 {\textstyle\int_{E_{z,t}\without
			C_{z,t}}} | D \theta_t (\zeta) || \ap A
			f(\zeta) \aplus ( - Dg(x)) |^2 \ud \mathscr{L}^\vdim
			\zeta \\
			& + \vdim^{1/2} {\textstyle\int_{D_{z,t}}} | D (
			( \eta_t \circ \qq ) \cdot ( \qq^\ast \circ
			\theta_t \circ \pp ) ) | \ud \| V \|.
		\end{aligned}
	\end{align*}
	The first and the third summand on the right hand side may be
	estimated by use of
	\ref{lemma:lipschitz_approximation}\,\eqref{item:lipschitz_approximation:estimate_b}
	as follows
	\begin{gather*}
		{\textstyle\int_{C_{z,t}}} | D \theta_t | \ud
		\mathscr{L}^\vdim \leq t^{-\vdim-1} \mathscr{L}^\vdim (
		C_{z,t} ) \leq \Delta_2 t^{-\vdim-1} \| V \| (
		B_{z,t} ) , \\
		\begin{aligned}
			& {\textstyle\int_{D_{z,t}}} | D ( (
			\eta_t \circ \qq ) \cdot ( \qq^\ast \circ
			\theta_t \circ \pp )) | \ud \| V \| \\
			& \quad \leq t^{-\vdim-1} ( 1 + \norm{D
			\eta}{\infty}{0,1}) \| V \| ( D_{z,t} ) \leq
			\Delta_2 t^{-\vdim-1} ( 1 + \norm{D
			\eta}{\infty}{0,1} ) \| V \| (B_{z,t})
		\end{aligned}
	\end{gather*}
	where $\Delta_2 =
	\Gamma_{\ref{lemma:lipschitz_approximation}\eqref{item:lipschitz_approximation:estimate_b}}
	( Q, \vdim )$, hence the density estimate for $B$ applies recalling
	$\alpha q \geq 1$. To estimate the remaining summand, one computes
	\begin{gather*}
		\begin{aligned}
			& {\textstyle\int_{E_{z,t} \without
			C_{z,t}}} | D \theta_t (\zeta) | | \ap A f
			(\zeta) \aplus (-Dg(x)) |^2 \ud \mathscr{L}^\vdim
			\zeta \\
			& \qquad \leq t^{-1-\vdim} {\textstyle\int_{
			\cball{x}{t} \cap \dmn f}} | \ap A f (\zeta)
			\aplus ( -Dg (x)) |^2 \ud \mathscr{L}^\vdim \zeta,
		\end{aligned}
	\end{gather*}
	uses the tilt estimate and recalls $z \in P'$.
\end{proof}
\begin{remark} \label{remark:dnorm}
	It would significantly simplify the treatment in
	\ref{lemma:l1_estimate}--\ref{lemma:arbi} if one could obtain an
	estimate in $\dnorm{\cdot}{r}{a,s}$ in \eqref{item:great_approx:diff}
	for some $r > 1$. However, in this case it seems to be unclear how to
	control the integral over $D_{z,t}$ in the last paragraph as this set
	may contain arbitrarily steep parts of the varifold, see Brakke's
	example in \cite[6.1]{MR485012}.
\end{remark}
\begin{miniremark} \label{miniremark:perp}
	If $f : \rel^\vdim \to \rel^\codim$ is a linear map, $v \in
	\rel^\adim$ is orthogonal to $\im ( \pp^\ast + \qq^\ast \circ f )$
	then $v \in \ker ( \pp^\ast + \qq^\ast \circ f )^\ast$, $\pp (v) = - (
	f^\ast \circ \qq ) (v)$ and
	\begin{gather*}
		( \qq^\ast - \pp^\ast \circ f^\ast ) ( \qq (v) ) = v.
	\end{gather*}
\end{miniremark}
\begin{theorem} \label{thm:main_theorem}
	Suppose $\vdim, \adim \in \nat$, $\vdim < \adim$, $U$ is an open
	subset of $\rel^\adim$, $V \in \IVar_\vdim ( U )$ and $\| \delta V \|$
	is a Radon measure.

	Then there exists a countable collection $C$ of $\vdim$ dimensional
	submanifolds of $\rel^\adim$ of class $\class{2}$ such that $\| V \| (
	U \without \bigcup C ) = 0$ and each member $M$ of $C$ satisfies
	\begin{gather*}
		\mathbf{h} ( V ; z ) = \mathbf{h} ( M ; z ) \quad \text{for
		$\| V \|$ almost all $z \in U \cap M$}.
	\end{gather*}
\end{theorem}
\begin{proof}
	First, note that for $\| V\|$ almost all $z \in U$ there holds
	$\Tan^\vdim ( \| V \|, z ) \in \grass{\adim}{\vdim}$ and
	\begin{gather*}
		\lim_{r \to 0+} r^{-1/2-\vdim/2} \big ( \tint{\cball{z}{r}
		\times \grass{\adim}{\vdim}}{} | \project{S} -
		\project{\Tan^\vdim ( \| V \|, z )} |^2 \ud V (\xi,S) \big
		)^{1/2} = 0
	\end{gather*}
	by Brakke \cite[5.7,\,5]{MR485012} or \cite[8.6]{snulmenn:decay.v2}.
	Let $\Psi$ denote the area integrand, abbreviate $\Phi = \Psi_0^\S$
	and note $D^2 \Phi (0) = \Upsilon$ with $\Upsilon$ as in
	\ref{miniremark:situation_F} by \cite[5.1.9]{MR41:1976}.  Define
	$\varepsilon = \varepsilon_{\ref{lemma:arbi}} ( \vdim, \adim)$,
	$\Delta = \Gamma_{\ref{lemma:cutoff}} ( \vdim ( \codim ), 2 )$, $s =
	\varepsilon/\Delta$ and choose $0 < \delta < \infty$ such that
	\begin{gather*}
		\| D^2 \Phi ( \sigma ) - D^2 \Phi (0) \| \leq s
		\quad \text{whenever $\sigma \in \Hom ( \rel^\vdim ,
		\rel^\codim ) \cap \cball{0}{\delta}$}.
	\end{gather*}
	Applying \ref{lemma:cutoff} with $H$, $k$, $l$, $a$ replaced by
	$\Hom ( \rel^\vdim, \rel^\codim )$, $2$, $3$, $0$, one
	obtains $F : \Hom ( \rel^\vdim, \rel^\codim ) \to \rel$ of class $3$
	such that
	\begin{gather*}
		D^i F ( \sigma ) = D^i \Phi ( \sigma ) \quad \text{for $i
		= \{0,1,2\}$, $\sigma \in \Hom ( \rel^\vdim, \rel^\codim )
		\cap \cball{0}{\delta/2}$}, \\
		\| D^2 F ( \sigma ) - D^2 \Phi (0) \| \leq \Delta s =
		\varepsilon \quad \text{whenever $\sigma \in \Hom ( \rel^\vdim
		, \rel^\codim )$}, \\
		\text{$D^3 F$ has compact support},
	\end{gather*}
	hence $\Lip D^2 F < \infty$. Define $L = \vdim^{-1/2} \delta/2$ and
	apply \ref{lemma:great_approx} with $p$, $q$, $\alpha$ replaced by
	$1$, $2$, $1/2$ to obtain $P$ and $H$ with the properties listed
	there. Fix $Z \in H$ and take $\pi_1 \in O$ and $\pi_2$, $g$, $G$, $K$
	as in \ref{lemma:great_approx} to infer from \ref{lemma:arbi},
	\ref{remark:arbi} and
	\ref{lemma:great_approx}\,\eqref{item:great_approx:diff}, noting
	\ref{lemma:great_approx}\,\eqref{item:great_approx:pde} with $\beta
	=1/2$ and $r=2$, the existence a sequence of functions $u_i :
	\rel^\vdim \to \rel^\codim$ of class $\class{2}$ such that with $A_i =
	\classification{K}{x}{g(x) =u_i (x)}$ for $i \in \nat$
	\begin{gather*}
		\left < D^2 u_i (x), C_F ( D u_i (x) ) \right > = \Phi (
		Du_i (x) ) \pi_2 ( \mathbf{h} ( V ; G (x) ) )
	\end{gather*}
	for $\mathscr{L}^\vdim$ almost all $x \in A_i$. Defining $M_i = \im (
	\pi_1^\ast + \pi_2^\ast \circ u_i )$ and noting
	\begin{gather*}
		\left < D^2 u_i (x) , C_{\Phi} ( D u_i (x) ) \right > =
		\Phi ( D u_i (x) ) \pi_2 ( \mathbf{h} ( M_i ; ( \pi_1^\ast +
		\pi_2^\ast \circ u_i ) ( x) ))
	\end{gather*}
	for $x \in \rel^\vdim$ where $C_\Phi$ is as in
	\ref{miniremark:situation_F} and
	\begin{gather*}
		C_{\Phi} ( \sigma ) = C_F ( \sigma ) \quad \text{for
		$\sigma \in \Hom ( \rel^\vdim, \rel^\codim ) \cap
		\cball{0}{\delta/2}$}, \\
		|Du_i (x) | = | D g( x) | \leq L \vdim^{1/2} = \delta/2
		\quad \text{for $\mathscr{L}^\vdim$ almost all $x \in A_i$}
	\end{gather*}
	by \ref{lemma:great_approx}\,\eqref{item:great_approx:lip}, one
	concludes
	\begin{gather*}
		\pi_2 ( \mathbf{h} ( V ; G (x) ) ) = \pi_2 ( \mathbf{h} ( M_i
		; G (x) ) ) \quad \text{for $\mathscr{L}^\vdim$ almost all $x
		\in A_i$},
	\end{gather*}
	hence by \ref{miniremark:perp}, since $\mathbf{h} (V; z ) \in
	\Nor^\vdim ( \| V \|, z )$ for $\| V \|$ almost all $z$ by Brakke
	\cite[5.8]{MR485012},
	\begin{gather*}
		\mathbf{h} ( V ; G (x) ) = \mathbf{h} ( M_i
		; G (x) ) \quad \text{for $\mathscr{L}^\vdim$ almost all $x
		\in A_i$}.
	\end{gather*}
	Finally, recall $\| V \| ( U \without P ) = 0$.
\end{proof}
\begin{remark} \label{remark:brakke}
	One could also prove Brakke \cite[5.8]{MR485012} instead of using it.
	Since the proof then still yields a collection $C$ with all properties
	except of the last one, one can define a $\| V \|$
	measurable function $h$ such that for $\| V\|$ almost all $z \in U$
	there holds $h(z) = \mathbf{h} ( M ; z )$ whenever $z \in U \cap M$
	and $M \in C$.  Following the above proof, one obtains
	\begin{gather*}
		\pi_2 ( \mathbf{h} ( V ; G (x) ) ) = \pi_2 ( h ( G(x) ) )
		\quad \text{for $\mathscr{L}^\vdim$ almost all $x
		\in A_i$}
	\end{gather*}
	whenever $\pi_1 \in O$, $\pi_2 \in \orthproj{\adim}{\codim}$ with
	$\pi_2 \circ \pi_1^\ast = 0$, and, as $O$ is open, this suffices to
	conclude
	\begin{gather*}
		\mathbf{h} ( V ; G (x) ) = h (G(x)) \in \Nor^\vdim ( \| V \|,
		G (x) ) \quad \text{for $\mathscr{L}^\vdim$ almost all $x \in
		A_i$}.
	\end{gather*}
\end{remark}
\begin{remark} \label{remark:hutchinson}
	Noting \cite[2.10.19\,(4)]{MR41:1976}, one infers that the function
	mapping $\| V \|$ almost all $z$ onto $\project{\Tan^\vdim ( \| V \|,
	z )} \in \Hom ( \rel^\adim, \rel^\adim )$ is $( \| V \|, \vdim )$
	approximately differentiable at $\| V \|$ almost all $z$.

	Therefore, combining \ref{thm:main_theorem} with Mantegazza
	\cite[Remark 3.9, Theorem 5.4]{MR1412686}, one obtains the following
	proposition on curvature varifolds with boundary in the sense of
	Mantegazza \cite[Definition 3.1]{MR1412686}: \emph{If $V$ is a
	curvature varifold with boundary in an open subset $U$ of $\rel^\adim$
	then then there exists a countable collection $C$ of $\vdim$
	dimensional submanifolds of $\rel^\adim$ of class $\class{2}$ such
	that $\| V \| ( U \without \bigcup C ) = 0$ and such that for each
	member $M$ of $C$ the second fundamenal forms of $V$ and $M$ agree at
	$\| V \|$ almost every $z \in U \cap M$.} Clearly, this includes
	curvature varifolds in the sense of Hutchinson \cite[5.2.3]{MR825628}.
\end{remark}
\section{Applications to decay rates of tilt-excess for integral varifolds}
\label{sec:applications}
The present section discusses some consequences of \ref{thm:main_theorem} in
terms of decay and differentiability of tilt quantities.
\begin{lemma} \label{lemma:prep_tilt}
	Suppose $\vdim, \adim, Q \in \nat$, $\vdim < \adim$, either $p = \vdim
	= 1$ or $1 < p < \vdim = 2$ or $1 \leq p < \vdim > 2$ and $\frac{\vdim
	p}{\vdim-p} = 2$, $0 < \delta \leq 1$, and $1 \leq M < \infty$.

	Then there exist positive, finite numbers $\varepsilon$ and $\Gamma$
	with the following property.

	If $a \in \rel^\adim$, $0 < r < \infty$, $V \in \IVar_\vdim (
	\oball{a}{6r} )$, $\psi$ and $p$ are related to $V$ as in
	\ref{miniremark:situation}, $T \in \grass{\adim}{\vdim}$, $Z$ is a $\|
	V \|$ measurable subset of $\cylinder{T}{a}{r}{3r}$,
	\begin{gather*}
		( Q - 1/2 ) \unitmeasure{\vdim} r^\vdim \leq \| V \| (
		\cylinder{T}{a}{r}{3r} ) \leq ( Q + 1/2 ) \unitmeasure{\vdim}
		r^\vdim, \\
		\| V \| ( \cylinder{T}{a}{r}{4r} \without
		\cylinder{T}{a}{r}{r} ) \leq ( 1/2 ) \unitmeasure{\vdim}
		r^\vdim, \\
		\measureball{\| V \|}{\oball{a}{6r}} \leq M
		\unitmeasure{\vdim} r^\vdim, \quad
		\| V \| ( \cylinder{T}{a}{r/2}{r/2} ) \geq ( Q - 1/4 )
		\unitmeasure{\vdim} (r/2)^\vdim, \\
		\| V \| ( \cylinder{T}{a}{r}{3r} \without Z ) \leq \varepsilon
		\unitmeasure{\vdim} r^\vdim, \quad
		\big ( \tint{}{} | \project{S} - \project{T} |^2 \ud V (z,S)
		\big )^{1/2} \leq \varepsilon r^{\vdim/2},
	\end{gather*}
	then
	\begin{align*}
			& \big ( r^{-\vdim} \tint{\cylinder{T}{a}{r/4}{r/4}
			\times \grass{\adim}{\vdim}}{} | \project{S} -
			\project{T} |^2 \ud V (z,S) \big )^{1/2} \\
			& \qquad \leq
			\begin{aligned}[t]
				& \delta \big ( r^{-\vdim}
				\tint{\cylinder{T}{a}{r}{r} \times
				\grass{\adim}{\vdim}}{} | \project{S} -
				\project{T} |^2 \ud V (z,S) \big )^{1/2} \\
				& + \Gamma \big ( r^{-\vdim-1} \tint{Z}{}
				\dist (z-a,T) \ud \| V \| z + r^{1-\vdim/p}
				\psi ( \oball{a}{6r} )^{1/p} \big ).
			\end{aligned}
	\end{align*}
\end{lemma}
\begin{proof}
	See \cite[7.5]{snulmenn:decay.v2}.
\end{proof}
\begin{theorem} \label{thm:decay_rates}
	Suppose $\vdim$, $\adim$, $p$, $U$, and $V$ are as in
	\ref{miniremark:situation}, $V \in \IVar_\vdim ( U )$ and
	\begin{gather*}
		\phi ( a,r,T )= \big ( r^{-\vdim} \tint{\oball{a}{r} \times
		\grass{\adim}{\vdim}}{} | \project{S} - \project{T} |^2 \ud V
		(z,S) \big )^{1/2}
	\end{gather*}
	whenever $a \in \rel^\adim$, $0 < r < \infty$, $\oball{a}{r} \subset
	U$, and $T \in \grass{\adim}{\vdim}$.

	Then the following two statements hold:
	\begin{enumerate}
		\item \label{item:decay_rates:gen} If either $\vdim = 2$ and
		$0 < \tau < 1$ or $\sup \{ 2, p \} < \vdim$ and $\tau =
		\frac{\vdim p}{2( \vdim-p)} < 1$ then
		\begin{gather*}
			\lim_{r \to 0+} r^{-\tau} \phi (a,r,T) = 0 \quad
			\text{for $V$ almost all $(a,T) \in U \times
			\grass{\adim}{\vdim}$}.
		\end{gather*}
		\item \label{item:decay_rates:log} If either $\vdim = 1$ or
		$\vdim = 2$ and $p > 1$ or $\vdim > 2$ and $p \geq 2 \vdim / (
		\vdim+2 )$ then
		\begin{gather*}
			\limsup_{r \to 0+} r^{-1} \phi (a,r,T) < \infty \quad
			\text{for $V$ almost all $(a,T) \in U \times
			\grass{\adim}{\vdim}$}.
		\end{gather*}
	\end{enumerate}
\end{theorem}
\begin{proof} [Proof of \eqref{item:decay_rates:gen}]
	From \ref{thm:main_theorem} one obtains a sequence of maps
	$R_i : U \to \Hom ( \rel^\adim, \rel^\adim )$ of class $\class{1}$
	such that the sets $A_i = \classification{U}{z}{ R_i (z) =
	\project{\Tan^\vdim ( \| V \|, z )}}$ cover $\| V \|$ almost all of
	$U$. By \cite[8.6]{snulmenn:decay.v2} and
	\cite[3.7\,(i)]{snulmenn.isoperimetric} one infers
	\begin{gather*}
		\lim_{r \to 0+} r^{-\tau-\vdim/2} \big ( \tint{\cball{z}{r}
		\times \grass{\adim}{\vdim}}{} | R_i (z) - \project{S} |^2 \ud
		V (\xi,S) \big )^{1/2} = 0
	\end{gather*}
	for $\| V \|$ almost all $z \in A_i$ and the conclusion follows.
\end{proof}
\begin{proof} [Proof of \eqref{item:decay_rates:log}]
	Assume that either $p = \vdim = 1$ or $1 < p < \vdim = 2$ or $1 \leq p
	< \vdim > 2$ and $\frac{\vdim p}{\vdim-p} = 2$. Choose $C$ as in
	\ref{thm:main_theorem}. Then by \ref{thm:main_theorem} and
	\cite[2.10.19\,(4), 2.9.5]{MR41:1976} for $\| V \|$ almost all $a \in
	U$ there holds for some $Q \in \nat$, $T \in \grass{\adim}{\vdim}$ and
	some $M \in C$
	\begin{gather*}
		T = \Tan ( M, a ), \quad \density^\vdim ( \| V \| \restrict
		U \without M, a ) = 0, \\
		\limsup_{r \to 0+} r^{-\vdim/p} \psi ( \cball{a}{r} )^{1/p} <
		\infty, \\
		r^{-\vdim} \tint{}{} \phi ( r^{-1} (z-a), S ) \ud V (z,S)
		\to Q \tint{T}{} \phi ( z, T ) \ud \mathscr{H}^\vdim z
		\quad \text{as $r \to 0+$}
	\end{gather*}
	whenever $\phi \in \ccspace{\rel^\adim \times \grass{\adim}{\vdim}}$.
	Note that
	\begin{gather*}
		\limsup_{r \to 0+} r^{-\vdim-2} \tint{\cylinder{T}{a}{r}{3r}
		\cap M}{} \dist ( z - a, T ) \ud \| V \| z < \infty
	\end{gather*}
	as $M$ is submanifold of class $\class{2}$. It follows with $\delta =
	2^{-\vdim-3}$, $\Delta_1 = 7^\vdim Q$ that there exist $0 < R <
	\infty$ and $0 \leq \gamma < \infty$ such that $\oball{a}{6R} \subset
	U$,
	\begin{gather*}
		r^{-\vdim-1} \tint{\cylinder{T}{a}{r}{3r} \cap M}{} \dist (
		z - a, T ) \ud \| V \| z + r^{1-\vdim/p} \psi (
		\oball{a}{6r} )^{1/p} \leq \gamma r
	\end{gather*}
	for $0 < r \leq R$, and $V$ satisfies the hypotheses of
	\ref{lemma:prep_tilt} for each $0 < r \leq R$ with $\varepsilon =
	\varepsilon_{\ref{lemma:prep_tilt}} ( \vdim, \adim, Q, p, \delta,
	\Delta_1 )$ and $M$, $Z$ replaced by $\Delta_1$,
	$\cylinder{T}{a}{r}{3r} \cap M$.  With $f ( r ) = r^{-\vdim/2} \big (
	\int_{\cylinder{T}{a}{r}{r} \times \grass{\adim}{\vdim}} | \project{S}
	- \project{T} |^2 \ud V ( z, S ) \big )^{1/2}$ for $0 < r \leq R$
	one defines
	\begin{gather*}
		\Delta_2 = \Gamma_{\ref{lemma:prep_tilt}} ( \vdim, \adim, Q,
		p, \delta, \Delta_1), \quad \Delta_3 = \sup \big \{
		2^{\vdim+3} \Delta_2 \gamma, 2^{\vdim+2} R^{-1} f (R) \big \},
	\end{gather*}
	one inductively infers from \ref{lemma:prep_tilt}
	\begin{gather*}
		f(r) \leq \Delta_3 r \quad \text{whenever $0 < r \leq R$};
	\end{gather*}
	in fact it holds for $R/4 \leq r \leq R$ and, provided it holds for
	$r$,
	\begin{gather*}
		f(r/4) \leq 2^\vdim ( \delta \Delta_3 r + \Delta_2 \gamma
		r ) \leq \Delta_3 (r/4)
	\end{gather*}
	by \ref{lemma:prep_tilt}. The conclusion is now evident.
\end{proof}
\begin{remark}
	Having \ref{thm:main_theorem} at one's disposal, the proof of
	\eqref{item:decay_rates:log} follows Sch\"atzle in \cite[Theorem
	3.1]{rsch:willmore} where the case $p \geq 2$ is treated. In extending
	the result to the present case, the main difference is the use of the
	coercive estimate in \cite[3.9]{snulmenn:decay.v2} in the proof of
	\ref{lemma:prep_tilt} replacing the use of corresponding estimate in
	Brakke \cite[5.5]{MR485012} (see also Allard \cite[8.13]{MR0307015}).
\end{remark}
\begin{remark} \label{remark:diff_optimal}
	For both parts the family of examples provided in
	\cite[1.2]{snulmenn.isoperimetric} shows that if $\vdim > 2$ then $p$
	cannot be replaced by any smaller number, see
	\cite[8.7]{snulmenn:decay.v2}.
\end{remark}
\begin{remark} \label{remark:l2_differentiability}
	In case of \eqref{item:decay_rates:log} combining this result with
	\cite[3.9]{snulmenn.isoperimetric}, one obtains
	\begin{gather*}
		{\textstyle\fint_{\cball{a}{r}}} (|R(z)-R(a)-\left< R(a)(z-a),
		\ap D R (a) \right > |/|z-a|)^2 \ud \| V \|
		z \to 0
	\end{gather*}
	as $r \to 0+$ for $\| V \|$ almost all $a$ where $R(z) =
	\project{\Tan^\vdim ( \| V \|, z )}$ and the approximate differential
	is taken with respect to $( \| V \|, \vdim )$.
\end{remark}
\begin{remark}
	Clearly, one can also obtain decay results for height quantities from
	this result by use of \cite[4.11]{snulmenn.poincare}.
\end{remark}
\appendix
\section{Lebesgue points for a distribution} \label{sec:app}
In this Appendix the part $q=1$ of Theorem \ref{ithm:distribution} of the
Introduction is provided. Its purpose is to clarify the relations of the
sets $A_1$ and $A_2$ occurring in \ref{lemma:arbi}.
\begin{lemma} \label{lemma:rademacher_error}
	Suppose $\vdim,\adim \in \nat$, $\vdim < \adim$, $A$ is a closed
	subset of $\rel^\vdim$, $R \in \mathscr{D}' ( \rel^\vdim ,
	\rel^\codim)$, $\dist ( \spt R, A ) > 0$, $0 \leq \gamma < \infty$,
	and $0 < r < \infty$ such that
	\begin{gather*}
		\dnorm{R}{1}{x,\varrho} \leq \gamma \, \varrho^{\vdim+1} \quad
		\text{whenever $0 < \varrho < 5r$, $x \in A$}.
	\end{gather*}

	Then
	\begin{gather*}
		\dnorm{R}{1}{a,r} \leq \Gamma \, \gamma \, r \,
		\mathscr{L}^\vdim ( \cball{a}{4r} \without A ) \quad \text{for
		$a \in A$}
	\end{gather*}
	where $\Gamma$ is a positive, finite number depending only on $\vdim$.
\end{lemma}
\begin{proof}
	Assume $r \leq \frac{2}{9}$, let $a \in A$, $\theta \in \mathscr{D} (
	\rel^\vdim, \rel^\codim )$ with $\spt \theta \subset \oball{a}{r}$,
	choose $0 < \varepsilon \leq \min \{ r, \dist ( \spt R, A ) \}$,
	define
	\begin{gather*}
		B = \classification{\rel^\vdim}{x}{\dist ( x, \spt ( R
		\restrict \theta ) ) \leq \varepsilon / 2}
	\end{gather*}
	where $R \restrict \theta \in \mathscr{E}_0 ( \rel^\vdim )$ is defined
	by $( R \restrict \theta ) ( v ) = R ( v \theta )$ for $v \in
	\mathscr{E}^0 ( \rel^\vdim )$, and apply \cite[3.1.13]{MR41:1976} to
	obtain $S$, $v_s$, and $h$ with $\Phi = \{ \rel^\vdim \without A,
	\rel^\vdim \without B \}$; in particular $S$ is a countable subset of
	$\bigcup \Phi$,
	\begin{gather*}
		h (x) = {\textstyle\frac{1}{20}} \max \{ \min \{ 1, \dist
		(x,A) \}, \min \{ 1, \dist (x,B) \} \} \quad \text{for $x \in
		{\textstyle\bigcup} \Phi$}
	\end{gather*}
	and $v_s$ for $s \in S$ form a partition of unity on $\bigcup \Phi$
	with $\spt v_s \subset \cball{s}{10h(s)}$ for $s \in S$. Noting
	$\bigcup \Phi = \rel^\vdim$, one defines $T = \classification{S}{s}{B
	\cap \spt v_s \neq \emptyset}$ and infers
	\begin{gather*}
		{\textstyle\sum_{s \in S \without T}} v_s (x) = 0 \quad
		\text{for $x \in \rel^\vdim$ with $\dist (x, \spt ( R
		\restrict \theta ) ) < \varepsilon / 2$},
	\end{gather*}
	hence $( R \restrict \theta ) ( \sum_{s \in S \without T} v_s ) = 0$
	and
	\begin{gather*}
		R ( \theta ) = R \big ( ( {\textstyle\sum_{s \in T}} v_s )
		\theta \big ) = {\textstyle\sum_{s \in T}} R ( v_s \theta ).
	\end{gather*}

	Choose $\xi (s) \in A$ for each $s \in T$ such that $|s-\xi(s)| =
	\dist (s, A)$. If $s \in T$ then there exists $y \in B \cap \spt v_s
	\subset \cball{a}{r + \varepsilon / 2}$ and one observes
	\begin{gather*}
		\dist ( y, A ) \leq | y - a | \leq r + \varepsilon / 2 \leq (
		3 / 2 ) r \leq {\textstyle\frac{1}{3}} < 1, \quad h (y) =
		{\textstyle\frac{1}{20}} \dist ( y, A ), \\
		| s - y | \leq 10 h (s) \leq 10 h (y) +
		{\textstyle\frac{1}{2}} | s - y |, \quad | s - y | \leq 20 h
		(y) = \dist ( y, A ) \leq |y-a|, \\
		\dist ( s, A ) \leq | s - y | + \dist ( y, A ) \leq 2 \dist (
		y, A ) \leq 3r \leq {\textstyle\frac{2}{3}} < 1, \\
		B \cap \cball{s}{10 h(s)} \neq \emptyset, \quad
		{\textstyle\frac{1}{20}} \dist ( s, B ) \leq
		{\textstyle\frac{1}{2}} h(s), \quad 0 < h
		(s) = {\textstyle\frac{1}{20}} \dist ( s, A ), \\
		| s - \xi (s) | \leq | s - a | \leq | s - y | + | y - a | \leq
		2 r + \varepsilon \leq 3r \leq {\textstyle\frac{2}{3}}, \\
		\cball{s}{h(s)} \subset \cball{a}{4r} \without A.
	\end{gather*}
	Moreover, for any $x \in \cball{s}{10h(s)}$, $s \in T$
	\begin{gather*}
		| x - \xi (s) | \leq | x - s | + | s - \xi (s) | \leq ( 3/2 )
		| s - \xi (s) | < 5r, \\
		\spt v_s \subset \cball{\xi (s)}{ ( 3/2 ) | s - \xi (s) |
		}, \\
		\dist ( s, A ) \leq \dist ( x, A ) + | x - s | \leq \dist ( x,
		A ) + {\textstyle\frac{1}{2}} \dist ( s, A), \\
		| s - \xi (s) | = \dist ( s, A ) \leq 2 \dist ( x, A ), \\
		\dist ( x, A ) \leq \dist ( s, A ) + | x - s | \leq
		{\textstyle\frac{3}{2}} \dist ( s, A ) \leq 1, \\
		h (x) \geq {\textstyle\frac{1}{20}} \dist (  x, A ) \geq
		{\textstyle\frac{1}{40}} | s - \xi (s) |.
	\end{gather*}

	Using the estimates of the preceding paragraph and the estimates of
	$|Dv_s|$ given in \cite[3.1.13]{MR41:1976}, one infers for $s \in T$,
	since $\theta$ has compact support in $\oball{a}{r}$,
	\begin{gather*}
		\norm{( Dv_s ) \theta}{\infty}{a,r} \leq 40 \Delta | s - \xi
		(s) |^{-1} r \norm{D \theta}{\infty}{a,r}, \\
		\norm{D ( v_s \theta )}{\infty}{a,r} \leq 40 \Delta ( | s -
		\xi (s) |^{-1} r + 1 ) \norm{D \theta}{\infty}{a,r}
	\end{gather*}
	where $\Delta$ is a positive, finite number depending only on $\vdim$
	with $40 \Delta \geq 1$, hence
	\begin{gather*}
		\begin{aligned}
			| R ( v_s \theta ) | & \leq \gamma ( 3/2 )^{\vdim+1} |
			s - \xi (s) |^{\vdim+1} 40 \Delta ( | s - \xi (s) |^{-1}
			r + 1 ) \norm{D \theta}{\infty}{a,r} \\
			& = \gamma ( 3/2 )^{\vdim+1} 40 \Delta | s - \xi (s)
			|^\vdim ( r + | s - \xi (s) | ) \norm{D
			\theta}{\infty}{a,r} \\
			& \leq \gamma 160 \Delta (3/2)^{\vdim+1}
			\unitmeasure{\vdim}^{-1} (20)^\vdim r \,
			\mathscr{L}^\vdim ( \cball{s}{h(s)}) \, \norm{D
			\theta}{\infty}{a,r}.
		\end{aligned}
	\end{gather*}
	Recalling from \cite[3.1.13]{MR41:1976} that the family $\{
	\cball{s}{h(s)} \with s \in S \}$ is disjointed, one concludes
	\begin{gather*}
		| R ( \theta ) | \leq \Gamma \, \gamma \, r \,
		\mathscr{L}^\vdim ( \cball{a}{4r} \without A ) \norm{D
		\theta}{\infty}{a,r}
	\end{gather*}
	where $\Gamma = 8 ( 30 )^{\vdim+1} \Delta \unitmeasure{\vdim}^{-1}$.
\end{proof}
\begin{remark}
	Some ideas of the proof were taken from Calder\'on and Zygmund
\cite[Theorem 10]{MR0136849} and \cite[2.9.17]{MR41:1976}.
\end{remark}
\begin{theorem} \label{thm:distrib_diff}
	Suppose $\vdim,\adim \in \nat$, $\vdim < \adim$, $U$ is an open subset
	of $\rel^\vdim$, $T \in \mathscr{D}' ( U, \rel^\codim )$, and $A$
	denotes the set of all $a \in U$ such that
	\begin{gather*}
		\limsup_{r \pluslim{0}} r^{-1-\vdim} \dnorm{T}{1}{a,r} < \infty.
	\end{gather*}

	Then $A$ is a Borel set and for $\mathscr{L}^\vdim$ almost all $a \in A$
	there exists a unique constant distribution $T_a \in \mathscr{D}' ( U,
	\rel^\codim )$ such that
	\begin{gather*}
		\lim_{r \pluslim{0}} r^{-1-\vdim} \dnorm{T-T_a}{1}{a,r} = 0.
	\end{gather*}
	Moreover, $T_a$ depends $\mathscr{L}^\vdim \restrict A$ measurably on
	$a$.
\end{theorem}
\begin{proof}
	The conclusion is local, hence one may assume $\spt T$ to be compact
	and $U=\rel^\vdim$. Since $\dnorm{T}{1}{a,r}$ depends lower
	semicontinuously on $(a,r)$, the sets
	\begin{gather*}
		A_i =
		\classification{\rel^\vdim}{a}{\text{$\dnorm{T}{1}{a,r} \leq i
		\, r^{\vdim+1}$ for $0 < r < (10)/i$}}
	\end{gather*}
	defined for $i \in \nat$ are closed. Observing $A = \bigcup \{ A_i
	\with i \in \nat \}$, the conclusion will be shown to hold
	for $\mathscr{L}^\vdim$ almost all $a \in A_i$.

	Let $0 < \varepsilon < 5/i$, choose $\Phi \in \mathscr{D}^0 (
	\rel^\vdim )$ with $\int \Phi \ud \mathscr{L}^\vdim = 1$,
	$\spt \Phi \subset \oball{0}{1}$ and define $\Phi_\varepsilon
	(x) = \varepsilon^{-\vdim} \Phi (\varepsilon^{-1} x)$ for $x \in
	\rel^\vdim$,
	\begin{gather*}
		T_\varepsilon ( \theta ) = T ( \Phi_\varepsilon \ast \theta )
		= {\textstyle\int} f_\varepsilon \bullet \theta \ud
		\mathscr{L}^\vdim \quad \text{for $\theta \in \mathscr{D} (
		\rel^\vdim, \rel^\codim )$}
	\end{gather*}
	with $f_\varepsilon \in \mathscr{E} ( \rel^\vdim, \rel^\codim
	)$ given by
	\begin{gather*}
		z \bullet f_\varepsilon (x) = T_y ( \Phi_\varepsilon ( y-x ) z
		) \quad \text{whenever $x \in \rel^\vdim$ and $z \in
		\rel^\codim$},
	\end{gather*}
	see \cite[4.1.2]{MR41:1976}. Clearly $T_\varepsilon \to T$ as
	$\varepsilon \pluslim{0}$ and
	\begin{gather*}
		| f_\varepsilon (x) | \leq i 2^{\vdim+1} \norm{D
		\Phi}{\infty}{0,1} \quad \text{for $x \in \rel^\vdim$,
		$a \in A_i$ with $|x-a| \leq \varepsilon$}.
	\end{gather*}
	One defines $a_\varepsilon$ to be the characteristic function of
	$\classification{\rel^\vdim}{x}{\dist (x, A_i ) \leq \varepsilon}$
	and $S_\varepsilon, R_\varepsilon \in \mathscr{D}' ( \rel^\vdim,
	\rel^\codim )$ by
	\begin{gather*}
		S_\varepsilon ( \theta ) = {\textstyle\int} a_\varepsilon
		f_\varepsilon \bullet \theta \ud \mathscr{L}^\vdim \quad
		\text{for $\theta \in \mathscr{D} ( \rel^\vdim , \rel^\codim
		)$}, \quad R_\varepsilon = T_\varepsilon - S_\varepsilon.
	\end{gather*}

	Estimating for $a \in A_i$, $0 < \varrho < 5r < 5 / i$, $\theta \in
	\mathscr{D} ( \rel^\vdim, \rel^\codim )$ with $\spt \theta
	\subset \oball{a}{\varrho}$ and $\norm{D \theta}{\infty}{a,\varrho}
	\leq 1$
	\begin{gather*}
		\spt ( \Phi_\varepsilon \ast \theta ) \subset
		\oball{a}{\varepsilon+\varrho}, \qquad | T_\varepsilon (
		\theta ) | \leq i ( \varepsilon + \varrho )^{\vdim+1} \leq i
		2^{\vdim+1} \varrho^{\vdim+1} \quad \text{if $\varepsilon \leq
		\varrho$}, \\
		\eqclassification{\spt R_\varepsilon}{x}{\dist (x,A_i) <
		\varepsilon} = \emptyset, \quad R_\varepsilon ( \theta ) = 0
		\quad \text{if $\varepsilon > \varrho$}, \\
		| S_\varepsilon ( \theta ) | \leq \norm{a_\varepsilon
		f_\varepsilon}{\infty}{a,\varrho} \, \norm{\theta}
		{1}{a,\varrho} \leq i 2^{\vdim+1} \norm{D \Phi}{\infty}{0,1}
		\unitmeasure{\vdim} \varrho^{\vdim+1} \\
		\dnorm{R_\varepsilon}{1}{a,\varrho} \leq \gamma \,
		\varrho^{\vdim+1} \quad \text{with $\gamma =
		2^{\vdim+1} i \big ( 1 + \norm{D \Phi}{\infty}{0,1} \,
		\unitmeasure{\vdim} \big )$},
	\end{gather*}
	Now, \ref{lemma:rademacher_error} may be applied with $A$, $R$
	replaced by $A_i$, $R_\varepsilon$ to obtain
	\begin{gather*}
		\dnorm{R_\varepsilon}{1}{a,r} \leq \Gamma \, \gamma \, r \,
		\mathscr{L}^\vdim ( \cball{a}{4r} \without A_i ) \quad
		\text{for $0 < r < 1/i$}.
	\end{gather*}

	Since $\Lp{1} ( \mathscr{L}^\vdim , \rel^\codim )$ is separable, one
	can use
	\cite[\printRoman{5}.4.2,\,\printRoman{5}.5.1,\,\printRoman{4}.8.3]{MR90g:47001a}
	to infer the existence of $S \in \mathscr{D} ' ( \rel^\vdim,
	\rel^\codim )$, $f \in \Lp{\infty} ( \mathscr{L}^\vdim, \rel^\codim )$
	and a sequence $\varepsilon_j$ with $\varepsilon_j \downarrow 0$ as $j
	\to \infty$ such that
	\begin{gather*}
		S ( \theta ) = {\textstyle\int} f \bullet \theta \ud
		\mathscr{L}^\vdim \quad \text{for $\theta \in \mathscr{D} (
		\rel^\vdim, \rel^\codim )$}, \quad S_{\varepsilon_j}
		\to S \quad \text{as $j \to \infty$}.
	\end{gather*}
	Defining $R = T - S$ and noting $R_{\varepsilon_j} \to R$ as $j \to
	\infty$,
	\begin{gather*}
		\dnorm{R}{1}{a,r} \leq \Gamma \, \gamma \, r \,
		\mathscr{L}^\vdim ( \cball{a}{4r} \without A_i) \quad
		\text{for $0 < r < 1/i$}
	\end{gather*}
	and \cite[2.9.11]{MR41:1976} implies
	\begin{gather*}
		\lim_{r \pluslim{0}} r^{-1-\vdim} \dnorm{R}{1}{a,r} = 0
		\quad \text{for $\mathscr{L}^\vdim$ almost all $a \in A_i$}.
	\end{gather*}
	Moreover,
	\begin{gather*}
		\big | {\textstyle\int} ( f(x) - f (a) ) \bullet \theta (x)
		\ud \mathscr{L}^\vdim x \big | \leq \big (
		{\textstyle\int_{\oball{a}{r}}} | f(x) - f (a) | \ud
		\mathscr{L}^\vdim x \big ) \, r \, \norm{D
		\theta}{\infty}{a,r}
	\end{gather*}
	whenever $a \in A$, $0 < r < \infty$, $\theta \in \mathscr{D} (
	\rel^\vdim, \rel^\codim )$ with $\spt \theta \subset \oball{a}{r}$ and \cite[2.9.9]{MR41:1976} implies that one can take $T_a$
	defined by $T_a ( \theta ) = \int \theta (x) \bullet f(a) \ud
	\mathscr{L}^\vdim x$ for $\theta \in \mathscr{D} ( \rel^\vdim,
	\rel^\codim )$ for $\mathscr{L}^\vdim$ almost all $a \in A_i$ in the
	existence part of the conclusion.

	The uniqueness follows from \ref{miniremark:constant_distributions}.
\end{proof}
\begin{remark}
	The splitting of $T$ into $S$ and $R$ was inspired by a similar
	procedure for functions used by Calder\'on and Zygmund in
	\cite[Theorem 7]{MR0136849}.
\end{remark}
{\small \noindent
Max-Planck-Institute for Gravitational Physics (Albert-Einstein-Institute),
\newline OT Golm, Am M{\"u}hlen\-berg 1, DE-14476 Potsdam, Germany \newline
\texttt{Ulrich.Menne@aei.mpg.de}}
\addcontentsline{toc}{section}{\numberline{}References}
\bibliography{C2Rekt} \bibliographystyle{alpha}
\end{document}